\documentclass[12pt]{article}
\usepackage{amsmath,amsfonts,amssymb,amsthm}
\usepackage{xcolor}
\usepackage{enumitem}
\usepackage[toc, page]{appendix}
\setitemize{itemsep=0.1pt}
\setenumerate{itemsep=0.1pt}

\usepackage{datetime2}
 
\RequirePackage[colorlinks,citecolor=blue,urlcolor=blue,backref = page]{hyperref}

\usepackage{mathtools}

\usepackage{soul,xcolor}
\usepackage[normalem]{ulem}
\setstcolor{black}

\newtheorem{proposition}{Proposition}[section]
\newtheorem{theorem}[proposition]{Theorem}
\newtheorem{corollary}[proposition]{Corollary}
\newtheorem{lemma}[proposition]{Lemma}
\newtheorem{remark}[proposition]{Remark}

\newtheorem{example}[proposition]{Example}

\newcommand{\nc}{\newcommand}
\nc{\md}{\mathrm{d}}
\nc{\I}{{\mathbf 1}}

\nc{\bU}{\mathbb{U}}
\nc{\cU}{\mathcal{U}}
\nc{\bN}{{\mathbf N}}
\nc{\bM}{{\mathbf M}}
\nc{\cB}{{\mathcal B}}
\nc{\cC}{{\mathcal C}}
\nc{\cK}{{\mathcal K}}
\nc{\cL}{{\mathcal L}}
\nc{\R}{{\mathbb R}}
\nc{\C}{{\mathbb C}}
\nc{\M}{{\mathcal M}}
\nc{\N}{{\mathbb N}}
\nc{\cN}{{\mathcal N}}
\nc{\Z}{{\mathbb Z}}
\nc{\bF}{{\mathbf F}}
\nc{\bw}{{\mathbf w}}
\nc{\tC}{\tilde{C}}
\nc{\tc}{\tilde{c}}
\nc{\hC}{\hat{C}}
\nc{\hc}{\hat{c}}
\nc{\tphi}{\tilde{\varphi}}
\nc{\tvphi}{\tilde{\delta}}
\nc{\hvphi}{\hat{\delta}}

\nc{\tPhi}{\tilde{\Phi}}
\nc{\Psif}{\Psi^{!}}
\nc{\Ff}{F^{!}}
\nc{\tbN}{\tilde{\mathbf{N}}}
\nc{\tx}{\tilde{x}}
\nc{\ty}{\tilde{y}}
\nc{\talpha}{\tilde{\alpha}}
\nc{\tf}{\tilde{f}}
\nc{\tGamma}{\tilde{\Gamma}}
\nc{\tmu}{\tilde{\mu}}
\nc{\Ks}{K^{\ast}}
\nc{\Ls}{L^{\ast}}
\nc{\Ce}{\textnormal{m}}

\DeclareMathOperator{\diam}{diam}
\nc{\BP}{\mathbb{P}}
\nc{\BE}{\mathbb{E}}
\nc{\BQ}{\mathbb{Q}}
\nc{\BX}{\mathbb{X}}
\DeclareMathOperator{\BV}{{\mathbb Var}}
\DeclareMathOperator{\BC}{{\mathbb Cov}}

\newcommand{\smallbox}{{\mathord{\scalebox{0.5}{$\Box$}}}}

\numberwithin{equation}{section}

\begin{document}

\renewcommand{\thefootnote}{\fnsymbol{footnote}}
\author{M.A. Klatt\footnote{German Aerospace Center (DLR), Institute for AI Safety and Security, Wilhelm-Runge-Str. 10, 89081 Ulm, Germany; German Aerospace Center (DLR), Institute of Frontier Materials on Earth and in Space, Functional, Granular, and Composite Materials, 51170 Cologne, Germany; Department of Physics, Ludwig-Maximilians-Universität München, Schellingstr. 4, 80799 Munich, Germany.}, G. Last\footnote{Institute for Stochastics, Karlsruhe Institute of Technology, Karlsruhe, Germany.}, L. Lotz\footnote{German Aerospace Center (DLR), Institute of Frontier Materials on Earth and in Space, Functional, Granular, and Composite Materials, 51170 Cologne, Germany; Institute for Stochastics, Karlsruhe Institute of Technology, Karlsruhe, Germany.} \, and D. Yogeshwaran\footnote{Theoretical Statistics and Mathematics Unit, Indian Statistical Institute, Bangalore, India.}}
\title{Invariant transports of stationary random measures:\\ asymptotic variance, hyperuniformity, and examples}

\date{\today}
\maketitle

\begin{abstract}
\noindent
We consider invariant transports of stationary random measures on $\R^d$
and establish natural mixing criteria that guarantee persistence of
asymptotic variances. To check our mixing assumptions, which are based
on two-point Palm probabilities, we combine factorial moment expansion
with stopping set techniques, among others. We complement our results by
providing formulas for the Bartlett spectral measure of the
destinations. We pay special attention to the case of a vanishing
asymptotic variance, known as hyperuniformity. By constructing suitable
transports from a hyperuniform source we are able to rigorously
establish hyperuniformity for many point processes and random measures.
On the other hand, our method can also refute hyperuniformity. For
instance, we show that finitely many steps of Lloyd's algorithm or of a
random organization model preserve the asymptotic variance if we start
from a Poisson process or a point process with exponentially fast
decaying correlation. Finally, we define a hyperuniformerer that turns
any ergodic point process with finite intensity into a hyperuniform
process by randomizing each point within its cell of a fair partition.
\end{abstract}

{
\hypersetup{hidelinks}
\tableofcontents
}

\newpage


\vspace{0.2cm}

\section{Introduction}\label{sec:intro}

Let $\Phi$ be a stationary random measure on $\R^d$ with positive
intensity~$\gamma$ which is square-integrable,  that is, $\BE \Phi(B)^2<\infty$ for all
bounded Borel sets $B\subset\R^d$. Let $\lambda_d$ be the Lebesgue
(volume) measure on $\R^d$ and let $B_r$ denotes a ball of radius $r\ge
0$, centred at the origin $0$. The {\em asymptotic variance} of $\Phi$
is defined by
\begin{align}\label{e1.2}
\sigma^2_\Phi:=\lim_{r\to\infty}\lambda_d(B_r)^{-1}\BV[\Phi(B_r)],
\end{align}
provided the limit exists. In  this paper we shall study persistence
properties of this variance under stationary transports. 

Our main motivation is the {\em hyperuniform} case $\sigma^2_\Phi=0$.
Hyperuniform point processes are characterized by an anomalous
suppression of density fluctuations on large scales~\cite{TS03,
Torquato2018}. They encompass lattices and many quasicrystals as well as
exceptional disordered ergodic point processes~\cite{Gabrielli2002,
TS03, T16, Torquato2018, BH2024}. These hyperuniform point processes
have recently attracted considerably attention in physics~\cite{TS03,
Corte08, hexner17, Torquato2018, K19, KT19a, ZKH24, DNC24,DereudreFlimmel} and
increasingly also in mathematics~\cite{GL17a, GL17b, KLY20, Coste2021,
L22a, L23, HueLeble24, BH2024}. It is known in special cases that
stationary transports keep the asymptotic variance. A key example is the perturbed
lattice to be discussed below and also later in this paper; see
\cite{Gacs75} for a seminal reference.
Another example is the stable matching between the lattice $\Z^d$ and a
Poisson process of higher intensity, explored in \cite{KLY20}. It was
proved there that the matched Poisson points are hyperuniform, even
though the case of the stationarized lattice was left open. Transports
from the stationarized lattice were also studied in the recent preprint
\cite{DFHL24}, where the authors obtained (among other things) sharp
persistence results for $d\in\{1,2\}$. Further examples, where
transports of random measures and asymptotic variance have been
studied in physics, are displacement fields~\cite{G04, Gabrielli2008,
KT18, KKT20}, diffusion processes~\cite{T21, ZKH24, DNC24}, random
self-organization~\cite{Corte08, HexnerLevine15, HexnerLevine17,
hexner17, GoldsteinLebowitzSpeer24}, construction principles for
hyperuniform porous media~\cite{KT19a, KT19b}, and the formation and
structural characteristics of foams and cellular
structures~\cite{FSFB14, K19, ChiecoDurian21, Newby2025}.
Relations between transports of random measures and the asymptotic
variance have also recently been studied in mathematics as well \cite{HueLeble24, LRY2024, Erbar25, Ezoe2025, Flimmel25}.

Our first aim here is to establish mixing criteria for the persistence of asymptotic variance
in the general setting of stationary transports.
Our second aim is to use these results for the construction of new rigorous
examples of hyperuniform point processes and random measures.

For our most general result, Theorem \ref{maintheoremmixing}, 
we consider a stationary locally square-integrable random measure $\Phi$, 
called the {\em source}, and two random {\em transport kernels} $K$ and $L$, 
assuming {\em joint stationarity}. The kernels $K,L$ transport  $\Phi$ to the {\em destinations}
$K\Phi$ and $L\Phi$. 
Under a suitable mixing-type assumption (expressed in terms of Palm expectations), 
these two destinations have the
same asymptotic variance. At first glance the choice of two destinations 
may sound counterintuitive, but is in fact a key ingredient of our approach.
For example, choosing $L$ to be the average of $K$ (in a certain sense), we show in
Theorem~\ref{maintheoremmixing2} persistence of hyperuniformity from the source $\Phi$ to
the destination $K\Phi$, again under a mixing assumption.
Alternatively, in Theorem \ref{t:splitting} we assume
$K$ to be a conditionally independent invariant randomization
of a given transport kernel $L$. Then $K\Phi$ and $L\Phi$
have the same asymptotic variance, provided one of them exists
and without further mixing assumptions. 
This works even if $\Phi$ is not locally square integrable.
A finite intensity is enough.
We  wish to stress that stationary transports can
destroy hyperuniformity, even if source and transport are independent; 
see \cite{DFHL24} for an example. 
Therefore, some assumptions are required to keep
the asymptotic variance.
Most likely our mixing assumption is not optimal. But in our opinion
it is a rather mild and natural constraint. In particular we do not need to make
any moment assumptions on the transport kernel.
We use our findings to generalize earlier results in the case of
transports independent of the source and to
construct several new rigorous examples of
hyperuniform random measures, some of them come from (or are at least
motivated by) the physics literature.
Even in the non-hyperuniform case, our results on equality of asymptotic variances can be used to provide variance lower bounds which are useful and sometimes important for central limit theorems; see for example \cite{Nazarov12,BYY19,KY23}.

The {\em Bartlett spectral measure} (also called {\em diffraction measure})
is an important tool for analyzing
second order properties of stationary random measures; see e.g.\
\cite{Bremaud20} and Subsection \ref{subBartlett}. Its Lebesgue density 
(multiplied by the intensity)
is known as {\em structure factor} in physics, where it plays a fundamental 
role in scattering experiments; see e.g.\ \cite{Torquato2018}. 
Estimating the structure factor is an important task for the statistics of spatial point
processes; see e.g.\ \cite{MugRenshaw96}.
We complement all of our results with formulas for the spectral measure
of the destinations. In the setting of Theorem \ref{maintheoremmixing2} 
for instance, Theorem~\ref{maintheoremmixing2fourier}  shows how the
spectral measure of $K\Phi$ can be expressed in terms of the spectral measure
of $\Phi$ and Palm expectations of  the spatially correlated Fourier transforms 
of $K$.

Rather than going here into further details of our general results, we illustrate them with
two examples.
In the first example we consider a simple (no multiple points) stationary point process
$\Psi$ with intensity $\gamma$ along with a random partition $\{C(x):x\in\Psi\}$ of $\R^d$.
We refer to $C(x)$ as {\em cell} associated with $x\in\Psi$ and assume that
the partition is translation covariant; see Example \ref{ex:5.2} for more detail.
We do not impose any topological restrictions on the cells; in particular,
they might not be connected. 
Let us now assume that the partition is {\em fair}, that is, we have almost surely that
$\lambda_d(C(x))=\gamma^{-1}$ for all $x\in\Psi$.
Fair partitions can be constructed
for any stationary and ergodic point process (with finite intensity),
even without further randomization; see \cite{HoHolPe06, HolroydPeres05}
and also \cite[Corollary 10.10]{LastPenrose17}.
Let $Z(x)$, $x\in\Psi$,  be random vectors in $\R^d$ which are conditionally independent
given $\Psi$ and $\{C(x):x\in\Psi\}$ and whose conditional distributions are
uniform on $C(x)$. As a consequence of our general Theorem \ref{t:splitting2}  we show 
in Example \ref{ex:hyperuniformerer} that the point process
\begin{align}
\Gamma:=\sum_{x\in\Psi}\delta_{Z(x)}
\end{align}
is hyperuniform, see also Figure \ref{f:hyperuniformerer GRF} (left). Hence, redistributing the points from $\Psi$
(conditionally) independent and completely at random in their associated
cells, results in a hyperuniform process $\Gamma$. We call such an
procedure that turns an ergodic point process with finite intensity
into a hyperuniform counterpart a {\em hyperuniformerer}.
In a simulation study, we apply our hyperuniformerer to a
cloaked lattice, the Poisson point process, and an
anti-hyperuniform hyperplanes intersection process, all of which are
turned by the hyperuniformerer into (apparently) different hyperuniform
point processes. Our
example for such a hyperuniformerer has been strongly motivated by
\cite{Gabrielli2008}, where the authors also considered a fair
partition, but moved the points to the centers of mass of their cells.
This choice requires a number of technical assumptions and some
mathematical details were omitted.

\begin{figure}[t]\label{f:hyperuniformerer GRF}
  \centering
  \includegraphics[width=0.49\textwidth]{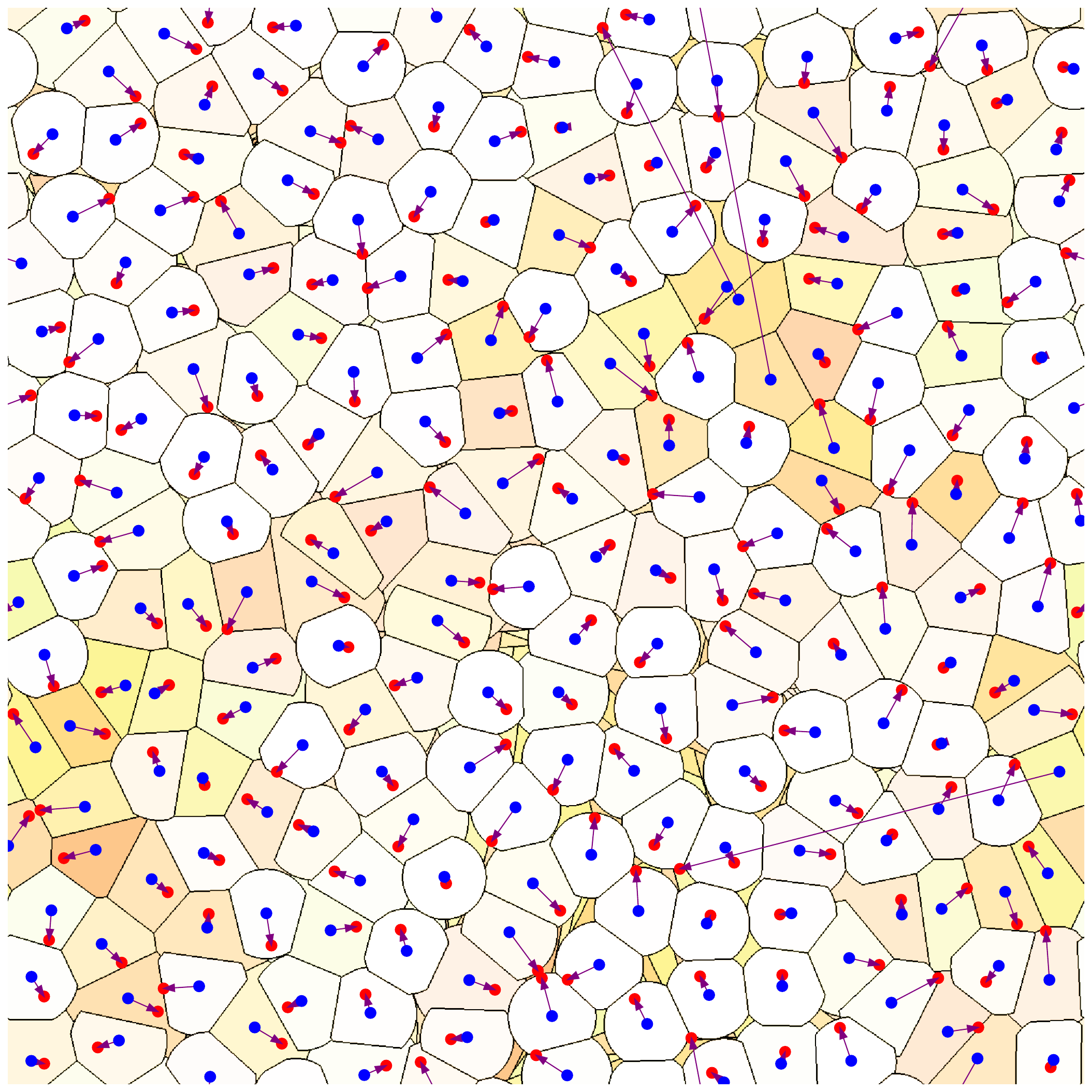}%
  \hfill%
  \includegraphics[width=0.49\textwidth]{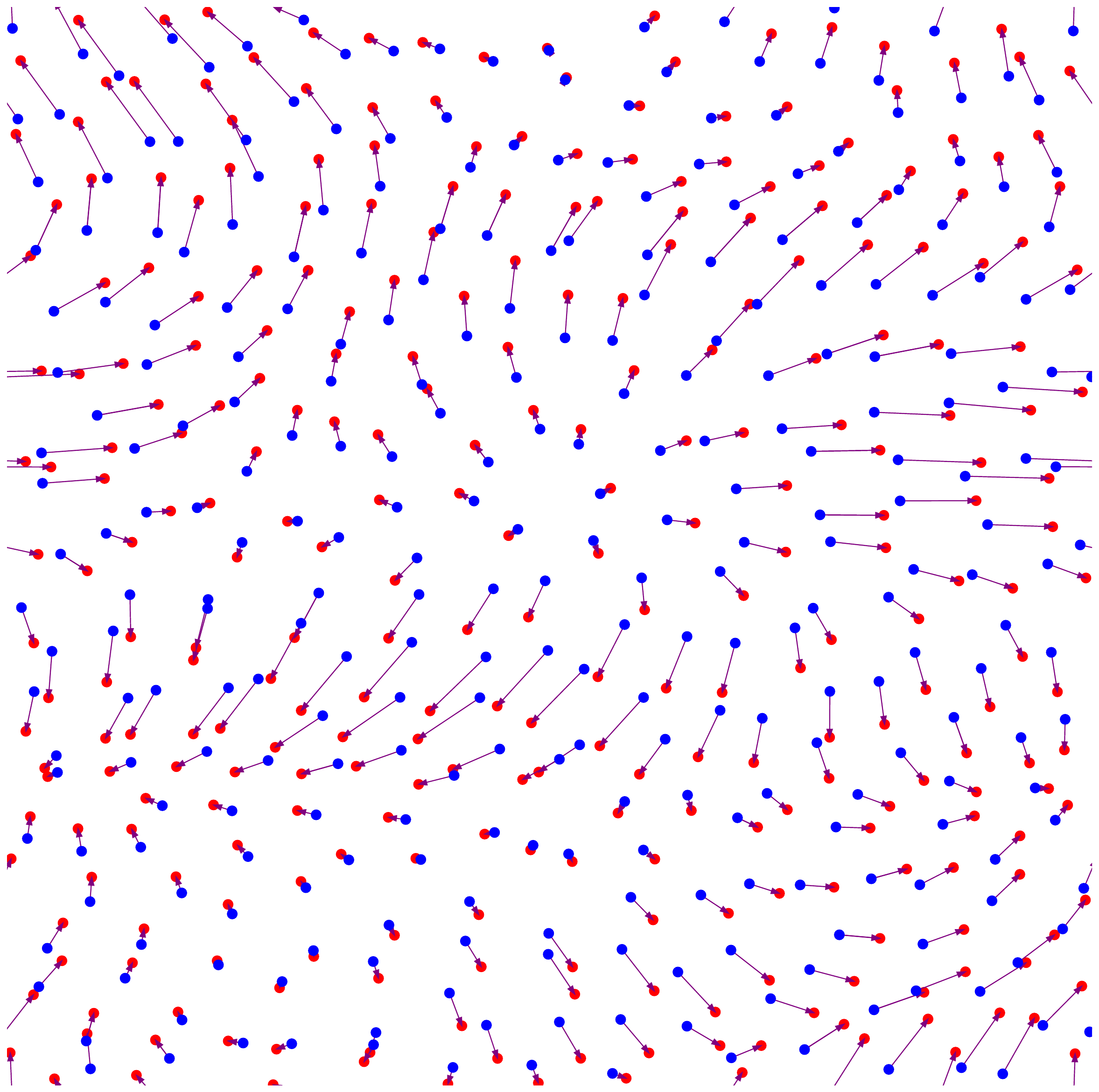}
  \caption{Two examples of invariant transports: (Left, \textit{hyperuniformerer}) We start from a
    non-hyperuniform point pattern (blue points) and construct via
    stable marriage a fair partition of space, where each cell
    has the same area. Then, we place in each cell~---~independently and
    uniformly distributed~---~a point; the resulting point process is
    hyperuniform (red points). 
    (Right, \textit{Gaussian displacements}) Each point of
    the hyperuniform source (blue points) is displaced according to a
    Gaussian random field (purple arrows) so that the destination (red
    points) is also hyperuniform; see
  Example~\ref{ex:gaussiandisplacements}.}
  \label{fig:intro}
\end{figure}

In our second example we again consider
a simple point process with finite intensity $\gamma$, this time denoted by $\Phi$.
Let $\tau\colon\R^d\to\R^d$ be an {\em invariant allocation}, that is, a (measurable) mapping which depends
on the underlying randomness in a translation covariant way; see \eqref{allocation}.
Then 
\begin{align}\label{e2.1c}
\Psi:=\sum_{x\in\Phi}\delta_{\tau(x)}
\end{align}
is a stationary point process with intensity $\gamma$. 
Let $\BP^\Phi_{0}$ be the {\em Palm probability measure} associated with $\Phi$,
describing the conditional distribution of the underlying randomness (including $\Phi$)
given that the origin $0$ is a point of $\Phi$. 
The {\em two-point  Palm probability measures} $\BP^\Phi_{0,y}$, $y\in\R^d$, admit a
similar interpretation; see Subsection \ref{subtwopoint} for more detail. 
For $y\in\R^d$ we define 
\begin{align}\label{kappaspecial}
\kappa(y):=\big\|\BP_{0,y}^\Phi((\tau(y)-y, \tau(0))\in\cdot) - \BP_0^\Phi(\tau(0)\in\cdot)^{\otimes2}\big\|,
\end{align}
where $\|\cdot\|$ denotes the total variation norm.
If $\kappa(y)\to 0$ as $\|y\|\to\infty$,
then we might expect $\Phi$ and $\Psi$  to have the same asymptotic variance.
Our Theorem \ref{maintheoremmixing2} indeed shows that the latter is the case, provided that
\begin{align}\label{eintro7}
\int \kappa(y)\,\alpha_\Phi(\md y)<\infty,
\end{align}
where $\alpha_\Phi(\cdot):=\gamma\int \mu(\cdot)\BP^\Phi_0(\md \mu)$
is the {\em reduced second moment measure} of $\Phi$. If $\Phi$ has a {\em pair correlation function}
$g$ satisfying $\int |g(x)-1|\,\md x<\infty$, then \eqref{eintro7} is equivalent to
\begin{align}\label{eintro8}
\int \kappa(y)g(y)\,\md y<\infty.
\end{align}
Hence, as $\|y\|\to\infty$,  $\kappa(y)g(y)$ should tend to $0$ sufficiently fast.
Assume now that $\tau$ is independent of $\Phi$ and any further randomness 
(if at all present). Then  the distribution of $(\tau(y), \tau(0))$ (resp.\ $\tau(0)$) under
$\BP^\Phi_{0,y}$ (resp.\ $\BP^\Phi_0$) is the stationary distribution of these random elements, 
simplifying the definition of $\kappa(y)$.
Such {\em independent (additive) displacements} were first studied in \cite{G04}.
An important special case is the {\em perturbed stationary lattice}. In this case
$\Phi=\sum_{x\in\Z^d}\delta_{x+U}$, where $U$ is uniformly distributed on the unit cube
and the random field $\{\tau(x)-x:x\in\Phi\}$ is stationary and independent of $U$.
Since $\alpha_\Phi=\sum_{y\in\Z^d}\delta_y$ condition \eqref{eintro7} boils down to
\begin{align}\label{HUlattice}
\sum_{y\in\Z^d} \kappa(y)<\infty,
\end{align}
and $\kappa(y)$ is the {\em $\beta$-mixing coefficient} between the random
variables $\tau(y)-y$ and $\tau(0)$. If \eqref{HUlattice} holds, 
then $\Psi$ is hyperuniform. The recent preprint \cite{Flimmel25} draws the same conclusion 
under assumptions on $\alpha$-mixing coefficients and an additional moment condition; 
see Remark \ref{r:Flimmel25}.
Further results on hyperuniformity of perturbed lattices can be found in \cite{KKT20,DFHL24}. 
We will treat perturbed lattices in Section \ref{s:indepdispl} in the more general setting of
independent translation fields (and kernels) applied to general random measures
and purely discrete point processes. In the case of a Gaussian translation field 
condition \eqref{HUlattice} translates into the integrability of the
covariance function of the field, see also Figure \ref{f:hyperuniformerer GRF} (right).

The paper is organized as follows.  Section \ref{sec:prelim}
summarizes some fundamental concepts for stationary random measures
and transports, used throughout the paper. Section
\ref{sec:Eq_asymp_var} contains two of our main theoretical findings.
Theorems \ref{maintheoremmixing} and \ref{maintheoremmixing2} deal
with the asymptotic variance of a transported random measure $\Phi$,
while Theorems \ref{maintheoremmixingfourier} and
\ref{maintheoremmixing2fourier} are the corresponding Fourier
versions.  In Section \ref{sec:randtransport} we assume that $\Phi$ is
purely discrete.  Theorems \ref{t:splitting} and \ref{t:splitting2}
provide significant generalizations of the hyperuniformerer, which is
discussed in Example \ref{ex:hyperuniformerer}. In Section
\ref{s:indepdispl} we investigate the special case when the transport
is given by an independent displacement field $Z$ (or displacement
kernel $K$).  Then the Palm expectations reduce to ordinary
(stationary) expectations and our mixing condition \eqref{eintro7} boils down to a
$\beta$-mixing condition on the two-dimensional marginals of
$Z$. Section \ref{s:mix_marked_pp} gives a general theorem to
  verify mixing condition \eqref{eintro7} for transport kernels based on stopping
  sets. This is done for point processes satisfying an asymptotic
  decorrelation property with the help of {\em factorial moment expansion} (FME). 
Sections  \ref{s:localalgo} and \ref{s:volumes}
  illustrate our general theorems by
  showing asymptotic equality of variances in some examples including
  random organization model and Lloyd's algorithm. 
The final Section \ref{s:hyprandset} contains an application
to hyperuniform random measure supported by random sets, inspired 
by \cite{KT19a, KT19b}.
It is possible to read sections \ref{s:indepdispl},  \ref{s:localalgo}, \ref{s:volumes}, and \ref{s:hyprandset} without the theoretical background from the other sections.
The Appendix \ref{AppendixPalm} 
provides an elaborate background on Palm calculus
along with a self-contained derivation of FME for point
processes. We give a quick derivation of total variation distance between two Gaussian random vectors in Appendix \ref{s:tvgrv}, which is used in Section \ref{s:indepdispl}.

\section{Preliminaries}\label{sec:prelim}

In this section, we recall some required notions about random measures
and point processes, and in particular, on stationary point processes
and random measures (Section \ref{substatpp}) and their Bartlett
spectral measure (Section \ref{subBartlett}). For more details on
point processes and random measures, we refer the reader to
\cite{Kallenberg17,LastPenrose17,BBK20,Bremaud20}. Finally, in Section
\ref{subTransportKernels}, we introduce transport maps and kernels, as
well as establish the relation between them. Palm theory and
higher-order correlations are introduced in Appendix
\ref{AppendixPalm}.

\subsection{Stationary point processes and random measures}\label{substatpp}

Given a metric space $\mathbb{X}$, we denote by $\bM(\mathbb{X})$
the space of all locally finite measures $\varphi$ on $\mathbb{X}$,
equipped with the smallest $\sigma$-field making the mappings $\varphi\mapsto \varphi(B)$
measurable for each Borel set $B\subset \mathbb{X}$. Of particular interest to
us are the cases $\mathbb{X}=\R^d$ and $\mathbb{X}=\R^d\times\R^d$.
A {\em random measure} on $\mathbb{X}$ is a random element $\Phi$ of 
$\bM(\mathbb{X})$ 
defined over some fixed probability space $(\Omega,\mathcal{A},\BP)$.
Note that $\Phi$ can be seen as kernel from $\Omega$ to $\R^d$.
Let $\bN(\mathbb{X})$ be the space
of all $\varphi\in \bM(\mathbb{X})$ taking values in $\N\cup\{\infty\}$.
This is a measurable subset of $\bM(\mathbb{X})$. We  
equip it with the trace $\sigma$-field. A measure $\varphi\in\bN$ is called {\em simple} if
$\varphi(\{x\})\in\{0,1\}$ for each $x\in\R^d$. In this case, we may
identify $\varphi$ with its {\em support} $\{x\in\R^d:\varphi(\{x\})>0\}$.
Let $\mathbf{N}_s$ denote the set of all such simple measures.
A {\em point process} on $\mathbb{X}$ is a random element of 
$\bN(\mathbb{X})$. 
It can be represented as
\begin{align}\label{e1.5}
\Phi=\sum^{\Phi(\mathbb{X})}_{n=1}\delta_{X_n},
\end{align}
where $X_1,X_2,\ldots$ are random elements of $\mathbb{X}$, and $\delta_x$ is the
Dirac measure at $x$. If  $X_m\ne X_n$ for $m,n\in\N$ with $m< n\le\Phi(\mathbb{X})$,
then $\Phi$ is said to be {\em simple} or equivalently a point process $\Phi$ with $\BP(\Phi\in \mathbf{N}_s)=1$ is called {\em simple}.

Now onwards, we consider a random measure on $\R^d$, equipped with the Euclidean metric
and the Borel $\sigma$-field $\cB^d$.
A  random measure $\Phi$ on $\R^d$ 
is {\em stationary} if $\theta_x\Phi\overset{d}{=}\Phi$ for each $x\in\R^d$,
where $\theta_x\varphi:=\varphi(\cdot+x)$ for $\varphi\in\mathbf{M}$. 
In this case, we have
$\BE \Phi(B)=\gamma\lambda_d(B)$, $B\in\mathcal{B}$,
where $\gamma:=\BE\Phi([0,1]^d)$ is the {\em intensity} of $\Phi$. 
We then have the {\em Campbell formula}
\begin{align}\label{eCampbell}
\BE \int f(x)\,\Phi(dx) = \gamma \int f(x)\,\md x
\end{align} 
for each measurable $f\colon\R^d\to[0,\infty]$.

Assume that $\Phi$ is a stationary random measure with finite intensity $\gamma$.
The {\em reduced second moment measure} of $\Phi$
is the measure $\alpha_\Phi$ on $\R^d$, defined by
\begin{align}\label{ersecm}
\alpha_\Phi(B):=\BE \int \I\{x\in[0,1]^d,y-x\in B\}\,\Phi^2(\md (x,y)),\quad B\in\cB^d.
\end{align}
Then, by the refined Campbell theorem \eqref{erefinedC} and \eqref{alphaPalm}, we have  
\begin{align}\label{e2.55}
\BE \int f(x,y)\,\Phi^2(\md (x,y))
&=\iint f(x,x+y)\,\,\alpha_\Phi(\md y)\,\md x
\end{align}
for each measurable $f\colon\R^d\times\R^d\to[0,\infty]$.
Moreover, if $\Phi$ is {\em locally square-integrable}, that is $\BE \Phi(B)^2<\infty$ for all bounded Borel sets $B\subset\R^d$, it  follows that $\alpha_\Phi$ is locally finite, as in \cite[Chapter 8]{LastPenrose17}.

Assume that $\Phi$ is a stationary locally square-integrable random measure.
Let $W \in \cK_0$, the space of convex bounded sets containing $0$ in their interior. 
Then $\Phi$ is said to be {\em hyperuniform} with respect to (w.r.t.) $W$ if
\begin{align}\label{e:hyperuniform}
\lim_{r\to\infty}\lambda_d(rW)^{-1}\BV[\Phi(rW)]=0.
\end{align}
In general, this property depends on $W$. If $W$ is the unit ball, then
we simply call $\Phi$ {\em hyperuniform}.
To study this and other second order properties of $\Phi$, it is
convenient to work with the {\em covariance measure} $\beta_\Phi$  of
$\Phi$. This is the signed measure 
\begin{align}\label{betaPhi}
\beta_\Phi:=\alpha_\Phi-\gamma^2\lambda_d,
\end{align}
well-defined and finite on bounded Borel sets; see also Subsection \ref{subBartlett}.
Let $f,g\colon\R^d\to\R$ be bounded measurable functions
with bounded support. Then,
it follows directly from \eqref{e2.55} and \eqref{eCampbell}
that
\begin{align}\label{e:covariance}
\BC[\Phi(f),\Phi(g)]=\int f\star g(y)\,\beta_\Phi(\md y),
\end{align}
where
$$
(f\star g)(y):=\int f(x)g(x-y)\,\md x,\quad y\in\R^d,
$$
is the {\em tilted convolution} of $f$ and $g$.
In particular, we have for any bounded $B\in\cB^d$ that
\begin{align}
\BV[\Phi(B)]=\int \lambda_d(B\cap (B+y))\,\beta_\Phi(\md y).
\end{align}
If the covariance measure of $\Phi$ has finite total variation, that is
\begin{align}\label{finitetotalv}
|\beta_\Phi|(\R^d)<\infty,
\end{align}
then it follows from dominated convergence that
the asymptotic variance $\eqref{e1.2}$ of $\Phi$ exists and
is given by
\begin{align}\label{e:asvariance}
\lim_{r\to\infty}\lambda_d(rW)^{-1}\BV[\Phi(rW)]=\beta_\Phi(\R^d),
\end{align}
where $W \in \cK_0$.
In particular, $\Phi$ is hyperuniform w.r.t.\ $W$ iff 
\begin{align}\label{HU}
\beta_\Phi(\R^d)=0.
\end{align}
This condition does not depend on $W$.
If $\Phi$ is a point process, then it is sometimes more convenient to
work with the {\em reduced second factorial moment measure} $\alpha^!_\Phi$ 
of $\Phi$, defined by 
\begin{align}\label{e2.22}
\alpha^!_\Phi(B):=\BE \int \I\{x\in[0,1]^d,y-x\in B\}\,\Phi^{(2)}(\md (x,y)),\quad B\in\cB^d;
\end{align}
see \cite[Chapter 8]{LastPenrose17}. Here, $\Phi^{(2)}$ is a point process on $\R^d\times\R^d$, defined
by
\begin{align*}
\Phi^{(2)}:=\sum_{m\ne n}\I\{(X_m,X_n)\in\cdot\},
\end{align*}
where $\Phi$ is given by \eqref{e1.5}.  Instead of  \eqref{e2.55}, 
we then have
\begin{align}\label{e2.56}
\BE \int f(x,y)\,\Phi^{(2)}(\md (x,y))=\iint f(x,x+y)\,\,\alpha^!_\Phi(\md y)\,\md x.
\end{align}
It is easy to see that
\begin{align}
\alpha^!_\Phi=\alpha_\Phi - \gamma\delta_0.
\end{align}
If $\alpha^!_\Phi$ has a Lebesgue density $\rho_2$, then 
\begin{align}\label{e:2098}
\beta_\Phi=(\rho_2-\gamma^2)\cdot\lambda_d+\gamma\delta_0.
\end{align}
Then, \eqref{finitetotalv} means
\begin{align}\label{e:totalint}
\int \big|\rho_2(x)-\gamma^2\big|\,\md x<\infty.
\end{align}
If the latter condition holds, then hyperuniformity is equivalent to
\begin{align}
\int \big(\rho_2(x)-\gamma^2\big)\,\md x=-\gamma.
\end{align}
The function $\gamma^{-2}\rho_2$ is known as the {\em pair correlation function} of $\Phi$.
If $\Phi$ is the stationary lattice, then $\beta_\Phi=\sum_{k\in\Z^d}\delta_k - \lambda_d$, and if $\Phi$ is a stationary Poisson process with finite intensity $\gamma$, then $\beta_\Phi=\gamma\delta_0$.

\subsection{The Bartlett spectral measure}\label{subBartlett}

In this paper, we understand a {\em signed measure} on $\R^d$ to be a $\sigma$-additive
function $\nu$ on the bounded Borel sets with $\nu(\emptyset)=0$. Then the restriction
of $\nu$ to a bounded set is the difference of two finite measures. If $B\in\cB^d$ is not
bounded, then $\nu(B)$ might not be defined. However, by a straightforward extension
procedure, the total variation measure $|\nu|$ of $\nu$ is well-defined and locally finite.
A signed measure $\nu$ is called {\em positive semidefinite} if
\begin{align*}
\int (f\star f)(y)\,\nu(\md y)\ge 0
\end{align*}
for all bounded measurable $f\colon\R^d\to\R$ with bounded support.
In this case, there exists a locally finite (non-negative) measure $\hat\nu$ on $\R^d$,
the {\em Fourier transform} of $\nu$, satisfying
\begin{align}\label{e:spectral}
\int f\star g(x) \,\nu(\md x)=\frac{1}{(2\pi)^d}\int \hat{f}(k)\overline{\hat{g}(k)}\,\hat{\nu}(\md k),
\end{align}
for all bounded measurable $f,g\colon\R^d\to\R$ with bounded support;
see e.g.\ \cite{BergForst75}.
Here, $\hat{f}$ denotes the {\em Fourier transform} of $f\in L^1(\lambda_d)$, defined by
\begin{align*}
\hat{f}(k):=\int f(x)e^{-i\langle k,x\rangle}\,\md x,\quad k\in\R^d.
\end{align*}
If $\nu$ has a finite total variation, then the Fourier transform of $\nu$ is absolutely
continuous w.r.t.\ $\lambda_d$. By \cite[Proposition 4.14]{BergForst75}, we have
in fact that
$\hat\nu(\md k)=\hat{\nu}(k)\md k$, where (with a common abuse of notation)
\begin{align*}
\hat{\nu}(k):=\int e^{-i\langle k,x\rangle}\,\nu(\md x),\quad k\in\R^d.
\end{align*}
The function $\hat{\nu}$ is continuous and bounded. One can actually also use this as the definition for the Fourier transform of any signed measure $\nu$ with finite total variation.

Let us now fix a stationary locally square-integrable random measure $\Phi$ on $\R^d$.
As before, we denote by $\gamma$ the intensity and by $\beta_\Phi$ the covariance measure
of $\Phi$. It follows from \eqref{e:covariance} that
the measure $\beta_\Phi$ is positive semi-definite.
Its Fourier transform $\hat{\beta}_\Phi$
is known as the {\em Bartlett spectral measure} of $\Phi$; see \cite{Bremaud20}.
Let $f,g\colon\R^d\to\R$ be measurable bounded functions with bounded support.
Combining \eqref{e:covariance} and \eqref{e:spectral} yields
\begin{align}
\BC[\Phi(f),\Phi(g)]=\frac{1}{(2\pi)^d}\int \hat{f}(k)\overline{\hat g(k)}\,\hat{\beta}_\Phi(\md k),
\end{align}
and in particular,
\begin{align}\label{e:spectralvariance}
\BV[\Phi(f)]=\frac{1}{(2\pi)^d}\int |\hat{f}(k)|^2\,\hat{\beta}_\Phi(\md k).
\end{align}
By \cite[Theorem 3.6]{BH2024}, $\Phi$ is hyperuniform w.r.t.\ a Fourier smooth (see \eqref{Fsmooth}) $W \in \cK_0$ iff
\begin{align}\label{spectralasymptoticvariance}
\lim_{\varepsilon\to 0}\varepsilon^{-d}\hat{\beta}_\Phi(B_\varepsilon)=0,
\end{align}
If $\Phi$ is the stationary lattice, then $\hat\beta_\Phi=\sum_{k\in\Z^d\setminus\{0\}}\delta_k$, and 
if $\Phi$ is a stationary Poisson process with finite intensity $\gamma$, then
$\hat\beta_\Phi=\gamma\lambda_d$;
see e.g. \cite{Torquato2018,BH2024,Coste2021}.
If $\beta_\Phi$ has finite total variation, then we write the density
in the form 
\begin{align}\label{sfactor}
\hat{\beta}_\Phi(\md k)=\gamma S_\Phi(k)\,\md k,
\end{align}
where $S_\Phi$ is continuous.
In the physics literature, the function $S_\Phi\colon\R^d\to\R^d$ is known
as {\em structure factor} of $\Phi$. Using \eqref{e:asvariance}, one can then see for any  $W \in \cK_0$,
\begin{equation}\label{e:asvariancefourier}
    \lim_{r\to\infty}\lambda_d(rW)^{-1}\BV[\Phi(rW)] = \beta_\Phi(\R^d) = \gamma S_\Phi(0).
\end{equation}
Thus, hyperuniformity of $\Phi$ is equivalent to
\begin{align}\label{e:S0}
S_\Phi(0)=0.
\end{align}
If $\Phi$ is a point process and $\alpha^!_\Phi$ has a density
$\rho_2$, then \eqref{e:2098} shows that 
\begin{align}\label{sfactor2}
S_\Phi=1+\gamma\hat{h}_2,
\end{align}
where $h_2\colon\R^d\to \R^d$ is the {\em total pair correlation function} given
by $h_2:=\gamma^{-2}\rho_2-1$.

\subsection{Transports and transport kernels}
\label{subTransportKernels}

A {\em  random transport} is a measurable mapping
$T\colon \Omega \to \bM(\R^d\times \R^d)$ such that $T(\omega,\cdot\times\R^d)$
is locally finite for all $\omega$. Hence $T$ is a random measure on $\R^d\times\R^d$ 
while $T(\cdot\times\R^d)$ is a random measure on $\R^d$.
We shall often drop the argument $\omega$ from $T$ but use it for clarity when needed. 
A random transport is {\em  stationary} if it is distributionally invariant
under diagonal shifts, that is,
$\theta_xT\overset{d}{=}T$ for each $x\in\R^d$, where, 
this time, the shift operator $\theta_x\colon \bM(\R^d\times \R^d)\to \bM(\R^d\times \R^d)$
is defined by
\begin{align}\label{shiftdiag}
\theta_x\varphi(B\times C):=\varphi((B+x)\times (C+x)),\quad \,x\in \R^d, \,B,C\in\cB^d.
\end{align}
The fact that we are using $\theta_x$ to denote the shift on 
$\bM(\R^d)$, on $\bM(\R^d\times \R^d)$, and on $\Omega$ (as in 
Subsection \ref{subPalm}) should (hopefully) not cause any confusion. The meaning
will always be clear from the context.

Unless stated otherwise, we will work  here and later in the setting of
Subsection \ref{subPalm}, which describes the Palm calculus.  
A random transport on $\R^d$ is said to be {\em invariant} if
\begin{align}\label{adapttransport}
T(\omega,(B+x)\times (C+x))=T(\theta_x\omega,B\times C),\quad \omega\in \Omega,\,x\in \R^d, \,B,C\in\cB^d.
\end{align}
In this case, $T(\cdot\times\R^d)$ is an invariant random measure in the sense of \eqref{adapt}.
It then follows from \eqref{Pstat} that $T$ (and of course also $T(\cdot\times\R^d)$) is stationary.

\begin{remark}\rm The terminology {\em invariant} random measure (or random transport) always refers
to a flow on the underlying sample space. This flow is either abstract or given explicitly on a canonical
space, like $\bM(\R^d\times \R^d)$ for instance. The underlying probability measure $\BP$ is
then always assumed to be stationary in the sense of \eqref{Pstat}.
This implies  (distributionally) stationarity of invariant random measures or other, suitably flow invariant
random objects.
\end{remark}

A kernel $K$ from $\Omega\times\R^d$ to $\R^d$ is called {\em invariant} if
\begin{align}\label{adaptkernel}
K(\omega,x,B+x)=K(\theta_x\omega,0,B),\quad \omega\in \Omega,\,x\in \R^d, \,B\in\cB^d.
\end{align}
It is called a {\em probability kernel} if $K(\omega,x,\R^d) = 1$ for all $\omega \in \Omega, x \in \R^d$. Similar to dropping $\omega$, we also refer 
to $T$ (or $K$) as a (random) transport from $\R^d$ to $\R^d$.
Given $x\in\R^d$ we mean by $K(x)\equiv K(x,\cdot)$ the random probability measure
$\omega\mapsto K(\omega,x,\cdot)$. For  convenience we often
write $K_x$ instead of $K(x)$.

\begin{proposition}\label{propkernel} Assume that $T$ is an invariant random transport 
such that $\Phi:=T(\cdot\times\R^d)$ has a finite intensity.
Then there exists an invariant probability kernel $K$ from $\Omega\times\R^d$ to $\R^d$ such that
\begin{align}\label{transportkernel}
T(\omega,\cdot)=\int \I\{(x,y)\in\cdot\}\, K(\omega,x,\md y)\,\Phi(\omega,\md x),\quad \BP\text{-a.e. $\omega\in\Omega$}.
\end{align}
The kernel $K$ can be chosen so that $K(\cdot,\cdot,B)$ is $\sigma(T)\otimes\cB^d$-measurable for each $B\in\mathcal{B}$.
\end{proposition}
\begin{proof} Let $A\in\mathcal{A}\otimes\mathcal{B}^d$ 
and consider the measure $M_A$ on $\R^d$ defined by
\begin{align*}
M_A:=\int \I\{x\in \cdot,(\theta_x\omega,y-x)\in A\}\,T(\omega,\md (x,y))\,\BP(\md \omega).
\end{align*}
Let $B\in\mathcal{B}^d$ and $z\in\R^d$. Then
\begin{align*}
M_{A}(B-z)&=\BE \int \I\{x\in B-z,(\theta_x,y-x)\in A\}\,T(\theta_0,\md (x,y))\\
&=\BE \int \I\{x+z\in B,(\theta_{x+z},y-x)\in A\}\,T(\theta_z,\md (x,y))\\
&=\BE \int \I\{x\in B,(\theta_{x},y-x)\in A\}\,T(\theta_0,\md (x,y))=M_A(B).
\end{align*}
Since 
\begin{align*}
M_A(B)\le \int \I\{x\in B\}\,T(\omega,\md (x,y))\,\BP(\md \omega)=\BE \Phi(B),
\end{align*}
and $\Phi$ has a finite intensity,
the measure $M_A$ is locally finite. Therefore
\begin{align}\label{e145}
M_A(B)=M_A([0,1]^d)\lambda_d(B)=\BQ_0(A)\lambda_d(B),
\end{align}
where the (finite) measure $\BQ_0$ on $\Omega \times \R^d$ is given by
\begin{align*}
\BQ_0:=\BE \int \I\{x\in [0,1]^d,(\theta_x,y-x)\in \cdot\}\,T(\md (x,y)).
\end{align*}
It follows from \eqref{e145} and basic principles of measure theory that
\begin{align*}
\BE \int \I\{(x,\theta_x,y-x)\in \cdot\}\,T(\md (x,y))=\int\I\{(x,\omega,y)\in\cdot\}\,\BQ_0(\md (\omega,y))\,\md x.
\end{align*}
Since we have assumed $(\Omega,\mathcal{A})$ to be Borel, 
there exists a probability kernel $K_0$ from $\Omega$ to $\R^d$ satisfying 
\begin{align*}
\BQ_0(\md (\omega,y))=K_0(\omega,\md y)\,\BQ_0(\md \omega\times\R^d);
\end{align*}
see e.g.\ \cite[Theorem A.14]{LastPenrose17}.  Note that, by the definition of the Palm probability measure \eqref{Palm}, $\BQ_0(\cdot\times\R^d)=\gamma\BP^\Phi_0$, where $\gamma$ is the intensity of $\Phi$.
Therefore, we obtain 
\begin{align}\label{e147}
\BE \int \I\{(x,\theta_x,y-x)\in \cdot\}\,T(\md (x,y))=\gamma \iiint\I\{(x,\omega,y)\in\cdot\}\,K_0(\omega,\md y)\,\BP^\Phi_0(\md \omega)\,\md x,
\end{align}
which generalizes the refined Campbell theorem \eqref{erefinedC}.
Applying the refined Campbell theorem \eqref{erefinedC} to the right-hand side of
\eqref{e147} gives
\begin{align}\label{e149}
\BE \int \I\{(x,\theta_x,y-x)\in \cdot\}\,T(\md (x,y))=\BE \iint\I\{(x,\theta_x,y-x)\in\cdot\}\,K(\theta_0,x,\md y)\,\Phi(\md x),
\end{align}
where the probability kernel from $\Omega\times\R^d$ to $\R^d$ is given by
\begin{align}\label{e152}
K(\omega,x,\cdot):=K_0(\theta_x\omega,\cdot-x),\quad (\omega,x)\in\Omega\times\R^d.
\end{align}
Hence, we obtain from \eqref{e149}
\begin{align*}
\BE \int \I\{(x,\theta_0,y)\in \cdot\}\,T(\md (x,y))=\BE \iint\I\{(x,\theta_0,y)\in\cdot\}\,K(\theta_0,x,\md y)\,\Phi(\md x),
\end{align*}
which shows \eqref{transportkernel}.

To prove the measurability assertion, we consider the space $\Omega':=\bM(\R^d\times\R^d)$ equipped
with the natural (diagonal) shift and the probability measure $\BP'$, given as the distribution of $T$.
Then we can construct an invariant probability kernel $K'$ as before and (re)define
$K(\omega,x,\cdot):=K'(T(\omega),x,\cdot)$.
\end{proof}

The kernel $K$ in \eqref{transportkernel} is called (Markovian) {\em 
transport kernel} in
\cite{LaTho09} and elsewhere. The random probability measure 
$K(x,\cdot)$ describes
how a unit mass at $x\in\R^d$ is displaced in space. It is often more
convenient to work with the kernel $\Ks$ from $\Omega\times\R^d$ to $\R^d$ defined
by
\begin{align}\label{transportkernel2}
\Ks(\omega,x,B):=K(\omega,x,B+x),\quad 
(\omega,x,B)\in\Omega\times\R^d\times\cB^d.
\end{align}
Then $\Ks(x,\cdot)$ describes the displacement relative to $x$. Note that $\Ks$ and $K$ are functionals of each other i.e., we can define either of them in terms of the other.
Instead of \eqref{adaptkernel}, we then have
\begin{align}\label{adaptkernel1}
\Ks(\omega,x,B)=K(\theta_x\omega,0,B),\quad \omega\in \Omega,\,x\in 
\R^d, \,B\in\cB^d,
\end{align}
that is, $\Ks$ is invariant under joint shifts of the first two arguments.

\begin{proposition}\label{p:propinvPalm} Let the assumption of Proposition \ref{propkernel} be satisfied and let  $\gamma$
be the intensity of $\Phi:=T(\cdot\times\R^d)$. Then $\Psi :=T(\R^d\times\cdot)$ is $\BP$-almost surely
locally finite and has intensity $\gamma$. Furthermore,
\begin{align}\label{invPalm}
\BE^{\Phi}_0\int \I\{\theta_x\in\cdot\}\,K_0(\md x)=\BP^{\Psi}_0,
\end{align}
where $K$ is an invariant probability kernel satisfying \eqref{transportkernel}.
\end{proposition}
\begin{proof} Let $B\in\mathcal{B}^d$. Then we obtain from 
\eqref{transportkernel} and the refined Campbell theorem
\begin{align*}
\BE T(\R^d\times B)&=\BE\int K(\theta_x,0,B-x)\,\Phi(\md x)\\
&=\gamma\,\BE^\Phi_0\int  K(\theta_0,0,B-x)\,\md x\\
&=\gamma\,\BE^\Phi_0\iint \I\{y+x\in B\} K(\theta_0,0,\md y)\,\md x=\gamma\lambda_d(B).
\end{align*}
This proves the first two assertions.
Equation \eqref{invPalm} follows from
\cite[Theorem 4.1]{LaTho09} or a direct calculation.
\end{proof}

\section{Equality of asymptotic variance} \label{sec:Eq_asymp_var}

In this section, we formulate our first general results on equality of asymptotic variances. First, we prove equality of asymptotic variances under two (random) invariant probability kernels, $K$ and $L$, in Theorem \ref{maintheoremmixing}. Later, we specialize to the case when $L$ is the mean of $K$ in Theorem \ref{maintheoremmixing2}. In parallel, we also present analogues of these results in Fourier space in Theorems \ref{maintheoremmixingfourier} and \ref{maintheoremmixing2fourier} respectively.

\subsection{Comparing the destinations of two transports}
If $K$ is a kernel from $\Omega\times\R^d$ to $\R^d$ and
$\Phi$, which we will call source, is a kernel from $\Omega$ to $\R^d$, then we write
$K\Phi$, which we will call destination, for the kernel from  $\Omega$ to $\R^d$ defined by
\begin{align*}
K\Phi (\omega,\cdot):=\int K(\omega,x,\cdot)\,\Phi(\omega,\md x),\quad \omega\in\Omega.
\end{align*}
We will be interested in the case where $\Phi$ and $K$ are invariant.
Then the destination $K\Phi$ is invariant in the sense of
\eqref{adapt}. If, in addition, $K$ is a probability kernel and
the source $\Phi$ is locally finite with a finite intensity, then
Proposition \ref{p:propinvPalm} justifies to consider the destination 
$K\Phi$ as a random measure. Recall that we also denote
the random measure $K(y,\cdot)$, for $y\in\R^d$, by $K(y)\equiv K_y$.

The {\em total variation norm} of a finite signed measure $\nu$ on  a
measurable space is defined by
\begin{align*}
\|\nu\|:=\sup \bigg|\int f\,\md \nu\bigg|,
\end{align*}
where the supremum extends over all measurable functions
with values in $[-1,1]$. If $\nu=\nu_1-\nu_2$ is the difference of two probability measures $\nu_1$ and $\nu_2$, then
\begin{equation}\label{e:tv definition relation}
    \|\nu\| = 2 \sup \bigg| \int f \,\md \nu_1 - \int f \,\md \nu_2 \bigg|,
\end{equation}
where the supremum extends over all measurable functions
with values in $[0,1]$. We denote the product measure of $\nu_1,\nu_2$ as $\nu_1 \otimes \nu_2$.

\begin{theorem}\label{maintheoremmixing} Let  $\Phi$ be a locally square-integrable invariant random measure,
and let $K,L$ be invariant probability kernels from $\Omega\times\R^d$ to $\R^d$.
Let $W \in \cK_0$. 
Define the function $\kappa\colon\R^d\to [0,2]$ by 
\begin{align}\label{e:kappa}        
\kappa(y):=\big\|\BE^\Phi_{0,y}[K_y\otimes K_0-L_y\otimes L_0]\big\|,\quad y\in\R^d.
\end{align}
Assume that
 \begin{equation}\label{maintheoremmixingcondition}
 \int \kappa(y)\, \alpha_\Phi(\md y) < \infty.
 \end{equation}
Then 
\begin{align}\label{e:equalasvariance}
\lim_{r\to\infty}\lambda_d(rW)^{-1}\BV[K\Phi(rW)]=\lim_{r\to\infty}\lambda_d(rW)^{-1}\BV[L\Phi(rW)]
\end{align}
if one of the limits exists.
In particular, $K\Phi$ is hyperuniform with respect to $W$ iff $L\Phi$ has this property.
  \end{theorem}
The total variation in \eqref{e:kappa} is finite and bounded by $2$ because $K$ and $L$ are probability kernels. The
    proof of Theorem \ref{maintheoremmixing}, and also later proofs,
    shall rely upon analysis of the signed random measure
\begin{align}\label{etafirst}
\eta:=\alpha_{K\Phi}-\alpha_{L\Phi},
\end{align}
where $K,L$ are as in Theorem \ref{maintheoremmixing}, and which is defined for all sets where the RHS is not $\infty-\infty$. If the destinations $K\Phi$ and $L\Phi$
are square-integrable, then $\eta$ is a (locally finite) signed measure.
In view of \eqref{betaPhi} and \eqref{e:asvariance}, one can expect that the destinations $K\Phi$ and
$L\Phi$ have the same asymptotic variance as soon as $\eta$ has total mass $0$. 
We will establish this fact in Lemma \ref{etafinitetv}.

The first step towards proving Lemma \ref{etafinitetv} is the following lemma, which expresses the expectations on the RHS of \eqref{etafirst} in terms of Palm expectations of $\Phi$. For two measures $\nu$ and $\nu'$ on $\R^d$, we denote the {\em tilted convolution} of $\nu$ and $\nu'$ by
\begin{align*}
\nu\star\nu':=\int \I\{x-y\in\cdot\}\,\nu(\md x)\,\nu'(\md y).
\end{align*}

\begin{lemma}\label{l:345} Suppose that $\Phi$ is an invariant random measure
with positive and finite intensity $\gamma$. Let
$K$ be an invariant probability kernel from $\Omega\times\R^d$ to $\R^d$. Then
\begin{align}\label{e:200}
\alpha_{K\Phi}(B) = \gamma\BE ^\Phi_0 \int (K_y\star K_0)(B)\,\Phi(\md y),\quad B\in\cB^d.
\end{align}
If $\Phi$ is locally square-integrable, then
\begin{align}\label{e:201}
\alpha_{K\Phi}(B) = \int \BE^\Phi_{0,y}[(K_y\star K_0)(B)]\,\alpha_\Phi(\md y), \quad B\in\cB^d.
\end{align}
\end{lemma}
\begin{proof}
Recall from Proposition \ref{p:propinvPalm} that $K\Phi$ has the same intensity as $\Phi$.
Let $B\in\cB^d$. By definition of $\alpha_{K\Phi}$ in \eqref{ersecm},
\begin{align*}
\alpha_{K\Phi}(B)=\BE\int\I\{x\in[0,1]^d\}K\Phi(B+x)\, K\Phi(\md x).
\end{align*}
Further, by definition of $K\Phi$ and invariance, we obtain that
\begin{align}\label{e290}\notag
&\alpha_{K\Phi}(B)
=\BE\iiint \I\{x\in[0,1]^d\}K(y,B+x)K(z,\md x)\,\Phi(\md y)\,\Phi(\md z)\\
&=\BE\iiint \I\{x+z\in[0,1]^d\}K(\theta_y,0,B+x+z-y)K(\theta_z,0,\md x)\,\Phi(\md y)\,\Phi(\md z).
\end{align}
Therefore, we obtain from the refined Campbell theorem \eqref{erefinedC} that
\begin{align*}
&\alpha_{K\Phi}(B)
=\BE\iiint \I\{x+z\in[0,1]^d\}K(\theta_{y+z},0,B+x-y)\, K(\theta_z,0,\md x)\,\Phi(\theta_z,\md y)\,\Phi(\md z)\\
&=\gamma\BE^\Phi_0\iiint \I\{x+z\in[0,1]^d\}K(\theta_{y},0,B+x-y)\, K(0,\md x)\,\Phi(\md y)\,\md z\\
&=\gamma\BE^\Phi_0\iint K(y,B+x)\, K(0,\md x)\,\Phi(\md y).
\end{align*}
This proves the first assertion.

Now, assume that $\Phi$ is locally square-integrable. Then we obtain from \eqref{e290} and \eqref{2pointPalm},
\begin{align*}
\alpha_{K\Phi}(B)
&=\iint \BE^\Phi_{0,y}\bigg[\int \I\{x+z\in[0,1]^d\}K(\theta_y,0,B+x-y)\, K(0,\md x)\bigg]\,\alpha_\Phi(\md y)\,\md z\\
&=\int \BE^\Phi_{0,y}\bigg[\int K(\theta_y,0,B+x-y)\, K(0,\md x)\bigg]\,\alpha_\Phi(\md y)\\
&=\int \BE^\Phi_{0,y}\bigg[\int K(y,B+x)\, K(0,\md x)\bigg]\,\alpha_\Phi(\md y).
\end{align*}
This proves the second assertion.
\end{proof}

\begin{lemma}\label{etafinitetv}
  Let $\Phi$ be a locally square-integrable invariant random measure,
  and let $K,L$ be invariant probability kernels from
  $\Omega\times\R^d$ to $\R^d$. Assume that $K\Phi$ or $L\Phi$ is
  locally square-integrable and that condition
  \eqref{maintheoremmixingcondition} is fulfilled.
Then $K\Phi$ and $L\Phi$ are both locally square-integrable, 
and $\eta$, 
defined by \eqref{etafirst},
is a signed measure with total mass $0$ and total variation of at most $\int \kappa(y)\,\alpha_\Phi(\md y)<\infty$.
\begin{proof}
Because $K\Phi$ or $L\Phi$ is locally square-integrable, the number 
$\eta(B)$ is well-defined  for bounded $B\in\cB^d$. (Potentially it might equal $-\infty$ or $\infty$.)
By Lemma \ref{l:345}, we have
\begin{align}\label{e:331}
\eta(B) 
 = \int \BE^\Phi_{0,y}[K_y\star K_0 - L_y\star L_0](B)\, \alpha_\Phi(\md y).
\end{align}
Even for unbounded $B$, the modulus of the integrand is bounded by $\kappa$.
Thus,
because $\kappa$ satisfies the integrability condition
\eqref{maintheoremmixingcondition}, the signed measure $\eta$ 
is defined on all Borel sets, and has total variation of at most
$\int \kappa(y)\,\alpha_\Phi(\md y)<\infty$. This also implies that
$K\Phi$ and $L\Phi$ are both locally square-integrable because
$\alpha_{K\Phi}(B)$ or $\alpha_{L\Phi}(B)$, and their
difference $\eta(B)$, are finite. Hence, both are finite. Finally, from
\eqref{e:331}, we can also derive $\eta(\R^d)=0$, as for any $y\in\R^d$,
$\BE^\Phi_{0,y}[K_y\star K_0 - L_y\star L_0]$
is the difference of two probability measures and thus has
total mass $0$.
\end{proof}
\end{lemma}
\begin{proof}[Proof of Theorem \ref{maintheoremmixing}]
  If $K\Phi$ and $L\Phi$ are not locally square-integrable, the
    limits in \eqref{e:equalasvariance} are not well-defined. Now
    assume the contrary. From Lemma \ref{etafinitetv}, we know that
    condition \eqref{maintheoremmixingcondition} implies
    that both $K\Phi$ and $L\Phi$ are locally square-integrable and
    that $\eta$ is a signed measure with finite total variation and
    total mass $0$. 
    
Take a bounded $B\in\cB^d$. Then 
\begin{align}
            \BV[K\Phi(B)] &= \BE[K\Phi(B)^2] - (\BE K\Phi(B))^2\nonumber\\
            &= \BE\bigg[\int_BK\Phi(B)\,K\Phi(\md x)\bigg] - \gamma^2\lambda_d(B)^2\nonumber\\
 &= \int_B \alpha_{K\Phi}(B-x)\,\md x - \gamma^2\lambda_d(B)^2\nonumber\\
 &= \int_B \alpha_{L\Phi}(B-x)\,\md x + \int_{B} \eta(B-x)\,\md x - \gamma^2\lambda_d(B)^2\nonumber\\
  &= \BV[L\Phi(B)] + \int_B \eta(B-x)\,\md x.
 \end{align}
It remains to show that for all $W \in \cK_0$
\begin{equation}\label{e:086}
\lim_{r\to\infty}\frac{1}{\lambda_d(rW)}\int_{rW} \eta(rW - x)\,\md x= 0.
\end{equation}
We have
\begin{align*}
\frac{1}{\lambda_d(rW)}\int_{rW} \eta(rW - x)\,\md x &= \frac{1}{\lambda_d(rW)}\int_{rW}\int \I\{y\in(rW - x)\}\,\eta(\md y)\,\md x\\
&= \frac{1}{\lambda_d(rW)}\iint_{rW} \I\{x\in(rW - y)\}\,\md x\,\eta(\md y)\\
&= \int \frac{\lambda_d(rW\cap(rW - y))}{\lambda_d(rW)}\,\eta(\md y).
\end{align*}
The above integrand is bounded by $1$ and tends pointwise to $1$ as $r\to\infty$. Hence, the dominated convergence theorem yields
\eqref{e:086} because $\eta$ has finite total variation and total mass $0$, as previously stated in Lemma \ref{etafinitetv}.
\end{proof}
As one can see in the proof of Lemma \ref{etafinitetv}, it is possible to relax condition \eqref{maintheoremmixingcondition} 
slightly by replacing definition \eqref{e:kappa} by
\begin{equation*}   
\kappa(y):=\big\|\BE^\Phi_{0,y}[K_y\star K_0-L_y\star L_0]\big\|,\quad y\in\R^d.
\end{equation*}
Further, in both versions one can replace the transport kernels $K$
and $L$ by the relative displacement kernels $\Ks$ and $\Ls$, as
defined in \eqref{transportkernel2}, without changing $\kappa$.

We can also formulate a version of Theorem \ref{maintheoremmixing} in
Fourier space. If $K$ is a kernel from $\Omega\times\R^d$ to $\R^d$,
we will simplify the notation for the Fourier transform to
$\hat{K}_y(k) := \widehat{K_y}(k),$ for $y,k\in\R^d$.
\begin{theorem}\label{maintheoremmixingfourier}
In the setting of Theorem \ref{maintheoremmixing}, assume that \eqref{maintheoremmixingcondition} holds. Then
\begin{equation}\label{bartletmeasuredifference}
    \hat\beta_{K\Phi} = \hat\beta_{L\Phi} + \hat\eta\cdot\lambda_d,
\end{equation}
where the signed measure $\eta$ is defined by \eqref{etafirst}. Further, we have
\begin{equation}\label{maintheoremmixingfourierdensity}
\hat\eta(k) = 
\int \BE_{0,y}^\Phi \Big[\hat{K}_y(k)\overline{\hat{K}_0(k)} - \hat{L}_y(k)\overline{\hat{L}_0(k)}\Big] \, \alpha_\Phi(\md y),\quad k\in\R^d,
\end{equation}
and $\hat{\eta}$ is continuous with $\hat{\eta}(0) = 0$.
\begin{proof}
By Definition of $\eta$ in \eqref{etafirst}, we have 
\begin{equation}\label{e:431}
    \alpha_{K\Phi} - \gamma^2 \lambda_d = \alpha_{L\Phi} - \gamma^2 \lambda_d + \eta.
\end{equation}
Moreover, by Lemma \ref{etafinitetv} $\eta$ has finite total variation
and total mass $0$. Thus, $\hat\eta$ is well-defined, continuous, and fulfills
$\hat\eta(0)=0$. Now, \eqref{bartletmeasuredifference} directly
follows from the application of the Fourier transform to
  \eqref{e:431}, and \eqref{maintheoremmixingfourierdensity} follows
from the definition of $\eta$ and \eqref{e:201}.
\end{proof}
\end{theorem}

\subsection{Comparing the source and destination of a transport}

To prepare Theorem \ref{maintheoremmixing2}, we need to introduce some
further terminology. Following \cite{BH2024}, we call a bounded set $B\in\cB^d$
{\em Fourier smooth} if the Fourier transform $\widehat{\I_B}$ of $\I_B$ satisfies
\begin{align}\label{Fsmooth}
\widehat{\I_B}(k)\le c (1+\|k\|)^{-(d+\vartheta)/2}, \quad k\in\R^d,
\end{align}
for some $c,\vartheta>0$. The Euclidean ball is Fourier smooth; see \cite[Remark 3.5]{BH2024}. Any (non-random)
probability measure $L_0$ on $\R^d$ can be extended to a (non-random) probability kernel
$L$ from $\R^d$ to $\R^d$ by choosing
\begin{align}\label{detL}
L(x,B):=L_0(B-x),\quad (x,B)\in\R^d\times\cB^d.
\end{align}
Then, the relative displacement kernel $\Ls$, defined by \eqref{transportkernel2}, does not depend on its first argument and is equal to $L_0$, i.e.,
\begin{align*}
\Ls_x = \Ls(x)=L_0,\quad x\in\R^d.
\end{align*}

\begin{theorem}\label{maintheoremmixing2}
Let  $\Phi$ be a locally square-integrable invariant random measure,
and let $K$ be an invariant probability kernel from $\Omega\times\R^d$ to $\R^d$.
Define the probability kernel $\Ks$ by \eqref{transportkernel2}, and then define the function $\kappa\colon\R^d\to [0,2]$ by
\begin{align}\label{e:kappa2}        
\kappa(y):=\big\|\BE^\Phi_{0,y}[\Ks_y\otimes \Ks_0]-\big(\BE^\Phi_{0}[\Ks_0]\big)^{\otimes 2}\big\|,\quad y\in\R^d.
\end{align}
Assume that \eqref{maintheoremmixingcondition} holds. Then the following statements hold.
\begin{enumerate}
 \item[{\rm (i)}] Let $W \in \cK_0$ and assume that $\Phi$ is hyperuniform with respect to $W$.
Then $K\Phi$ is hyperuniform with respect to $W$.
\item[{\rm (ii)}] Assume, that either $W$ is Fourier smooth or $\beta_{\Phi}$ has finite total variation, i.e., \eqref{finitetotalv} holds.
Then 
\begin{align*}
\lim_{r\to\infty}\lambda_d(rW)^{-1}\BV[\Phi(rW)]=\lim_{r\to\infty}\lambda_d(rW)^{-1}\BV[K\Phi(rW)].
\end{align*}
\end{enumerate}
\end{theorem}
\begin{proof} We apply Theorem \ref{maintheoremmixing} with the deterministic kernel $L$ determined by 
\begin{equation}
\label{e:Lsdetkernel}
    \Ls_y=\BE^\Phi_0[ K_0], \quad y\in\R^d.
\end{equation}
By the forthcoming Lemma \ref{l:variancemono}, $L\Phi$ is hyperuniform if $\Phi$ is. Moreover, since $L$ is deterministic,
the mixing functions \eqref{e:kappa} and \eqref{e:kappa2} coincide, which gives the first result.
Under each of the additional assumptions, we have the upcoming equality of asymptotic variances of $\Phi$ and $L\Phi$ \eqref{e:4095}. So the second assertion follows from \eqref{e:equalasvariance} in Theorem \ref{maintheoremmixing}.
\end{proof}

Like with Theorem \ref{maintheoremmixingfourier} for Theorem \ref{maintheoremmixing}, we can also formulate a version of 
Theorem \ref{maintheoremmixing2} in Fourier space.

\begin{theorem}\label{maintheoremmixing2fourier}
  In the setting of Theorem \ref{maintheoremmixing2}, assume that
  \eqref{maintheoremmixingcondition} holds. Then we obtain
\begin{equation}\label{bartletmeasuredifference2}
    \hat\beta_{K\Phi} = \big|\BE_0^\Phi \big[\hat{K}^*_0\big]\big|^2\cdot \hat\beta_{\Phi} + \hat\eta\cdot\lambda_d,
\end{equation}
where the signed measure $\eta$ is defined as in \eqref{etafirst} with $L$ as in
\eqref{e:Lsdetkernel}. Further, we have
\begin{equation}\label{maintheoremmixingfourierdensity2}
  \hat\eta(k) =
\int e^{-i\langle k, y\rangle} \Big(\BE_{0,y}^\Phi \Big[\hat{K}^*_y(k)\overline{\hat{K}^*_0(k)}\Big] - \big|
\BE_0^\Phi \big[\hat{K}^*_0(k)\big]\big|^2 \Big)\, \alpha_\Phi(\md y),\quad k\in\R^d,
\end{equation}
and $\hat{\eta}$ is continuous with $\hat{\eta}(0) = 0$.
\begin{proof}
The assertions directly follow from the following Lemma \ref{l:variancemonofourier} and
Theorem \ref{maintheoremmixingfourier}, when one defines $L$ as in
\eqref{e:Lsdetkernel} (see \eqref{transportkernel2} also) in the proof of Theorem \ref{maintheoremmixing2}.
\end{proof}
\end{theorem}

An {\em invariant allocation} is a measurable mapping
$\tau\colon\Omega\times\R^d\rightarrow\R^d\cup\{\infty\}$ that is {\em equivariant}
in the sense that
\begin{align}\label{allocation}
\tau(\theta_y\omega,x-y)=\tau(\omega,x)-y,
\quad  x,y\in\R^d,\, \omega\in\Omega.
\end{align}
Let $\Phi$ be an invariant random measure with finite intensity $\gamma$.
Define 
\begin{equation}\label{e:allocaltion transport}
    \tau\Phi:= K(\I\{\tau(\cdot)\neq\infty\}\cdot\Phi) = \int \I\{\tau(x) \in \cdot\}\, \Phi(\md x)
\end{equation}
with $K_x:=\delta_{\tau(x)}$ if $\tau(x)\ne\infty$. Otherwise, let $K_x$ equal some fixed probability measure. Then $\tau\Phi$ is invariant, and a simple calculation (as in the proof of Proposition \ref{p:propinvPalm})
shows that $\tau\Phi$ has intensity $\BP^\Phi_0(\tau(0)\ne\infty)\gamma$. 

\begin{corollary}\label{c:alloc} Let $\Phi$ satisfy the assumptions of Theorem \ref{maintheoremmixing2}. Suppose that $\tau$ is an invariant allocation
such that $\BP^\Phi_0(\tau(0)\ne \infty)=1$.  Define the function $\kappa$ by \eqref{kappaspecial} and assume that
\eqref{maintheoremmixingcondition} holds. Then the assertions (i) and (ii) of Theorem \ref{maintheoremmixing2} hold
with $K\Phi$ replaced by $\tau\Phi$.
\end{corollary}
\begin{proof} Define the invariant probability kernel $K$ as in \eqref{e:allocaltion transport}.
Then $K\Phi=\tau\Phi$, as $\BP^\Phi_0(\tau(0)\ne \infty)=1$. Moreover, the mixing coefficient \eqref{e:kappa2} boils down to \eqref{kappaspecial}.
Therefore, the assertions follow from Theorem \ref{maintheoremmixing2}.
\end{proof}

\begin{lemma}\label{l:variancemono}  Let $\Phi$ be an invariant locally square-integrable
random measure on $\R^d$, and let $L$ be a probability kernel from  $\R^d$ to $\R^d$ given as in
\eqref{detL}. Then we have for each bounded $B\in\cB^d$ that 
\begin{equation}
            \BV[L\Phi(B)] \leq \BV[\Phi(B)].
\end{equation}
If, in addition, either $W$ is Fourier smooth or $\beta_{\Phi}$ has finite total variation (i.e., \eqref{finitetotalv} holds), then
\begin{align}\label{e:4095}
\lim_{r\to\infty}\lambda_d(rW)^{-1}\BV[\Phi(rW)]=\lim_{r\to\infty}\lambda_d(rW)^{-1}\BV[L\Phi(rW)]
\end{align}
\end{lemma}
\begin{proof} Proposition \ref{p:propinvPalm} shows that $L\Phi$ and $\Phi$ have the
same intensity. Let $B\in\cB^d$. Then
\begin{align*}
L\Phi(B)&=\int L_0(B-x)\, \Phi(\md x)
= \iint \I_{B-x}(y)\,L_0(\md y)\, \Phi(\md x)
= \int \Phi(B-y)\,L_0(\md y).
\end{align*}
Therefore, we obtain from the Cauchy-Schwarz inequality and invariance that
\begin{align*}
\BE[L\Phi(B)^2] &= \BE\bigg[\int \Phi(B-x)\Phi(B-y)\, L_0^2(\md (x,y)) \bigg]\\
            &= \int \BE[\Phi(B-x)\Phi(B-y)]\, L_0^2(\md (x,y))\\
            &\le \int \sqrt{\BE[\Phi(B-x)^2]\BE[\Phi(B-y)^2]}\,L_0^2(\md (x,y))
= \BE[\Phi(B)^2].
\end{align*}
This inequality implies the first assertion.

The second assertion can be derived from the upcoming Lemma \ref{l:variancemonofourier}. First assume that $W$ is Fourier smooth. Then \eqref{e:4095} can be derived from \eqref{e:451} using \eqref{e:spectralvariance}. Now, instead assume that $\hat\beta_\Phi$ has finite total variation. Then using forthcoming \eqref{e:bartletsmoothing} and $\hat{L}_0(0)= L_0(\R^d) = 1$, we derive that
\begin{equation}
    S_{L\Phi}(0) = |\hat{L}_0(0)|^2 S_\Phi(0) = S_\Phi(0),
\end{equation}
which now gives the second assertion via \eqref{e:asvariancefourier}.
\end{proof}
In Fourier space, one can get the following explicit formula that was used in the previous proof.
\begin{lemma}\label{l:variancemonofourier}
In the setting of Lemma \ref{l:variancemono}, we have
\begin{equation}\label{e:bartletsmoothing}
    \hat\beta_{L\Phi} = |\hat{L}_0|^2\cdot \hat\beta_\Phi.
\end{equation}
Further, suppose that $W \in \cK_0$ and also a Fourier smooth set. Then
\begin{equation}\label{e:451}
    \lim_{r\to\infty} \lambda_d(rW)^{-1} \int \big|\widehat{\I_{rW}}(k)\big|^2\, \hat\beta_\Phi(\md k) = \lim_{r\to\infty} \lambda_d(rW)^{-1} \int \big|\widehat{\I_{rW}}(k)\big|^2\, \hat\beta_{L\Phi}(\md k).
\end{equation}
\begin{proof}
From Lemma \ref{l:345} we obtain
\begin{align}
    \alpha_{L\Phi}(B) &= \int (L_{y}\star L_{0})(B)\, \alpha_\Phi(\md y)\nonumber\\
    &= \int (L_0\star L_0)(B-y)\alpha_\Phi(\md y),\quad B\in\cB^d.
\end{align}
Now, let $f,g$ be measurable and bounded functions with compact
support. Further, define $g_z$ by $g_z(x) := g(x-z)$ for
$x,z\in\R^d$. Note that
$\hat{g}_z(k) = e^{-i\langle k, z\rangle}\hat{g}(k)$ for
$k\in\R^d$. Using this, we get
\begin{align}
    \int \hat{f}(k)\overline{\hat{g}(k)}\, \hat\alpha_{L\Phi}(\md k) &= \int (f\star g)(x)\, \alpha_{L\Phi}(\md x)\nonumber\\
    &= \iint (f\star g)(y+z)\, (L_0\star L_0)(\md z)\, \alpha_\Phi(\md y)\nonumber\\
    &= \iint (f\star g_z)(y)\, \alpha_\Phi(\md y)\, (L_0\star L_0)(\md z)\nonumber\\
    &= \iint \hat{f}(k)\overline{\hat{g}_z(k)}\, \hat\alpha_\Phi(\md k)\, (L_0\star L_0)(\md z)\nonumber\\
    &= \iint e^{i\langle k, z\rangle}\, (L_0\star L_0)(\md z)\, \hat{f}(k)\overline{\hat{g}(k)} \hat\alpha_\Phi(\md k)\nonumber\\
    &= \int \hat{f}(k)\overline{\hat{g}(k)} |\hat{L}_0(k)|^2 \hat\alpha_\Phi(\md k).
\end{align}
As $f,g$ were arbitrary, this equality implies
\begin{equation}
    \hat\alpha_{L\Phi} = |\hat{L}_0|^2\cdot \hat\alpha_\Phi,
\end{equation}
and thus \eqref{e:bartletsmoothing} holds as well because $\hat{L}_0(0)=L_0(\R^d)=1$ , $\hat\lambda_d = \delta_0$, and $L\Phi$ has same intensity as $\Phi$ by Proposition \ref{p:propinvPalm}.

Now, suppose that $W \in \cK_0$ and also a Fourier smooth set with constants $c,\vartheta>0$. Let $\varepsilon>0$. Then, as $\big|\widehat{\I_{rW}}(\cdot)\big| = r^{d}\big|\widehat{\I_{W}}(r\cdot)\big|$, we get
\begin{align}
    r^{-d} \int_{B_\varepsilon^c} \big|\widehat{\I_{rW}}(k)\big|^2\, \hat\beta_\Phi(\md k) &= r^d \int_{B_\varepsilon^c} \big|\widehat{\I_{W}}(rk)\big|^2\, \hat\beta_\Phi(\md k)\nonumber\\
    &\leq r^d \int_{B_\varepsilon^c} c(1+r\|k\|)^{-(d+\vartheta)}\, \hat\beta_\Phi(\md k)\nonumber\\
    &\leq r^{-\vartheta} c\int_{B_\varepsilon^c} \|k\|^{-(d+\vartheta)}\, \hat\beta_\Phi(\md k).
\end{align}
From \cite[Proposition 4.9]{BergForst75}, we know that $\hat\beta_\Phi$ is translation bounded. Hence, the last integral is finite just like it is with respect to the Lebesgue measure. Thus,
\begin{equation}
    \lim_{r\to\infty} r^{-d} \int_{B_\varepsilon^c} \big|\widehat{\I_{rW}}(k)\big|^2\, \hat\beta_\Phi(\md k) = 0.
\end{equation}
Because the same holds for $L\Phi$, for \eqref{e:451} it suffices to show that 
\begin{equation}\label{e:452}
    \lim_{\varepsilon\to0}\lim_{r\to\infty} r^{-d} \int_{B_\varepsilon} \big|\widehat{\I_{rW}}(k)\big|^2\, \hat\beta_{L\Phi}(\md k) = \lim_{\varepsilon\to0}\lim_{r\to\infty} r^{-d} \int_{B_\varepsilon} \big|\widehat{\I_{rW}}(k)\big|^2\, \hat\beta_\Phi(\md k).
\end{equation}
From \eqref{e:bartletsmoothing} and the fact that $|\hat{L}_0| \leq 1$, we obtain 
\begin{align}
    r^{-d} \int_{B_\varepsilon} \big|\widehat{\I_{rW}}(k)\big|^2\, \hat\beta_{L\Phi}(\md k) &\leq r^{-d} \int_{B_\varepsilon} \big|\widehat{\I_{rW}}(k)\big|^2\, \hat\beta_\Phi(\md k),\\
    r^{-d} \int_{B_\varepsilon} \big|\widehat{\I_{rW}}(k)\big|^2\, \hat\beta_{L\Phi}(\md k) &\geq \inf_{k\in B_\varepsilon} (|\hat{L}_0(k)|^2)\, r^{-d} \int_{B_\varepsilon}\big|\widehat{\I_{rW}}(k)\big|^2\, \hat\beta_\Phi(\md k).
\end{align}
As $\hat{L}_0$ is continuous and $\hat{L}_0(0)=1$, these bounds imply \eqref{e:452}, and thus also \eqref{e:451}.
\end{proof}
\end{lemma}

\section{Randomizing transports}\label{sec:randtransport}

In this section we assume that $\Phi$ is a (stationary) discrete random measure.
As in Theorem \ref{maintheoremmixing} we consider two invariant random probability kernels
$K$ and $L$. But this time we assume that given $L\Phi$, the random measures $K(x)$, $x\in\Phi$,
have the conditional mean $L(x)$ and are uncorrelated for different $x\in\Phi$.
Heuristically, we then have $\kappa(y)=0$ for all $y\in\R^d\setminus\{0\}$, so that
assumption \eqref{maintheoremmixingcondition} can be dropped.
On the other hand we do not need to assume that $\Phi$ is locally square integrable.
It is enough that $L\Phi$ has this property. Then Theorem \ref{t:splitting} shows
that $L\Phi$ and $K\Phi$ have the same asymptotic variance.
Theorem \ref{t:splitting2} is a more explicit version of this result.
In this theorem we assume the random measures $K(x)$, $x\in\Phi$, to be
conditionally independent with a conditional distribution that is determined
by $L(x)$ in a certain invariant way. Effectively this means that 
$K$ is constructed by randomizing $L$ 
in an invariant way.
This theorem is applied in Subsection \ref{ss:applications_randomtransp}
to construct hyperuniform processes starting with a general simple stationary
point process. The interested reader might go there directly without
studying the general background in Subsection \ref{ss:genrandomtransp}.

As said above we consider a {\em discrete} random measure $\Phi$, that is, a random element
of the space $\bM_d(\R^d)\subset \bM(\R^d)$ of  all $\varphi\in\bM(\R^d)$ 
with discrete support. For $\varphi\in \bM_d(\R^d)$ and $x\in\R^d$, we write $x\in\varphi$
if $\varphi\{x\}:=\varphi(\{x\})>0$. 
Unless stated otherwise we are working in the setting of
Subsection \ref{subPalm}.

\subsection{General results}\label{ss:genrandomtransp}

The first result of this section shows that transporting $\Phi$ with two
invariant probability kernels $L$ and $K$ leads to random measures with
the same asymptotic variance, provided the conditional covariance structure of $K$
is determined by $L$ in a specific way.

\begin{theorem}\label{t:splitting} Suppose that $\Phi$ is  an invariant purely discrete random measure
with finite intensity.
Let $L$ be an invariant probability kernel from $\Omega\times\R^d$ to $\R^d$, and
let $T$ be the random transport given by $T(d(x,y)):=L(x,dy)\Phi(dx)$.
Further, let $K$ be another  invariant probability kernel from $\Omega\times\R^d$ to $\R^d$,
satisfying
\begin{align}
\label{a:2}
\BE[K_x\otimes K_y \mid T] = L_x\otimes L_y,\quad x,y\in\Phi, x\ne y, \BP\text{-a.s.}
\end{align}
Assume that 
\begin{align}\label{851}
\BE^\Phi_0\Phi\{0\}<\infty.
\end{align} 
Then $K\Phi$ is locally square-integrable iff $L\Phi$ is locally square-integrable. In this case, we again have equality of asymptotic variances of $K\Phi$ and $L\Phi$ i.e.,
\eqref{e:equalasvariance} holds if one  of the limits exist.
  \end{theorem}
\begin{proof} We follow the proof of Theorem \ref{maintheoremmixing}.
We cannot apply the latter directly, since we have not assumed that
$\Phi$ is locally square-integrable. Let $B\in\cB^d$. By definition of $K\Phi$ and $\alpha_{K\Phi}$,
we have
\begin{align*}
\alpha_{K\Phi}(B) = M_1(B)+M_2(B),
\end{align*}
where 
\begin{align*}
M_1(B)&:=\BE\iiint \I\{x\in[0,1]^d,y=z\}K(y,B+x)K(z,\md x)\,\Phi(\md y)\,\Phi(\md z),\\
M_2(B)&:=\BE\iiint \I\{x\in[0,1]^d,y\ne z\}K(y,B+x)K(z,\md x)\,\Phi(\md y)\,\Phi(\md z).
\end{align*}
As in the proof of \eqref{e:200}, we obtain
\begin{align*}
M_1(B)&=\BE^\Phi_0\int K(0,B+x)K(0,\md x)\Phi\{0\}\\
&=\BE^\Phi_0 (K_0 \star K_0)(B)\Phi\{0\}.
\end{align*}
By assumption \eqref{a:2}, and because $\Phi$ is measurable with respect to $T$,
\begin{align*}
M_2(B) &= \BE\iiint \I\{x\in[0,1]^d,y\neq z\}\BE\left[K(y,B+x)K(z,\md x)\mid T\right]\,\Phi(\md y)\,\Phi(\md z)\\
&= \BE\iiint \I\{x\in[0,1]^d,y\ne z\}L(y,B+x)L(z,\md x)\,\Phi(\md y)\,\Phi(\md z).
\end{align*}
A similar decomposition of $\alpha_{L\Phi}$ yields
\begin{align*}
\alpha_{L\Phi}(B)=M'_1(B)+M_2(B),
\end{align*}
where 
\begin{align*}
M'_1(B):=\BE^\Phi_0 (L_0\star L_0)(B)\Phi\{0\}.
\end{align*}
By our assumption that $\BE^\Phi_0\Phi\{0\}<\infty$ and that $K,L$ are probability kernels, it follows that 
$M_1(B)+M_1'(B)<\infty$. If $B$ is bounded, this implies that 
$\alpha_{K\Phi}(B)<\infty$ iff $\alpha_{L\Phi}(B)<\infty$ iff $M_2(B) < \infty$. Hence, because $B$ was arbitrary, $\BE(K\Phi(B)^2)<\infty$ iff 
$\BE(L\Phi(B)^2)<\infty$, proving the first assertion.

In the remainder of the proof, we assume that $M_2$ is locally finite,
so that $K\Phi$ and $L\Phi$ are locally square-integrable.
The signed measure $\eta$ (see \eqref{etafirst})  is then given by
\begin{align}\label{e:501}
\eta=\gamma\,\BE^\Phi_0 [(K_0\star K_0-L_0\star L_0)\Phi\{0\}].
\end{align}
Therefore, $\eta(\R^d)=0$ and
the total variation of $\eta$ is bounded by $2\gamma\,\BE^\Phi_0\Phi\{0\}$.
Now, the second assertion follows as in the proof of Theorem \ref{maintheoremmixing}.
\end{proof}  

Again, we can express this in Fourier space, but with a much simpler
formula compared to Theorems \ref{maintheoremmixingfourier} and
\ref{maintheoremmixing2fourier}.

\begin{theorem}\label{t:splittingfourier}
  In the setting of Theorem \ref{t:splitting}, assume that
  $\BE^\Phi_0\Phi\{0\}<\infty$, and that $K\Phi$ or $L\Phi$ is locally
  square-integrable. Then we obtain
\begin{equation}\label{bartletmeasuredifference3}
    \hat\beta_{K\Phi} = \hat\beta_{L\Phi} + \hat\eta\cdot\lambda_d,
\end{equation}
where the signed measure $\eta$ is defined by \eqref{etafirst}. Further, we have
\begin{equation}\label{splittingfourierdensity}
    \hat\eta(k) = 
\gamma\,\BE_0^\Phi \big[\big(\big|\hat{K}_0(k)\big|^2 - \big|\hat{L}_0(k)\big|^2\big)\Phi\{0\}\big],\quad k\in\R^d,
\end{equation}
and $\hat\eta$ is continuous with $\hat\eta(0)=0$.
\end{theorem}
\begin{proof} First, we recall \eqref{e:431}.
Moreover, $\hat\eta$ is well-defined, as from the proof of Theorem \ref{t:splitting}, we get that $\eta$ has finite total
variation. Now, \eqref{bartletmeasuredifference3} directly follows
from the application of the Fourier transform, and
\eqref{splittingfourierdensity} follows from the representation
\eqref{e:501} of $\eta$ and \eqref{e:200} in Lemma \ref{l:345}. Continuity of $\hat\eta$ follows because $\eta$ has finite total variation, and $\hat\eta(0)=0$ because $\eta(\R^d)=0$, as seen in the proof of Theorem \ref{t:splitting}. 
\end{proof}

With the following theorem, we would like to make the assumptions of
Theorem \ref{t:splitting} more explicit. Let $\bM^1\equiv \bM^1(\R^d)$ be the space of probability measures
on $\R^d$, a measurable subset of $\bM(\R^d)$. 
Let $\Pi_d$ be the space of all probability measures on $\bM^1(\R^d)$, equipped
with the standard $\sigma$-field.
We shall a consider a measurable mapping
$F\colon \bM^1(\R^d)\to \Pi_d$ with the mean value property 
\begin{align}\label{F1}
\int \pi'(\cdot) F(\pi)(\md \pi') =\pi(\cdot),
\end{align}
and the translation covariance property
\begin{align}\label{F2}
F(\theta_x\pi)=\int \I\{\theta_x\pi'\in\cdot\}\,F(\pi)(\md \pi'),\quad x\in\R^d,\pi\in \bM^1(\R^d).
\end{align}
Examples will be given in the next subsection.

\begin{theorem}\label{t:splitting2} Let $T$ be a stationary transport on $\R^d$
such that $\Phi:=T(\cdot\times\R^d)$ is  purely discrete and has a finite
intensity.
Let $L$ be a probability kernel from $\Omega\times\R^d$ to $\R^d$ 
satisfying $T(\md (x,y))=L(x,\md y)\Phi(\md x)$.
Let $K$ be another  probability kernel from $\Omega\times\R^d$ to $\R^d$
such that the family $\{K_x:x\in\Phi\}$ is conditionally independent given $T$, and 
\begin{align}\label{erandomization}
\BP(K_x\in \cdot \mid T)=F(L_x),\quad x\in\Phi,\, \BP\text{-a.s.},
\end{align}
where the measurable mapping $F\colon \bM^1(\R^d)\to \Pi_d$ satisfies \eqref{F1} and \eqref{F2}.
Finally, assume that \eqref{851} holds. 
Then $K\Phi$ is locally square-integrable iff $L\Phi$ has this property. In this case, 
the equality of asymptotic variances of $K\Phi$ and $L\Phi$ as in
\eqref{e:equalasvariance} holds if one of the limits exists.
\end{theorem}
\begin{proof} Our goal is to apply Theorem \ref{t:splitting}. To do so, we shall
construct a probability space $(\Omega',\mathcal{A}',\BP')$, equipped
with a flow $\{\theta_x':x\in\R^d\}$ keeping $\BP'$ invariant. On this space, we shall define invariant
versions $\Phi'$ and $L'$ of $\Phi$ and $L$, respectively, along with an invariant probability kernel $K'$ 
such that \eqref{a:2} holds for the $'$-objects and, moreover,
$L\Phi\overset{d}{=}L'\Phi'$ and $K\Phi\overset{d}{=}K'\Phi'$. 

Let $\Omega^\smallbox$ be set of all $\psi\in\bM(\R^d\times\R^d)$
that $\psi(\cdot\times\R^d)\in\bM_d$. 
We  equip this space with the natural $\sigma$-field $\mathcal{A}^\smallbox$ 
the diagonal shift
and the probability measure $\BP^\smallbox:=\BP(T\in\cdot)$.
By assumption $\BP^\smallbox$ is stationary. Let $T^\smallbox$ denote the identity on $\Omega^\smallbox$.
We can disintegrate $T^\smallbox$ as
\begin{align*}
T^\smallbox(\md (x,y))=L^\smallbox(T^\smallbox,x,\md y)\,T^\smallbox(\md x\times\R^d),
\end{align*}
 where the probability kernel $L^\smallbox$ from $\Omega^\smallbox\times\R^d$ to $\R^d$
is defined by
\begin{align}
L^\smallbox(T^\smallbox,x,\cdot):=\frac{T^\smallbox(\{x\}\times\cdot)}{T^\smallbox(\{x\}\times\R^d)},
\end{align}
if $T^\smallbox(\{x\}\times\R^d)>0$. If $T^\smallbox(\{x\}\times\R^d)=0$, then we
take $L^\smallbox(T,x,\cdot)$ as a fixed probability measure on $\R^d$.
The kernel $L^\smallbox$ is invariant.
Over $\Omega^\smallbox$ we define the discrete random measure $\Psi^\smallbox$ on $\R^d\times\bM^1$
by 
\begin{align}
\Psi^\smallbox:=\int \I\{(x,L(T^\smallbox,x))\in\cdot\}\,T^\smallbox(\md x\times\R^d).
\end{align}
This random measure is a measurable function of $T^\smallbox$ and vice versa.
It is easy to check that $\Psi^\smallbox$ is stationary w.r.t.\ joint shifts, that is
\begin{align}\label{e0345}
\int\I\{(x-w,\theta_w\pi)\in\cdot\}\,\Psi^\smallbox(\md (x,\pi))\overset{d}{=}\Psi^\smallbox,\quad w\in\R^d.
\end{align}
Define a probability kernel $H$ from $\R^d\times\bM^1$ to $\bM^1$ by
\begin{equation}\label{e:H def}
    H(x, \pi, \cdot) := F(\pi),\quad (x,\pi)\in\R^d\times \bM^1.
\end{equation}
We now extend the probability space $(\Omega^\smallbox,\mathcal{A}^\smallbox,\BP^\smallbox)$  
so as to carry a (position dependent) $H$-marking $\tilde\Psi$ of $\Psi^\smallbox$. 
This marking attaches to every point from $\Psi^\smallbox$ a random mark from $M^1$, so that 
$\tilde \Psi$ becomes a random measure on $\R^d\times\bM^1\times\bM^1$.
Given $\Psi^\smallbox$, the marks are conditionally independent with conditional
distribution $F(L^\smallbox(T^\smallbox,x))$ for $x\in\Psi^\smallbox$.
The marking can be based on a representation
\begin{align}\label{e:4.826}
T^\smallbox(\cdot\times\R^d)=\sum^\infty_{n=1}Y_n\delta_{X_n},
\end{align}
where $Y_1,Y_2,\ldots$ are non-negative random variables and
$X_1,X_2,\ldots$ are random vectors in $\R^d$  such that $X_m\ne X_n$ whenever
$Y_m\ne 0$ or $Y_n\ne 0$; see e.g.\ \cite[Chapter 6]{LastPenrose17}.
The marking is then defined just as in \cite[Chapter 5]{LastPenrose17},
where the case $Y_n\equiv 1$ is treated. For notational convenience we still denote
the extended probability space by $(\Omega^\smallbox,\mathcal{A}^\smallbox,\BP^\smallbox)$
and keep our notation for $T^\smallbox$ and $L^\smallbox$.
We claim that the random measure $\tilde{\Psi}$ is stationary w.r.t.\ joint shifts. 
To check this, we
take a measurable $g\colon \R^d\times\bM^1\times\bM^1\to[0,\infty)$
and $w\in\R^d$. As in the proof of Proposition 5.4 in \cite{LastPenrose17}
 we obtain that
\begin{align*}
\BE^\smallbox &\exp\bigg[-\int g(x-w,\theta_w\pi,\theta_w\pi')\,\tilde\Psi(\md (x,\pi,\pi'))\bigg]\\
&=\BE^\smallbox\exp\bigg[\sum^\infty_{n=1} \log\int\exp\big[- Y_ng(X_n-w,\theta_w L^\smallbox(T^\smallbox,X_n),\theta_w\pi')\big]
F(L^\smallbox(T^\smallbox,X_n))(\md \pi')\bigg] \\
&=\BE^\smallbox\exp\bigg[\sum^\infty_{n=1} \log\int\exp\big[- Y_ng(X_n-w,L^\smallbox(T^\smallbox,X_n),\pi')\big]
F(\theta_wL^\smallbox(T^\smallbox,X_n))(\md \pi')\bigg], 
\end{align*}
where the second identity comes from the translation covariance \eqref{F2} of $F$.
In view of \eqref{e0345} this can easily be seen to be independent of $w$,
so that \cite[Proposition 2.10]{LastPenrose17} implies the claim.

Now we define a random measure
$\tilde{T}$ on $\R^d\times\R^d\times\R^d$ by
\begin{align}
\tilde{T}=\int \I\{(x,y,z)\in\cdot\}\,\pi(\md y)\,\pi'(\md z)\,\tilde\Psi(d(x,\pi,\pi')).
\end{align}
Since $\tilde\Psi$ is stationary, $\tilde{T}$ is stationary under diagonal (joint) shifts.

Finally, we can choose $\Omega':=\bM(\R^d\times\R^d\times\R^d)$ with
the appropriate $\sigma$-field $\mathcal{A}'$, the probability measure
$\BP':=\BP_{\tilde{T}}$ and diagonal shifts. 
By stationarity of $\BP'$ and Proposition \ref{propkernel} there exist
invariant kernels  $L'$ and $K'$ satisfying
\begin{align*}
\omega(B\times C\times\R^d) &= \int_B L'(\omega,x, C)\, \omega(\md x\times\R^d\times\R^d),\\
\omega(B\times\R^d\times C) &= \int_B K'(\omega,x,C)\, \omega(\md x\times\R^d\times\R^d),\quad B,C\in\cB^d,
 \end{align*}
for $\BP'$-a.e.\ $\omega\in\Omega'$. 
The kernels $L',K'$ satisfy
\eqref{a:2} by choice of $H$ in \eqref{e:H def}, property
\eqref{F1} of $F$ and the conditional independence of the position
dependent marking in the construction of $\tilde\Psi$.
Therefore and by our assumption \eqref{851},
Theorem \ref{t:splitting} applies. It remains to show the required
distributional identities.
Define $T'(\omega):=\omega(\cdot\times \R^d)$ for $\omega\in\Omega'$.
By construction, $T'\overset{d}{=}T^\smallbox\overset{d}{=}T$. Therefore
we have $L\Phi\overset{d}{=}L'\Phi'$  and in particular
$\Phi\overset{d}{=}\Phi'$, where $\Phi'(\omega):=\omega(\cdot\times \R^d\times\R^d)$.
By definition of a position dependent marking
the family $\{K'_x:x\in\Phi'\}$ is conditionally independent given $T'$, and
\eqref{erandomization} holds for the '-objects.
Since $\Phi\overset{d}{=}\Phi'$, this implies $K\Phi\overset{d}{=}K'\Phi'$,
finishing the proof.
\end{proof}

\begin{remark}\label{r:isotropic splitting}\rm
In the setting of Theorem \ref{t:splitting2}, assume that $T$ is isotropic, i.e., 
    \begin{equation*}
        \rho(T) \overset{d}{=} T
    \end{equation*}
    for any isometry $\rho:\R^d\to\R^d$, where $\rho(T)(B\times C) := T(\rho^{-1}(B)\times \rho^{-1}(C))$ for $B,C\in\cB^d$. Further, assume that $F$ is isometry-covariant, i.e.,
    \begin{equation*}
        F(\rho(\pi)) = \int \I\{\rho(\pi')\in\cdot\}\, F(\pi)(\md \pi')
    \end{equation*}
    for any isometry $\rho:\R^d\to\R^d$ and $\pi\in M^1(\R^d)$. Then $K\Phi$ is isotropic, i.e.,
    \begin{equation*}
        \rho(K\Phi) \overset{d}{=} K\Phi,
    \end{equation*}
where $\rho(K\Phi)(B) := K\Phi(\rho^{-1}(B))$ for $B\in\cB^d$. 

Similarly as in the proof of Theorem    \ref{t:splitting2},
this can be proved using the Laplace functional of $K\Phi$.
\end{remark}

Next we formulate the Fourier version of Theorem \ref{t:splitting2}.
Given a measurable function $f\colon\bN\to\R$ we write
$\BE^\Phi_0 f(\Phi)$ to denote the integral of $f$ w.r.t.\ the Palm probability
measure of $\Phi$ as defined on the canonical space $\bN$.
This slight abuse of notation should not cause any confusion. 

\begin{theorem}\label{t:splittingfourier2} Let the assumptions of
Theorem \ref{t:splitting2} be satisfied. Then the spectral measure of $K\Phi$ is given by
\eqref{bartletmeasuredifference3}, where
\begin{equation}
    \hat\eta(k) = \gamma\,\BE_0^\Phi\bigg[\bigg(\int |\hat{\pi}(k)|^2F(L_0)(\md \pi) - \big|\hat{L}_0(k)\big|^2\bigg)\Phi\{0\}\bigg],\quad k\in\R^d.
\end{equation}
\end{theorem}
\begin{proof} We are using the notation from the
proof of Theorem \ref{t:splitting2}. Using the invariance properties of $L'$ and
$F$ it can be easily shown that 
\begin{align*}
(\BP')^{\Phi'}_0(K'(0)\in\cdot)=\iint \I\{\pi\in\cdot\}F(L'(\varphi,0))(\md \pi)\,\BP^{\Phi'}_0(\md \varphi).
\end{align*}
Since $L'\Phi'\overset{d}{=}L\Phi$ (and in particular $\Phi'\overset{d}{=}\Phi$), 
the above right-hand side equals
\begin{align*}
\BE^\Phi_0\bigg[\int \I\{\pi\in\cdot\}F(L_0)(\md \pi)\bigg].
\end{align*}
Therefore the result follows from Theorem \ref{t:splittingfourier}.
\end{proof}

\subsection{The hyperuniformerer}
\label{ss:applications_randomtransp}

In order to state a (still quite general) application of Theorem \ref{t:splitting2},
we need some definitions. 
Let $\Psi$ be an invariant simple point process with finite intensity $\gamma$ and
$\BP(\Psi(\R^d)=0)=0$ and let $\tau$ be an allocation; see \eqref{allocation}.
We call the pair $(\Psi,\tau)$ an {\em invariant partition} if
$\tau(\omega,x)\in\Psi(\omega)$ for all $\omega$ with $\Psi(\omega)\ne 0$
and all $x\in\R^d$. 
Given such an invariant partition, we define
\begin{align*}
C^\tau(x):=\{y\in\R^d: \tau(y)=x\}, \quad x\in\R^d.
\end{align*}
Then, $\{C^\tau(x):x\in\Psi\}$ is a random partition of $\R^d$, whenever $\Psi\ne 0$. 
For $x\in\Psi$, we refer to $C^\tau(x)$ as the cell with {\em center} $x$, even though 
we do not assume that $x\in C^\tau(x)$.

\begin{example}\label{ex:5.2}\rm Let $(\Psi,\tau)$ be an invariant partition,
and assume that 
\begin{align*}
\BP(\text{$0<\lambda_d(C^\tau(x))<\infty$ for all $x\in\Psi$})=1,
\end{align*}
and that the volume of the {\em zero cell} has
a finite expectation, that is
\begin{align}\label{zerocell}
\BE \lambda_d(V_0)<\infty,
\end{align}
where $V_0:=\{x\in\R^d:\tau(x)=\tau(0)\}$.

Define a random transport $T$ by
\begin{align}\label{eTTT}
T:=\iint \I\{(x,y)\in\cdot\}\I\{y\in C^\tau(x)\}\,\md y\,\Psi(\md x).
\end{align} 
Then we have $\BP$-a.s.\ that
\begin{align}\label{e:9333}
T(\cdot\times\R^d)=\sum_{x\in\Psi}\lambda_d(C^\tau(x))\delta_x,\quad
T(\R^d\times \cdot)=\lambda_d.
\end{align}
Let $Z(x)$, $x\in\Psi$,  be random vectors in $\R^d$ which are conditionally independent
given $T$ and satisfy
\begin{align}\label{e:9435}
\BP(Z(x)\in\cdot\mid T)=L(x,\cdot),\quad x\in\Psi,\,\BP\text{-a.s.},
\end{align}
where 
\begin{align}\label{e:9222}
L(x,\cdot):=\lambda_d(C^\tau(x))^{-1}\lambda_d(C^\tau(x)\cap\cdot)
\end{align}
if $x\in \Psi$ and $0<\lambda_d(C^\tau(x))<\infty$. If $x\notin\Psi$ or if
$\lambda_d(C^\tau(x))\in\{0,\infty\}$, we choose $L(x,\cdot)$ as a fixed probability
measure. Note that we have a.s.\ that $L\Phi=\lambda_d$, where
$\Phi:=T(\cdot\times\R^d)$.
We would like to apply Theorem \ref{t:splitting} 
to show that the random measure
\begin{align}\label{eGamma}
\Gamma:=\sum_{x\in\Psi}\I\{0<\lambda_d(C^\tau(x))<\infty\}\lambda_d(C^\tau(x))\delta_{Z(x)}
\end{align}
is hyperuniform. See Figure \ref{fig:weighted_voronois} (left) for an illustration of $\Gamma$. To do so, we choose 
\begin{align}\label{e:F_hyperuniformerer}
F(\pi):=\int\I\{\delta_z\in\cdot\}\,\pi(dz)
\end{align}
and $K(x,\cdot)=\delta_{Z(x)}$, $x\in\Phi$. Then $L,K$ satisfy the assumptions
of Theorem \ref{t:splitting2} and $K\Phi=\Gamma$ a.s.

It remains to make sure that assumption \eqref{851} holds. It follows by the refined Campbell theorem
that the intensity $\gamma_\Phi$ of $\Phi$ equals one, and that
\begin{align*}
\BE^\Phi_0\Phi\{0\}=\gamma_\Psi \BE^\Psi_0[\lambda_d(C^\tau(0))^2].
\end{align*}
By \cite[Corollary 4.1]{Last06}, we have $\gamma_\Psi \BE^\Psi_0\lambda_d(C^\tau(0))^2=\BE \lambda_d(V_0)$
which is finite by assumption \eqref{zerocell}. 

We can also apply Theorem \ref{t:splittingfourier2} to calculate the structure factor of $\Gamma$. 
Due to the special form \eqref{e:F_hyperuniformerer} of $F$ we have
for all $k\in\R^d$ and all $\pi'\in\bM^1(\R^d)$ that
$\hat{\pi}(k)=1$ for $F(\pi')$-a.e.\ $\pi$. Because $\beta_{\lambda_d}=0$ we
therefore obtain from Theorem \ref{t:splittingfourier2} that 
\begin{align*}
    S_\Gamma(k) &= \BE_0^\Phi\big[\big((\Phi\{0\})^2 - \big|\widehat{\I_{C^\tau(0)}}(k)\big|^2\big) (\Phi\{0\})^{-1}\big]\\
&= \BE_0^\Phi\big[\lambda_d(C^\tau(0)) - \lambda_d(C^\tau(0))^{-1}\big|\widehat{\I_{C^\tau(0)}}(k)\big|^2\big) \big]
\end{align*}
Using the refined Campbell theorem and then \cite[Proposition 4.3]{Last06} we obtain
\begin{align}\notag
S_\Gamma(k) &= \gamma_\Psi\, \BE_0^\Psi\big[\lambda_d(C^\tau(0))^2 -\big|\widehat{\I_{C^\tau(0)}}(k)\big|^2\big]\\
\label{SGamma}
&= \BE\big[\big(\lambda_d(V_0) - \lambda_d(V_0)^{-1}\big|\widehat{\I_{V_0}}(k)\big|^2\big)\big].
\end{align}
\end{example}

Next, we specialise Example \ref{ex:5.2} to the case where all cells have equal volume.
Choosing in each of the cells a point purely at random and conditionally independent (given
$(\Psi,\tau)$) for different cells, yields a hyperuniform point process.

\begin{figure}[t]
  \centering
 \includegraphics[width=\textwidth]{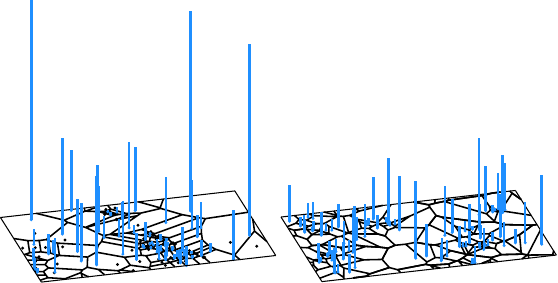}
  \caption{{Hyperuniform weighted point processes: (Left) We start from a
    Poisson hyperplane intersection process (PHIP) and construct the
    corresponding Voronoi tessellation. Despite the long-range correlations
    of this hyperfluctuating model, we can construct a hyperuniform random
    measure according to \eqref{eGamma}, i.e., we place in each
    cell~---~independently and uniformly distributed~---~a point with a
    weight equal to the cell's area. (Right) For an initial point process
    with exponentially fast decay of correlations, here a Mat\'ern process,
    we can place the weighted point at the Voronoi center (i.e., without
    further randomness); see Sec.~\ref{s:volumes}.}}
  \label{fig:weighted_voronois}
\end{figure}

\begin{example}[Hyperuniformerer]\label{ex:hyperuniformerer}\rm Let $(\Psi,\tau)$ be an invariant partition.
Assume that the partition is {\em fair} (or balanced), that is
\begin{align}\label{fairpartition}
\BP(\text{$\lambda_d(C^\tau(x))=\gamma^{-1}$ for all $x\in\Psi$})=1.
\end{align}
Taking $Z(x)$, $x\in\Psi$, as in Example \ref{ex:5.2}, it then follows that the point process
\begin{align}\label{eGamma2}
\Gamma:=\sum_{x\in\Psi}\delta_{Z(x)}
\end{align}
is hyperuniform. By \eqref{SGamma} the structure factor is given by
\begin{align}
S_\Gamma(k) &= \gamma\big(\gamma^{-2} -\BE_0^\Psi\big[\big|\widehat{\I_{C^\tau(0)}}(k)\big|^2\big]\big) \nonumber\\
&= \gamma\big(\gamma^{-2} -\BE\big[\big|\widehat{\I_{V_0}}(k)\big|^2\big]\big),\quad k\in\R^d,
\end{align}
which for $\gamma=1$ further simplifies to
\begin{align}
S_\Gamma(k) &= 1 -\BE_0^\Psi\big[\big|\widehat{\I_{C^\tau(0)}}(k)\big|^2\big] \nonumber\\
&= 1 -\BE\big[\big|\widehat{\I_{V_0}}(k)\big|^2\big],\quad k\in\R^d. \label{e:SF hyperuniformerer}
\end{align}

Fair partitions were constructed in the seminal papers
\cite{HoHolPe06, HolroydPeres05} based on a spatial version of the
Gale--Shapley algorithm. They exist if $\Psi$ is ergodic; see also
\cite[Corollary 10.10]{LastPenrose17}. If $\Psi$ is a Poisson process and $d\ge 3$,
then the gravitational allocation from \cite{CPPeresR10} is
a fair partition with much better moment properties. Both, the spatial Gale--Shapley algorithm, and the gravitational allocation are isometry-covariant. Therefore, if $\Psi$ is isotropic, these will lead to an isotropic $\Gamma$ by Remark \ref{r:isotropic splitting}.

\end{example}

Example \ref{ex:hyperuniformerer} can easily be generalized as follows.

\begin{example}\label{ex:5.5}\rm Let $(\Psi,\tau)$ be a fair partition
as in Example \ref{ex:hyperuniformerer}, and define the random transport $T$ by \eqref{eTTT}.
Fix $m \in \N$. Suppose that for each $x\in\Psi$, we have
$m$  random vectors  $Z_1(x),\ldots,Z_m(x)$ whose conditional distribution
given $T$ has the following two properties. First,  $Z_1(x),\ldots, Z_m(x)$ are
independent and uniformly distributed on $C(x)$. 
Second, for different $x\in\Psi$, the
random elements $(Z_1(x),\ldots, Z_m(x))$ are independent.
Then the random measure
\begin{align}\label{eGamma3}
\Gamma:=\sum_{x\in\Psi}\big(\delta_{Z_1(x)}+\cdots +\delta_{Z_m(x)}\big)
\end{align}
is hyperuniform. Indeed, we can apply Theorem \ref{t:splitting2} with 
\begin{align*}
F(\pi):=\frac{1}{m}\int \I\{\delta_{z_1}+\cdots +\delta_{z_m}\in\cdot\}\,\pi^m(\md (z_1,\ldots,z_m)),
\end{align*}
showing that $m^{-1}\Gamma$, and hence also $\Gamma$, are hyperuniform.
\end{example}

Our next example exhibits a hyperuniform, purely discrete random measure,
whose atoms are everywhere dense.

\begin{example}\label{ex:5.6}\rm A {\em Dirichlet process} with directing probability measure $\pi\in\bM^1(\R^d)$
is a random probability measure $\zeta$ on $\R^d$, such that $(\zeta(B_1),\ldots,\zeta(B_m))$
has a Dirichlet distribution with parameters $\pi(B_1),\ldots,\pi(B_m)$, whenever $B_1,\ldots,B_m$ is a
measurable partition of $\R^d$; see e.g.\ \cite[Exercise 15.1]{LastPenrose17}. Let $F(\pi)$ denote the
distribution of $\zeta$. The resulting mapping $F$ satisfies \eqref{F1}. By the Poisson construction
of $\zeta$, it does also satisfy \eqref{F2}.
Now, consider a fair partition as in Example \ref{ex:hyperuniformerer}.
Let $\{\zeta(x):x\in\Psi\}$ be a family of conditionally  independent (given $T$) random measures, such that
the conditional distribution of $\zeta(x)$ is that of a Dirichlet process with directing measure
$L(x)$. Then
\begin{align*}
\Gamma:=\sum_{x\in\Psi}\zeta(x)
\end{align*}
is a hyperuniform random measure. 
\end{example}

\begin{figure}[t]
  \centering
  \includegraphics[width=\textwidth]{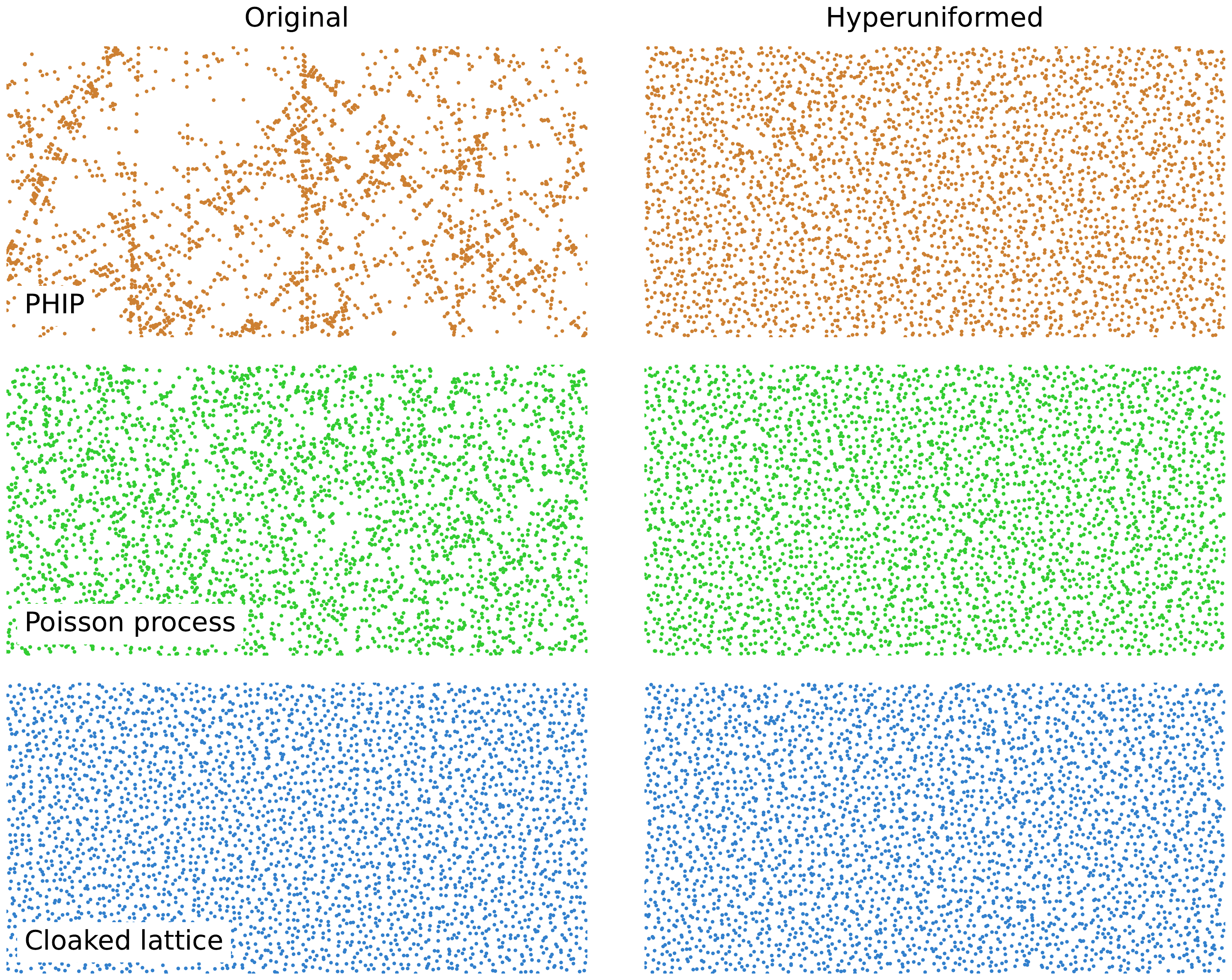}
  \caption{Exemplary applications of the hyperuniformerer: (Left) samples
    of point processes with asymptotic variances that are~---~from top to
    bottom~---~infinite, positive, or zero;
    (Right) the corresponding results of the hyperuniformerer.}
  \label{fig:hyperuniformerer_samples}
\end{figure}

\begin{figure}[t]
  \centering
 \includegraphics[width=\textwidth]{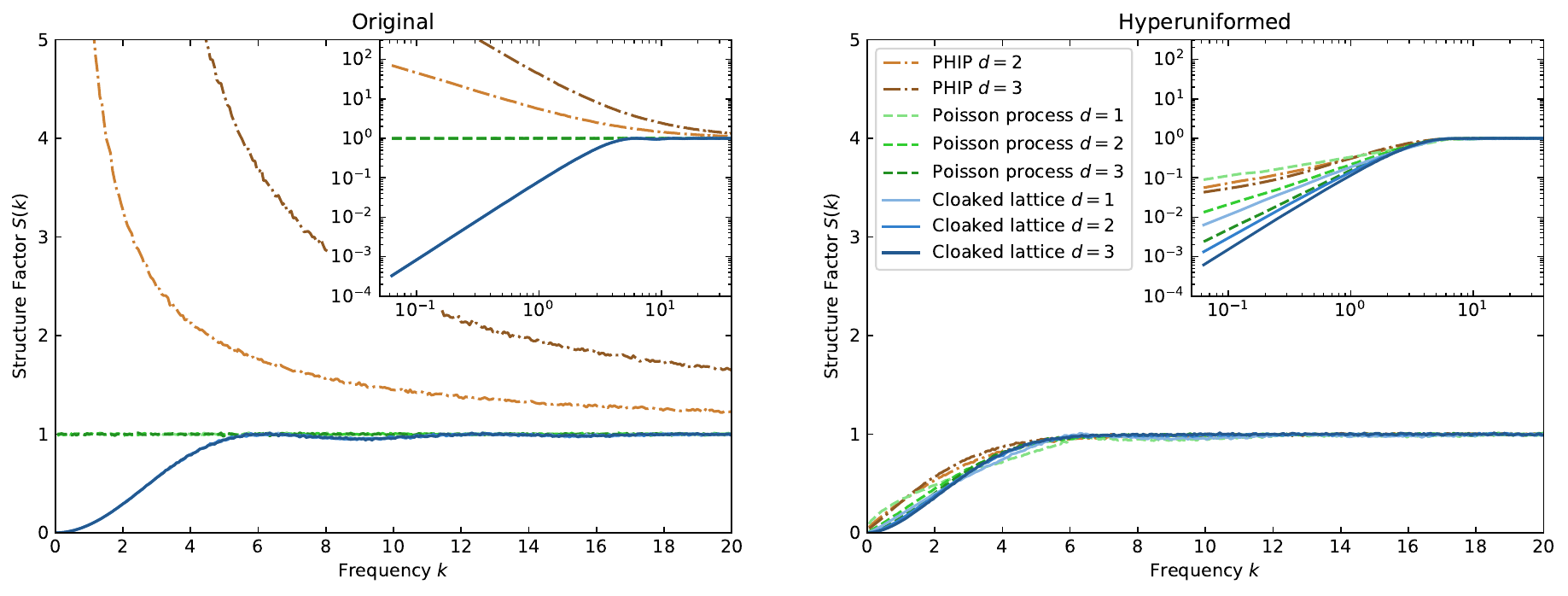}
  \caption{Structure factors before and after the application of the
    hyperuniformerer for the three models from
    Fig.~\ref{fig:hyperuniformerer_samples} across the first three
    dimensions. The insets show the same data in log-log plots.}
  \label{fig:hyperuniformerer_structure_factors}
\end{figure}

In both preceding examples, $F$ is isometry-covariant, and therefore an isotropic fair partition $(\Psi, \tau)$ will lead to an isotropic $\Gamma$ by Remark \ref{r:isotropic splitting}.
Next we formulate discretized versions of 
Examples  \ref{ex:5.2} and \ref{ex:hyperuniformerer}.

\begin{example}\label{ex:discrHyp}\rm Let $(\Psi,\tau)$ be an invariant partition
and let $\Phi$ be a stationary lattice, 
with intensity $k\gamma$ for some $k\in\N$, assume to be
invariant w.r.t.\ the underlying flow.
Assume that  $\tau(x)\in\Psi$ for all $x\in\Phi$ and that 
\begin{align*}
\BP(\text{$0<\Phi(C^\tau(x))<\infty$ for all $x\in\Psi$})=1.
\end{align*}
Just as at \eqref{eTTT} we define a random transport $T$ by
\begin{align*}
T:=\iint \I\{(x,y)\in\cdot\}\I\{y\in C^\tau(x)\}\,\Phi(\md y)\,\Psi(\md x).
\end{align*} 
Instead of \eqref{e:9333} we now have
\begin{align*}
T(\cdot\times\R^d)=\sum_{x\in\Psi}\Phi(C^\tau(x))\delta_x,\quad
T(\R^d\times \cdot)=\Phi.
\end{align*}
Define the (random) probability kernel $L$ by \eqref{e:9222}
with $\Phi$ instead of $\lambda_d$ and let 
$Z(x)$, $x\in\Psi$,  be as in \eqref{e:9435}.
Assume that \eqref{zerocell} holds with $\Phi$ in place of $\lambda_d$.
Then the random measure
\begin{align*}
\Gamma:=\sum_{x\in\Psi}\I\{0<\Phi(C^\tau(x))<\infty\}\Phi(C^\tau(x))\delta_{Z(x)}
\end{align*}
is hyperuniform. Indeed, in order to apply Theorem \ref{t:splitting}
as in Example \ref{ex:5.2}, we only need to replace the hyperuniform Lebesgue measure
by the hyperuniform point process $\Phi$.

An interesting special case arises if
\begin{align}\label{discretebalance}
\BP(\Phi(C^\tau(x)) = k \textnormal{ for all } x\in\Psi) = 1.
\end{align}
Then $T(\cdot\times\R^d)=k\Psi$ and   $\Gamma=k\sum_{x\in\Psi}\delta_{Z(x)}$.
Allocations with the balancing property \eqref{discretebalance} can be constructed
by a suitable version of the spatial Gale–Shapley algorithm from \cite{HoHolPe06, HolroydPeres05}.
In fact, it can be expected, that then the partition $\{C^\tau(x):x\in\Psi\}$ approximates,
as $k\to\infty$,  the 
stable allocation in continuum. 

We have exploited this idea in a simulation study, where we have applied
the hyperuniformerer to three exemplary point processes, exemplary
in the sense that their asymptotic variances are infinite, positive, or
zero. These examples are the anti-hyperuniform Poisson hyperplane
intersection process (PHIP) with an isotropic orientation
distribution~\cite{HugSchn24, K19}, the Poisson point
process~\cite{LastPenrose17},  and the hyperuniform cloaked
lattice~\cite{KKT20}. We simulated the first model in two and three
dimensions and the second and third in one, two, and three dimensions.
For each of the corresponding models, we simulated $30$, $15$,
and $10$ samples with an average number of points of $\bar{N}=100$,
$\bar{N}=10{,}000$, and $\bar{N}=1{,}000{,}000$ in one, two, and three
dimensions respectively.

We then implemented the hyperuniformerer using a variant of the spatial
stable allocation from~\cite{HoHolPe06, HPPS09} under periodic boundary
conditions. We approximated the Lebesgue measure by a lattice of
resolution $10{,}000$, $1{,}000$, and $500$, which corresponds to
an average of $100$ sites, $100$ pixels, and $125$ voxels per
sample point, in one, two, and three dimensions respectively.
For each sample of the original point process, we simulated
$1{,}000$ realizations of the hyperuniformed point process because
\eqref{e:SF hyperuniformerer} shows that the structure factor of the
hyperuniformed point process depends only on the distribution of a
single cell, of which we have an average of $\bar{N}$ in every sample.
Note that the number of points per sample fluctuates for the PHIP and
Poisson process. If the number of points is larger than the intensity,
not all points can be saturated during the matching and the possibility
arises that they are moved outside of the observation window by the
hyperuniformerer. If the number of points is smaller, some sites,
pixels, or voxels of the lattice remain unmatched after all points have
been saturated during the matching. Then we assume that these are all
matched to different points outside of the observation window which
introduces the maximal number fluctuation that is theoretically possible
from what we cannot observe. We will publish open-source code for our
hyperuniformerer together with the paper.
Figure~\ref{fig:hyperuniformerer_samples} shows portions of the final
samples for each of the three examples. We then estimate the structure
factor of the final point process via the scattering intensity for wave
vectors that are commensurable with the observation window; see
\cite{KlattLastHenze} for details.
Figure~\ref{fig:hyperuniformerer_structure_factors} shows the estimates
of the structure factors before and after the application of the
hyperuniformerer.

\end{example}

If $d\in\{1,2\}$ and under a moment assumption  a fair partition 
is necessarily associated with a hyperuniform process:

\begin{remark}\label{r:d12}\rm
Let $d\in\{1,2\}$, and suppose that $(\Psi,\tau)$ is an invariant fair partition. Assume that
\begin{equation}\label{e:551}
    \BE_0^\Psi\bigg[\int_{C^\tau(0)} \|x\|^d \, \md x\bigg] <\infty.
\end{equation}
It then follows from \cite{DFHL24} that  $\Psi$ is hyperuniform. 
Indeed, one can easily see that condition \eqref{e:551} implies that $\Psi$
is at finite {\em $d$-Wasserstein distance} (see \cite{DFHL24} for the definition) 
to Lebesgue measure and hence also to the stationary lattice.
Hence the asserted hyperuniformity is a consequence
of \cite[Remark 1.1]{DFHL24} and its preceding theorem.
Remarkably, this conclusion is wrong if $d\ge 3$. Proposition 4.3 in \cite{DFHL24}
provides an example of a fair partition with even uniformly bounded cells,
where $\Psi$ is not hyperuniform.
\end{remark}

We continue with an example of a fair partition where
\eqref{e:551} fails. 

\begin{example}\label{e:non-ergodic}\rm
Let $d=2$, and suppose that $X$ is a random variable, that is positive almost surely. Then define
\begin{equation}
    \Psi := \sum_{(z_1,z_2)\in\Z^2} \delta_{(Xz_1, X^{-1}z_2) + U},
\end{equation}
where given $X$, $U$ is 
uniformly distributed on  
$[-X/2, X/2)\times [-X^{-1}/2, X^{-1}/2)$. 
Thus, $\Psi$ is a stationary lattice, that is randomly stretched in one
direction by $X$ and in the other by $X^{-1}$, preserving the
area of the unit cell. Because $X$ is random, $\Psi$ is not ergodic.
Moreover, if $X$ or $X^{-1}$ has an infinite first moment,
$\Psi$ is not locally square integrable and in particular not
hyperuniform. To confirm this claim, it suffices to check that the Palm
version $\Psi_0$ of $\Psi$ is not locally integrable.
But $\Psi_0$ is just a non-stationary randomly stretched lattice.
For $\varepsilon>0$, we have
\begin{equation}
    \BE[\Psi_0(B_\varepsilon)] \geq \BE[2\varepsilon(X+X^{-1})-3] = \infty,
\end{equation}
which establishes the assertion. 
If $X$ and $X^{-1}$ have a finite first moment, but $X^{1+\varepsilon}$ or $X^{-(1+\varepsilon)}$ does not for some $\varepsilon\in(0,1)$, then, with some further calculation, one can show that $\Psi$ is still not hyperuniform even though it is locally square-integrable in this case.
Still it is possible to apply our
hyperuniformerer from Example \ref{ex:hyperuniformerer}
to construct a hyperuniform point process $\Gamma$, 
by moving the lattice points
in a suitable way.

This transport from $\Psi$ to $\Gamma$ does not fulfil the mixing condition from Theorem \ref{maintheoremmixing2}.
However, the hyperuniformerer is derived from
Theorem \ref{t:splitting}, which, in turn, under the additional
condition that the source $\Phi$ (here $\Psi$) is locally square-integrable, can be derived
from Theorem \ref{maintheoremmixing}. This shows that
Theorem \ref{maintheoremmixing} is more general than Theorem
\ref{maintheoremmixing2}.
\end{example}

\section{Displacements independent of the source}
\label{s:indepdispl}

In this section, we consider a square-integrable stationary random measure $\Phi$
with intensity $\gamma_\Phi$ along with a stationary
$\R^d$-valued random field $Z=\{Z(x):x\in\R^d\}$.  We assume that
$\Phi$ and $Z$ are independent. We focus on general stationary
random measures in Section \ref{ss:indepdisp_gen} and then consider
the (essentially) special case when $\Phi$ is a stationarized lattice in Section
\ref{ss:indepdisp_latt}.
Specialising our earlier results, we first show in Theorem \ref{t:displacement} that the random measure 
\begin{align}\label{e:displace}
\Gamma:=\int\I\{x+Z(x)\in\cdot\}\,\Phi(\md x)
\end{align}
has the same asymptotic variance as $\Phi$ under a mixing assumption
on $Z$. The same applies to the stationarized lattice $\Phi$ with the
definition of $Z$ and $\Gamma$ suitably modified; see Theorem
\ref{t:displacedlattice}. As in previous sections, we also present
Fourier versions of these theorems. 
In Subsection \ref{s:indepdisplkern}, we shall generalize the setting 
and replace the field $Z$ by a stationary family $K^*=\{K^*(x):x\in\R^d\}$
of random probability measures on $\R^d$.

Independent stationary displacements of stationary point processes were discussed
in the seminal work \cite{G04}. The forthcoming Theorems \ref{t:displacement2} and \ref{c:6.24} can be found there; see also \cite{KKT20}.  The paper \cite{G04} does also contain a discussion of spatially correlated displacement fields.  The results of this section are new in this generality.

\subsection{A general source}\label{ss:indepdisp_gen}

To make \eqref{e:displace} and our assumptions on $Z$ meaningful,
we need to impose some technical assumptions on $Z$.
Consider the Skorohod space $\bF\subset (\R^d)^{\R^d}$ of all c\`adl\`ag functions $\bw\colon\R^d\to\R^d$;
see e.g.\ \cite{Janson20}, where real-valued functions were considered.
For each $x\in \R^d$ we define the shift-operator 
$\theta_x\colon \bF\to \bF$ by $\theta_x\bw:=\bw(\cdot+s)$.
Equip $\bF$ with the smallest $\sigma$-field rendering the mappings
$\bw\mapsto \bw(x)$, $x\in\R^d$, measurable. 
Then even $(\bw,x)\mapsto \bw(x)$ is measurable, and
therefore also $(\bw,x)\mapsto \theta_x\bw$.
We assume that $Z$ is a random element of $\bF$ which is stationary, that
is $\theta_xZ\overset{d}{=}Z$, $x\in\R^d$.

\begin{theorem}\label{t:displacement} Let $\Phi$ be a stationary square-integrable random measure
and let $Z=\{Z(x):x\in\R^d\}$ be a stationary random element of $\bF$, independent of $\Phi$.
Assume that
\begin{align}\label{mixingfield}
\int \|\BP( \, (Z(y),Z(0))\in\cdot) - \BP(Z(0)\in\cdot)^{\otimes2}\|\alpha_\Phi(\md y) < \infty.
\end{align}
Let $W\in\cK_0$. If $\Phi$ is hyperuniform w.r.t.\ $W$, then so is $\Gamma$,
as defined by \eqref{e:displace}.
If either $W$ is Fourier smooth or $|\beta_\Phi|(\R^d)<\infty$, then
$\Phi$ and $\Gamma$  have the same asymptotic variance w.r.t.\ $W$,
provided one of these asymptotic variances exists. 
\end{theorem}
\begin{proof} It is no restriction of generality to assume that $\Phi$ and $Z$ are
defined on our basic probability space $(\Omega,\mathcal{A},\BP)$, equipped with
a flow $\{\theta_x:x\in\R^d\}$. Indeed, we can work with the product space
$\Omega:=\bM(\R^d)\times\bF$, equipped with the product $\sigma$-field and
the (obviously defined) product flow. Then, the probability measure is the product of the
distributions of $\Phi$ and $Z$. Therefore, we could redefine
$\Phi$ and $Z$ as the projections onto the first (resp.\ second) coordinate.
Hence, by definition of the shift on $\bF$, the mapping $x\mapsto \tau(x):=x+Z(x)$ is an invariant allocation.

The assertion follows from Corollary \ref{c:alloc} once we establish the following two claims:
\begin{align}\label{e:7.79}
\BP^\Phi_{0,y}( \, (\tau(y)-y, \tau(0))\in\cdot)=\BP( \, (Z(y),Z(0))\in\cdot),\quad \alpha_\Phi\text{-a.e. }y\in\R^d.
\end{align}
and
\begin{align}\label{e:7.80}
\BP^\Phi_{0}(\tau(0)\in\cdot)=\BP(Z(0)\in\cdot).
\end{align}
Let $f\colon \R^d\to[0,\infty)$ and $g\colon \R^d\times\R^d\to[0,\infty)$ 
be measurable functions. Then we have
\begin{align*}
\int f(y)& \BE^\Phi_{0,y}  g(\tau(y)-y,\tau(0))\,\alpha_\Phi(\md y)\\
&=\iint \I\{x\in[0,1]^d\}f(y) \BE^\Phi_{0,y}   g(\tau(y)-y, \tau(0))\,\alpha_\Phi(\md y)\,\md x\\
&=\BE \int \I\{x\in[0,1]^d\}f(y-x) g(\tau(\theta_x,y-x)-(y-x), \tau(\theta_x,0))\,\Phi^2(\md (x,y))\\
&=\BE \int \I\{x\in[0,1]^d\}f(y-x) g(\tau(y)-y, \tau(x)-x)\,\Phi^2(\md (x,y)),
\end{align*}
where we have used \eqref{2pointPalm2} to obtain the second identity.
By definition of the allocation $\tau$ and the independence 
of $\Phi$ and $Z$, the above equals
\begin{align*}
\BE \iint &\I\{x\in[0,1]^d\}f(y-x) g(\bw(y),\bw(x))\,\Phi^2(\md (x,y))\,\BP(Z\in \md \bw)\\
&=\BE \iint \I\{x\in[0,1]^d\}f(y-x) g(\bw(y-x),\bw(0))\,\Phi^2(\md (x,y))\,\BP(Z\in \md \bw),
\end{align*}
where we have used stationarity of $Z$. By \eqref{e2.55}, this equals
\begin{align*}
\iint f(y)g(\bw(y),\bw(0))\,\alpha_\Phi(\md y)\,\BP(Z\in \md \bw),
\end{align*}
and so \eqref{e:7.79} follows. A similar (even simpler) calculation shows \eqref{e:7.80} and thereby completing the proof.
\end{proof}

An important class of displacement fields are the so called Gaussian displacement fields. 
For these, the mixing condition \eqref{mixingfield} simplifies to a condition on correlations as seen in the next example.

\begin{example}(Gaussian displacement fields)\label{ex:gdf}\rm 
\label{ex:gaussiandisplacements}
Let $\Phi$ be a stationary random measure with
finite intensity, and suppose that $Z=\{Z(x):x\in\R^d\}$ is a
stationary $\R^d$-valued Gaussian random field with c\`adl\`ag-paths, i.e., $Z$ is a random
element in $\bF$, and for $n\in\N, x_1,...,x_n\in\R^d$,
$(Z(x_1), \dots, Z(x_n))$ follows a multivariate normal
distribution. Further, assume that $\Phi$ and $Z$ are independent, and
that
    \begin{equation}\label{mixinggaussianfield}
        \int \|\BC(Z(y), Z(0))\|\alpha_\Phi(\md y) < \infty,
    \end{equation}
    where the choice of the norm is arbitrary.
    Let $W\in\cK_0$. If $\Phi$ is hyperuniform w.r.t.\ $W$, then so is $\Gamma$,
    as defined by \eqref{e:displace}.
    If either $W$ is Fourier smooth or $|\beta_\Phi|(\R^d)<\infty$, then
    $\Phi$ and $\Gamma$  have the same asymptotic w.r.t.\ $W$,
    provided one of these asymptotic variances exists. 

    The assertion follows from Theorem \ref{t:displacement} using the fact that by Lemma \ref{l:gauss beta cov bound} there exists a $c>0$ such that
    \begin{equation*}
        \|\BP((Z(y),Z(0))\in\cdot) - \BP(Z(0)\in\cdot)^{\otimes2}\| \leq c \|\BC(Z(y), Z(0))\|, \quad y\in\R^d.
    \end{equation*}
\end{example}

We continue with the Fourier version of the preceding theorem.

\begin{theorem}\label{t:displaceFourier} Let the assumptions of Theorem \ref{t:displacement} 
be satisfied. Then 
\begin{align}\label{e:6.03}
\hat\beta_\Gamma=\alpha_\Phi(\{0\})(1-|\hat\BQ|^2)\cdot\lambda_d+|\hat\BQ|^2\cdot\hat\beta_\Phi
+\chi_{\Phi,Z}\cdot\lambda_d,
\end{align}
where $\BQ:=\BP(Z(0)\in\cdot)$ and the mapping $\chi_{\Phi,Z}\colon\R^d\to\R$ is defined by 
\begin{align}\label{e:6.04}
\chi_{\Phi,Z}(k):=\int\I\{y\ne 0\} e^{-i\langle k, y\rangle} \big(\BE\big[e^{-i\langle k,Z(y)-Z(0)\rangle}\big]-|\hat\BQ(k)|^2\big)\,\alpha_\Phi(\md y),\quad k\in\R^d.
\end{align}
\end{theorem}
\begin{proof} Let the assumptions of Theorem \ref{t:displacement} hold.
We apply Theorem \ref{maintheoremmixing2fourier} with $\Ks(y)=\delta_{\tau(y)}$, $y\in\R^d$,
where the allocation $\tau$ satisfies \eqref{e:7.79} and \eqref{e:7.80}. It follows that 
\begin{align*}
\BE_0^\Phi \big[\hat{K}^*_0]&=\hat\BQ,\\
\BE_{0,y}^\Phi \Big[\hat{K}^*_y(k)\overline{\hat{K}^*_0(k)}\Big]
&=\BE \big[e^{-i\langle k, Z(y) \rangle} e^{i\langle k, Z(0) \rangle}\big],\quad \alpha_\Phi\text{-a.e.\ $y\in\R^d$}.
\end{align*}
For $y=0$ the preceding expression equals $1$.
Therefore \eqref{e:6.03} follows from \eqref{bartletmeasuredifference2}.
\end{proof}
Note that equation \eqref{e:6.04} can be rewritten as
\begin{equation*}
    \chi_{\Phi,Z}(k)=\int\I\{y\ne 0\} e^{-i\langle k, y\rangle}\BC\big[e^{-i\langle k,Z(y)\rangle}, e^{-i\langle k,Z(0)\rangle}\big]\,\alpha_\Phi(\md y),\quad k\in\R^d.
\end{equation*}

Assume now that the random measure $\Phi$ is purely discrete.
Then one often assumes the random vectors  $Z(x)$ to be conditionally
independent for different $x\in\Phi$ with a conditional distribution $\BQ$, say,
which is independent of $x$.
Since a c\`adl\`ag assumption on $Z$ might be
at odds with this independence, we treat this case by using independent marking as follows.
We can  represente $\Phi$ as
\begin{align}\label{e:4.926}
\Phi=\sum^\infty_{n=1}Y_n\delta_{X_n},
\end{align}
where $Y_1,Y_2,\ldots$ are non-negative random variables and
$X_1,X_2,\ldots$ are random vectors in $\R^d$ 
$\R^d$ such that $X_m\ne X_n$ whenever
$Y_m\ne 0$ or $Y_n\ne 0$; see e.g.\ \cite[Chapter 6]{LastPenrose17}.
Let $Z_1,Z_2,\ldots$ be independent $\R^d$-valued random variables with distribution $\BQ$,
independent of $\Phi$. We shall show that the random measure
\begin{align}\label{Gammadis}
\Gamma:=\sum^\infty_{n=1}Y_n\delta_{X_n+Z_n}
\end{align}
has the same asymptotic variance as $\Phi$. Informally, we might still
think of $\Gamma$ as given by \eqref{e:displace}, where $\Phi$ and $Z$ are
independent and the random variables $Z(x)$, $x\in\R^d$, are independent with distribution
$\BQ$. We  would need, however, a measurable version of $Z$ to make
sense of \eqref{e:displace}.

\begin{theorem}\label{t:displacement2} Suppose that $\Phi$ is a locally square-integrable 
purely discrete stationary random measure
and define the random measure $\Gamma$ by \eqref{Gammadis}.
Then the assertions of Theorem \ref{t:displacement} hold. 
Moreover, \eqref{e:6.03} holds with $\chi_{\Phi,Z}\equiv 0$.
\end{theorem}
\begin{proof}   Let $\bM_d^*$ be the measurable set of all $\psi\in\bM(\R^d\times\R^d)$ such
that $\psi(\cdot\times\R^d)\in\bM_d$. Define a flow $\{\theta^*_x:x\in\R^d\}$
on this space by setting 
$\theta^*_x\psi:=\int \I\{(y-x,z)\in\cdot\}\,\psi(\md (y,z))$, for $\psi\in\bM_d^*$ and $x\in\R^d$. 
Define a random element $\Psi$ of $\bM_d^*$ by
\begin{align*}
\Psi:=\sum^\infty_{n=1}Y_n\delta_{(X_n,Z_n)}
\end{align*}
Proceeding as in the proof of \cite[Proposition 5.4]{LastPenrose17}), for instance, it
is not hard to show that $\Psi$ is stationary w.r.t.\ the flow $\{\theta^*_x:x\in\R^d\}$.
Since $\Phi=\Psi(\cdot\times\R^d)$ and $\Gamma$ is a function of $\Psi$,
it is no restriction of generality to work on 
the canonical space $\bM_d^*$ equipped with the distribution of $\Psi$
as probability measure. Define an allocation $\tau$ by setting
$\tau(\psi,x)=x+z$ if $(x,z)\in\psi$. Using the proof of Theorem \ref{t:displacement}
we obtain for $\alpha_\Phi$-a.e.\ $y\in\R^d$ the intuitively obvious identity 
\begin{align*}
\BP^\Phi_{0,y}((\tau(y)-y, \tau(0))\in\cdot)=\I\{y\ne 0\}\BQ^{\otimes 2}+\I\{y=0\}\int\I\{(z,z)\in\cdot\}\,\BQ(\md z).
\end{align*}
Furthermore we have $\BP^\Phi_{0}(\tau(0)\in\cdot)=\BQ$.
Therefore $\kappa(y)$, as defined by \eqref{kappaspecial}, vanishes for
$\alpha_\Phi$-a.e.\ $y\ne 0$ and the first assertions follow from
Corollary \ref{c:alloc}. 

Formula \eqref{e:6.03} with $\chi_{\Phi,Z}\equiv 0$ follows from
Theorem \ref{maintheoremmixing2fourier} by taking there $K^*=\delta_\tau$ and
using the Palm identities mentioned above. 
\end{proof}

\subsection{The lattice as a source}
\label{ss:indepdisp_latt}

Next, we consider stationary displacements of the stationary
lattice. We avoid calling this a perturbed lattice, since it would
suggest (as elsewhere in the mathematical literature) that the
displacements are small. In fact, the displacements can be huge and
are not even assumed to have a finite moment. In view of this, it
might be a bit surprising that 
a displacement independent of the source cannot break hyperunifomity, provided it satisfies a mixing
assumption.  
A seminal work on the asymptotic number variance of a displaced lattices 
(with iid-displacements) is  \cite{Gacs75}.

\begin{theorem}\label{t:displacedlattice} Let $\Phi$ be the stationary lattice, i.e., $\Phi=\sum_{x\in\Z^d}\delta_{x+U}$, 
where $U$ is uniformly distributed on the unit cube.
Suppose that $\{Z(x):x\in\Z^d\}$ is a stationary family of $\R^d$-valued random vectors, independent of $U$. 
Assume that
\begin{align}\label{mixinglattice}
    \sum_{y\in\Z^d} \|\BP((Z(y),Z(0))\in\cdot) - \BP(Z(0)\in\cdot)^{\otimes2}\| < \infty.
\end{align}
Then
\begin{align*}
\Gamma:=\sum_{x\in\Z^d} \delta_{x+U+Z(x)}
\end{align*}
is a stationary point process and hyperuniform with respect to any Fourier smooth $W\in\cK_0$.
\end{theorem}
\begin{proof} 
Let $V$ be independent of $(U,Z)$ with the same distribution as $U$. Define
\begin{align*}
\tilde{Z}(x):=Z(\lfloor x+V\rfloor),\quad x\in\R^d.
\end{align*} 
The field $\tilde{Z}:=\{\tilde{Z}(x):x\in\R^d\}$ satisfies the general assumptions of 
Theorem \ref{t:displacement} and is independent of $\Phi$. Since
$y+V -\lfloor y+V\rfloor\overset{d}{=}V$ for each $y\in\R^d$, it is easy to
see that $\tilde{Z}$ is stationary. We wish to apply Theorem \ref{t:displacement}
with $\tilde{Z}$ in place of $Z$. To check the assumptions of that theorem, we 
note that $\alpha_\Phi=\sum_{x\in\Z^d} \delta_{x}$. Furthermore
$\tilde{Z}$ coincides on $\Z^d$ with $Z$.
Hence assumption \eqref{mixinglattice}  is the same as \eqref{mixingfield}.
To conclude the assertion it remains to note that
\begin{align*}
\tilde\Gamma:=\int\I\{x+\tilde{Z}(x)\in\cdot\}\,\Phi(\md x)=\sum_{x\in\Z^d} \delta_{x+U+Z(\lfloor x+V\rfloor)}
\end{align*}
has the same distribution as $\Gamma$. This easily follows from the independence of 
$U,V,Z$ and stationarity of $Z$.
\end{proof}

\begin{remark}\rm
An example in \cite{DFHL24} shows for $d\ge 3$ that
the displaced lattice $\Gamma$ in Theorem \ref{t:displacedlattice}
need not be hyperuniform without any further assumptions on the
displacement field. Even a (deterministically) arbitrarily small
displacement can break hyperunformity of the stationary lattice. On
the other hand, it has also been proved in \cite{DFHL24} that $\Gamma$
is hyperuniform in the case $d=1,2$, as soon as
$\BE \|Z(0)\|^d<\infty$. In fact it was shown in
\cite{LRY2024,HueLeble24,Butez24} that in dimension $d=2$ a
hyperuniform point process is a displaced lattice, provided some
integrability assumption holds.
\end{remark}

Once again, as in Example \ref{ex:gdf}, the mixing condition
\eqref{mixinglattice} simplifies to a condition on the correlations
if the displacement field is Gaussian:

\begin{example}\label{e:Gaussperturb}\rm (Gaussian displacements) Let $\Phi$ be the
  stationary lattice 
and suppose that
  $\{Z(x):x\in\Z^d\}$ is a stationary $\R^d$-valued Gaussian random
  field. Further, assume that $U$ and $Z$ are independent, and that
    \begin{equation}\label{mixinggrflattice}
        \sum_{y\in\Z^d}\|\BC(Z(y), Z(0))\| < \infty,
    \end{equation}
    where the choice of the norm is arbitrary.  Then $\Gamma$, as
    defined in Theorem \ref{t:displacedlattice}, 
is a stationary point process and hyperuniform with respect to any Fourier smooth $W\in\cK_0$.
    This assertion follows from Theorem \ref{t:displacedlattice} like
    Example \ref{ex:gdf} followed from Theorem \ref{t:displacement}.
\end{example}

\begin{remark}\label{r:Flimmel25}\rm Theorem \ref{t:displacedlattice} is closely
reated to Theorem 1 in the recent preprint \cite{Flimmel25}.
There hyperuniformity of the translated lattice was
proved for an $\alpha$-mixing field $Z$ under 
moment assumptions on $Z(0)$. The mixing assumption
is similar to \eqref{mixinglattice} and involves a fractional power
of the $\alpha$-mixing coefficient of the field. Instead
we are using no moment assumption and the $\beta$-mixing coefficient, but without a fractional power and only for 
two values of the field and not for the whole trajectories.
In the Gaussian case from Example \ref{e:Gaussperturb} 
our required  decay of correlations is significantly slower
than in \cite[Corollary 2]{Flimmel25}. There, a
decay with a power strictly larger than $2d$ is required, whereas
for condition \eqref{mixinggrflattice} a decay with a power strictly
larger than $d$ suffices. In the $m$-dependent case, 
we can get rid of the moment assumption in \cite[Corollary 1]{Flimmel25}.
\end{remark}

\begin{theorem} Let the assumptions of Theorem
\ref{t:displacedlattice} be satisfied. Then the assertions of Theorem \ref{t:displaceFourier}
hold with $\chi_{\Phi,Z}(k)$ given by
\begin{align*}
\chi_{\Phi,Z}(k)=\sum_{y\in\Z^d\setminus\{0\}} e^{-i\langle k, y\rangle} 
\big(\BE\big[e^{-i\langle k,Z(y)-Z(0)\rangle}\big]-|\hat\BQ(k)|^2\big),\quad k\in\R^d.
\end{align*}
\end{theorem}
\begin{proof} As in the proof of Theorem \ref{t:displacedlattice}
we can apply Theorem \ref{maintheoremmixing2fourier} with $\Ks(y)=\delta_{\tau(y)}$,
where the allocation $\tau$ satisfies
\eqref{e:7.79} and \eqref{e:7.80}. This concludes 
the proof.
\end{proof}

\begin{corollary}\label{c:6.24}
Let $\Phi$ and $\Gamma$ be given as in Theorem \ref{t:displacedlattice}. 
Assume that $Z(0)$ and $Z(x)$ are independent  for each $x\in \Z^d\setminus\{0\}$.
Then $\Gamma$ is hyperuniform. Furthermore,
\begin{align}\label{e:6.032}
\hat\beta_\Gamma=(1-|\hat\BQ|^2)\cdot\lambda_d+\sum_{k\in\Z^d\setminus \{0\}}|\hat\BQ(k)|^2\delta_k,
\end{align}
where $\BQ:=\BP(Z(0)\in\cdot)$.
\end{corollary}

\subsection{Displacement kernels independent of the source}
\label{s:indepdisplkern}

We still consider a locally square-integrable stationary random measure $\Phi$
with finite intensity $\gamma_\Phi$. But instead of translating $\Phi$
with a random field, we use a more general object, namely a 
family $K=\{K(x):x\in\R^d\}$ of random probability measures
on $\R^d$. This allows to split the mass of $\Phi$.
We assume that $\Phi$ and $K$ are independent
and that $\Ks$ is stationary and consider the random measure
\begin{align}\label{e:displacekernel}
\Gamma:=\iint\I\{y\in\cdot\}\,K(x,\md y)\,\Phi(\md x).
\end{align}
In the case $\Ks(x)=\delta_{Z(x)}$ (that is $K(x)=\delta_{x+Z(x)}$) this definition boils down
to \eqref{e:displace}. 
In the general case we can identify $K$ with a probability kernel
from $\Omega\times\R^d$ to $\R^d$ defined by
$K(\omega,x,\cdot):=K(\omega)(x)(\cdot)$.
In our previous notation this means that $\Gamma=K\Phi$.
For consistency we prefer the latter notation.

To make sense of the preceding assumptions we assume that
$K$ is a random element of the space  $\mathbf{K}_d$ 
to be defined as follows. Equipped with the weak topology, the space 
$\bM^1(\R^d)$ becomes Polish; see e.g.\ \cite[Lemma 4.5]{Kallenberg17}. Hence we can
define to $\mathbf{K}_d$ as the set of all {\em c\`adl\`ag functions} 
$N\colon\R^d\to \bM^1(\R^d)$; see again  \cite{Janson20} for the real-valued 
case. We equip this space with the smallest $\sigma$-field
making the mappings $N\mapsto N(x)$ measurable for each $x\in\R^d$.
Then even the mapping $(N,x)\mapsto N(x)$ is measurable.
Stationarity of $K^*$ then 
means that $\{K^*(x+y):x\in\R^d\}\overset{d}{=}K^*$ for each $y\in\R^d$.

\begin{theorem}\label{maintheoremindependent} Let  $\Phi$ and $K$ satisfy the preceding
assumptions. Assume that
\begin{align}\label{6187}        
\int \big\|\BE[\Ks_y\otimes \Ks_0]-\BE[\Ks_0]^{\otimes 2}\big\|\, \alpha_\Phi(\md y) < \infty.
\end{align}
Then the assertions of Theorem \ref{maintheoremmixing2} hold.
  \end{theorem}
\begin{proof} As in the proof of Theorem \ref{t:displacement}
it is no restriction of generality to assume that $\Phi$ and $Z$ are
defined on our basic probability space, equipped with
a flow. This could be achieved, for instance, with the product space
$\Omega:=\bM(\R^d)\times\mathbf{K}_d$. 
The flow on this space can be defined as $\theta_x((\varphi,N)):=(\theta_x\varphi,\theta_xN)$,
where $(\theta_xN)(y,B):=N(y+x,B+x)$ for $y\in\R^d$ and $B\in\cB^d$.
Just as in the cited proof one can then show that 
\begin{align*}
\BE^\Phi_{0,y}[\Ks_y\otimes \Ks_0]=\BE[\Ks_y\otimes \Ks_0],\quad \alpha_\Phi\text{-a.e.\ $y\in\R^d$},
\end{align*}
and $\BE^\Phi_{0}\Ks_0=\BE\Ks_0$.
Therefore the  assertions follow from Theorem \ref{maintheoremmixing2}.
\end{proof}

\begin{remark}\label{remshotnoise}\rm We might interpret $\{K(x):x\in\R^d\}$ as a 
(spatially dependent) {\em noise field} perturbing $\Phi$.
If $\Phi$ is diffuse, then
$\alpha_\Phi\{0\}=0$.
A very special case is $K(x,\cdot)=\int\I\{y\in\cdot\}\,f(y-x)\,\md y$, $x\in\R^d$,
where $f\colon\R^d\to[0,\infty]$ is measurable and satisfies
$\int f(y)\,\md y=1$. Then we have $K\Phi=\int\I\{y\in\cdot\} \xi_y\,\md y$,
where
\begin{align*}
\xi_y:=\int f(y-x)\,\Phi(\md x),\quad y\in\R^d,
\end{align*}
is known as the {\em shot-noise field} based on the kernel function $f$
and $\Phi$.
\end{remark}

We next formulate the kernel version
of Theorem \ref{t:displacement2} for a purely discrete $\Phi$.
Let $K_0,K_1,\ldots$ be a sequence of independent 
random elements of $\bM^1(\R^d)$, independent of $\Phi$ and all with the same distribution. 
Represent $\Phi$ as at \eqref{e:4.926} and define a random measure $\Gamma$ by 
\begin{align}\label{Gammakernel}
\Gamma=\sum_{n=1}^\infty Y_n\int \I\{X_n+z\in\cdot\}\,K_n(\md z).
\end{align} 

\begin{theorem}\label{t:displacementkernel} Suppose that $\Phi$ is a locally square-integrable 
purely discrete stationary random measure
and define the random measure $\Gamma$ by \eqref{Gammakernel}.
Then the assertions of Theorem \ref{t:displacement} hold.
Moreover, the Bartlett spectral measure of $\Gamma$ is given by
\begin{align*}
    \hat\beta_\Gamma = \big|\BE[\hat{K}_0]\big|^2\cdot \hat\beta_{\Phi} + \hat\eta\cdot\lambda_d,
\end{align*}
where
\begin{align*}
    \hat\eta(k) &= \alpha_\Phi\{0\} \big(\BE[|\hat{K}_0(k)|^2] - \big|\BE[\hat{K}_0(k)]\big|^2\big),\quad k\in\R^d.
\end{align*}
\end{theorem} 
\begin{proof}
The point process 
\begin{align*}
\Psi:=\sum^\infty_{n=1}Y_n\delta_{(X_n,K_n)}
\end{align*}
is stationary w.r.t.\ the flow, defined similarly as in the proof of
Theorem \ref{t:displacement2}. In particular, $\Gamma$ is stationary.
Then we can proceed as in the cited proof.
\end{proof}

\begin{remark}\rm Similar as in Remark \ref{remshotnoise}.
the random probability measure $K_n$, $n\in\N$, can be interpreted as a noise 
perturbing $X_n$.  
But this time the noise is uncorrelated in space. Note that $\hat\eta$ 
equals the variance of the Fourier transform of the \textit{typical} noise
multiplied by $\alpha_\Phi\{0\}$.
\end{remark}

It is clearly possible to formulate the kernel versions
of the results in Subsection \ref{ss:indepdisp_latt} on translated lattices. 
We leave this to the reader.

\section{Stopping sets and mixing of transport maps of point processes}
\label{s:mix_marked_pp}

We consider the setup of Theorem \ref{maintheoremmixing2} in the case,
where $\Phi$ is a simple point process. We give a general theorem to
verify the integrability assumption of the mixing coefficient
$\kappa$, given by \eqref{e:kappa2}. In many examples, the probability
kernel or transport is defined via an auxiliary point process $\Gamma$
and furthermore, if the  kernels are determined by so-called
stopping sets, this allows us to bound the mixing coefficients
of the transport in terms of the mixing coefficients of the underlying
point processes and the tail probabilities of the stopping sets. This
shall be the content of the main theorem of the section; see Theorem
\ref{tgen_mix_pert_pp}. The key tool in the proof of Theorem
\ref{tgen_mix_pert_pp} is a factorial moment expansion for a
functional of two point processes (Lemma \ref{l:FME2pp}) which is
stated and proved in Section \ref{s:deccorrtransmap}. Proof of
Proposition \ref{p:dec_corr} which is crucial for the proof of Theorem
\ref{tgen_mix_pert_pp} is deferred to Section \ref{s:mix_2pp}. Section
\ref{s:ex_pp_mix} gives examples of point processes satisfying the
assumptions in Theorem \ref{tgen_mix_pert_pp}. Examples illustrating
the use of Theorem \ref{tgen_mix_pert_pp} will be furnished in
Sections \ref{s:localalgo}, \ref{s:volumes} and \ref{s:hyprandset}. The general factorial moment 
expansion result is stated and proved in Appendix \ref{subFME}.

A function $\delta \colon[0,\infty) \to [0,\infty)$ is said to be 
{\em  fast-decreasing} if $\lim_{s \to \infty}s^m\delta(s) = 0$ for all
$m \in \N$ and it is said to be {\em exponentially fast-decreasing} if
$\limsup_{s \to \infty}s^{-b}\log \delta(s) < 0$ for some $b > 0$. The
exponent $b$ is implicit in the choice of a exponentially fast
decreasing function $\delta$ and will not always be mentioned
explicitly.

In this and the next section we will be mainly dealing with simple point processes. 
Therefore we find it convenient to write $\Phi\cap B$ to denote the
restriction $\Phi_B$ to a measurable set $B\subset\R^d$.
Sometimes we use notation $\mu\cap B:=\mu_B$
even for $\mu\in\bN$ which are not simple.

\subsection{Decay of correlations and mixing of transport maps}
\label{s:deccorrtransmap}
We shall first introduce the notion of decay of correlations of point
processes and stopping sets. 
The following definition is similar to
the definition of weak exponential decay of correlations introduced
in \cite{Malyshev75,Martin80} but 
our presentation shall borrow from \cite{Nazarov12,BYY19}. 
Assume that $\Phi$ is a simple point process that has correlation functions of all orders.
(A point process having a second order correlation function must be simple.)
Recall the definition
of correlation functions from \eqref{e:lcorrfn}.
A point process $\Phi$ is said to have {\em fast decay of correlation functions} 
if there exists a fast decreasing function
$\delta \colon[0,\infty) \to [0,\infty)$ such that for all
$p,q \in \N, x_1,\ldots,x_{p+q} \in \R^d$, we have
\begin{align}\label{defn:admissible}
| \rho^{(p+q)}(x_1,\ldots ,x_{p+q}) - \rho^{(p)}(x_1,\ldots ,x_p)\rho^{(q)}(x_{p+1},\ldots ,x_{p+q}) | \leq C_{p+q}\delta(s),
\end{align}
where $s := d(\{x_1,\ldots,x_p\},\{x_{p+1},\ldots,x_{p+q}\})$ and
$C_n$, $n \in \N$, are finite constants. We term $\delta$ as the 
{\em  decay function}.  We also assume without loss of
generality that $C_n$ are non-decreasing in $n$.
We say $\Phi$ has {\em exponentially fast decay of correlations} if it
has fast decay of correlations with a decay function $\delta$ which
is exponentially fast decreasing.
We say that $\Phi$ has {\em finite-range dependence} if
$\delta$ is compactly supported.

Let $\mathcal{F}$ be the space of all closed sets equipped with Fell topology
and the corresponding Borel $\sigma$-algebra; see e.g.\ \cite{Kallenberg17,LastPenrose17}. 
A measurable 
function $S\colon \bN \to \mathcal{F}$, is said to be a {\em stopping set} if for all
compact sets $B\subset\R^d$
\begin{align}\label{stopping1}
\{\mu \in \bN: S(\mu) \subset B \} = \{\mu \in \bN : S(\mu_B) \subset B \}.
\end{align}
In other words, whether $S(\mu)$ is contained in $B$ or not is determined
by the restriction of $\mu$ to $B$. 
A convenient way to check
\eqref{stopping1} for a measurable mapping $S\colon \bN \to \mathcal{F}$ is 
\begin{align}\label{stopping2}
S(\mu)=S(\mu_{S(\mu)}+ \mu'_{S(\mu)^c}),\quad \mu,\mu'\in\bN;
\end{align}
see also \cite{LPY23}.
Indeed, by a suitable choice of $\mu'$ it can easily be shown that
\eqref{stopping2} implies \eqref{stopping1}. (The converse is true as well,
but not needed here.) The notion of stopping sets can be straightforwardly extended also to functions
$S \colon \bN \times \bN \to \mathcal{F}$ i.e., set-valued functions of two point
processes. As will be seen soon, there are many interesting examples
of stopping sets apart from deterministic sets; see Sections
Sections \ref{s:localalgo}, \ref{s:volumes} and \ref{s:hyprandset}.

We shall assume that the (random) probability kernel $K$ 
in Theorem \ref{maintheoremmixing2}
depends only on $\Phi$ and another independent
simple point processes $\Gamma$. Clearly we can assume
that $\Phi$ and $\Gamma$ are invariant point processes, defined on our basic probability
space $(\Omega,\mathcal{A},\BP)$ equipped with a flow.
(For instance $\Omega$ might be the product of $\bN\times\bN$ with another
space.) We shall assume that
$K$ is a {\em factor} of $(\Phi,\Gamma)$, meaning that
there is a invariant probability kernel $\tilde{K}$ from $\bN_s \times \bN_s \times \R^d$ to $\R^d$
such that
\begin{align}\label{factor}
K(x,\cdot)\equiv \tilde{K}(\Phi,\Gamma,x,\cdot),\quad x\in\R^d.
\end{align} 
In other words, $K$ depends only on $\Phi,\Gamma$ and no additional source of
randomness. As before, invariance of $\tilde K$ means that
$\tilde{K}(\theta_x\varphi,\theta_x\mu,0,\cdot) =\tilde{K}(\varphi,\mu,x,\cdot+x)$ for all
$(\varphi,\mu,x)\in\bN\times\bN\times\R^d$.

With this background, we now state our main theorem that helps us to verify 
mixing of various transport kernels in Sections \ref{s:localalgo}, \ref{s:volumes} and \ref{s:hyprandset}.

\begin{theorem}\label{tgen_mix_pert_pp}
Let $\Phi,\Gamma$ be independent stationary point processes with
non-zero intensity $\gamma$ such that they have exponentially fast
decay of correlations with the same decay function $\delta$ and
constants $C_k$, $k \in \N$, with $C_k = O(k^{ak})$ for some $a <1$. 
Let $\tilde K$ be an invariant probability kernel from
$\bN\times \bN \times \R^d$ to $\R^d$ and define the 
probability kernel $K$ from $\Omega\times\R^d$ to $\R^d$ by \eqref{factor}.
Assume that  there is a stopping set $S\colon \bN \times \bN \to \mathcal{F}$ such that
\begin{align}
\label{e:KdetS}
\tilde{K}(\varphi,\mu,0,\cdot) = \tilde{K}(\varphi \cap S(\varphi,\mu),\mu \cap S(\varphi,\mu),0,\cdot\cap S(\varphi,\mu)),
\quad \varphi,\mu\in\bN,
\end{align}
and
that there exists a decreasing function $\delta_1 \leq 1$ such that
\begin{align}\label{stopbounds}
\max \{ \BP^{\Phi}_{0}(S(\Phi,\Gamma) \not\subset B_t), 
\sup_{y \in \R^d}\BP^{\Phi}_{0,y}(S(\Phi,\Gamma) \not\subset B_t) \} \leq \delta_1(t).
\end{align}
Then the mixing coefficient in \eqref{e:kappa2} satisfies
\begin{equation}
\label{e:alphabdstopset}
\kappa(y) \leq \hC \big( \delta_1((\|y\|/8)^{\beta}) + \hvphi(\|y\|) \big),\quad y\in\R^d,
\end{equation}
where $\hC$ is a finite constant, $\hvphi$  is a fast-decreasing function and
$\beta < \frac{b(1-a)}{(d+2)}, \beta \leq 1$, where $b$ is the exponent of the decay function $\delta$.
Further, as a consequence we have that \eqref{maintheoremmixingcondition} holds 
(i.e., $\int \kappa(y) \alpha_{\Phi}(\md y) < \infty$) if 
\begin{equation}
\label{e:intvarphi1}
\int_1^{\infty}s^{\frac{d}{\beta}-1}\delta_1(s) \md s < \infty.
\end{equation}
Thus if \eqref{e:intvarphi1} holds, then the assertions (i) and (ii)
of Theorem \ref{maintheoremmixing2} hold.

Furthermore, if $\Phi,\Gamma$ are finite-range dependent point
  processes (i.e., there exists $r_0 < \infty$ such that
  $\delta(s) = 0$ for all $s \geq r_0$ with constants
  $C_k \leq k! c^k$ for some $c < \infty$), then we have that
\begin{equation}
\label{e:kappadelcompsupp}
\kappa(y) \leq 4\delta_1(\|y\|/3), \text{ for all } \|y\| \geq 3r_0.
\end{equation}
\end{theorem}
Even though we can have the decay functions and constants of
$\Phi,\Gamma$ different, it is clear that these can be combined to
yield a common decay function and constant. More explicit bounds on
$\hvphi$ in \eqref{e:alphabdstopset} can be deduced with some work
from the proof of upcoming Proposition \ref{p:dec_corr}, the key to
proving the theorem. Both the proposition and hence the above
theorem can possibly be extended to include more general kernels
that involve additional randomness, coming, for instance, from
marked point processes.

Unless stated otherwise, now we fix the kernel $\tilde{K}$ from Theorem \ref{tgen_mix_pert_pp}.
To rewrite the mixing coefficient $\kappa$ in \eqref{e:kappa2} we find it convenient
to introduce two simple point processes $\tPhi,\tGamma$ along with probability measures
$\BP^{(y)}$, $y\in\R^d$, on the underlying sample space $(\Omega,\mathcal{A})$
satisfying 
\begin{align*}
\BP^{(y)}((\Phi,\Gamma)\in A,(\tPhi,\tGamma)\in A')
=\BP^{\Phi}_{0}((\Phi,\Gamma)\in A)\BP^{\Phi}_{0}((\theta_{-y}\Phi,\theta_{-y}\Gamma)\in  A'),\quad y\in\R^d,
\end{align*}
for all measurable $A,A'\subset\bN\times\bN$. Hence the pairs
$(\Phi,\Gamma)$ and $(\tPhi,\tGamma)$ are independent under
$\BP^{(y)}$. Moreover, 
$\BP^{(y)}((\Phi,\Gamma)\in\cdot)=\BP^\Phi_0((\Phi,\Gamma)\in\cdot)$
and  $\BP^{(y)}((\tPhi,\tGamma)\in\cdot)=\BP^\Phi_0((\theta_{-y}\Phi,\theta_{-y}\Gamma)\in\cdot)$;
see also \eqref{multPalm}.
We note here that $\BP^\Phi_0((\Phi,\Gamma)\in\cdot)$ is the product
of the Palm distribution $\BP^\Phi_0(\Phi\in\cdot)$ of $\Phi$
and the distribution $\BP(\Gamma\in\cdot)$ of $\Gamma$.  Given a
bounded measurable function $f\colon\R^d\times\R^d \to [0,1]$ we
define $\tf\colon \R^d \times \bN^4 \to [0,1]$ as
\begin{align}\label{e:tfdef}
\tf(y,\varphi,\mu,\tphi,\tmu) := \iint f(x,z-y) \tilde{K}(\varphi,\mu,0,\md x) \tilde{K}(\tphi,\tmu, y,\md z).
\end{align} 
Then we can write the mixing coefficient $\kappa$ in \eqref{e:kappa2} as 
\begin{align}\label{e:kappaEdef}
\kappa(y) = 2\sup_{f } \Big| \BE^{\Phi}_{0,y}\Big[\tf(y,\Phi,\Gamma,\Phi,\Gamma) \Big] 
- \BE^{(y)} \Big[ \tf(y,\Phi,\Gamma,\tPhi,\tGamma) \Big] \Big|,
\end{align}
where the supremum is taken over all measurable functions $f\colon\R^d\times\R^d\to[0,1]$ 
and $\BE^{(y)}$ denotes expectation w.r.t.\ $\BP^{(y)}$. We again note that $\BP^\Phi_{0,y}((\Phi,\Gamma)\in\cdot)$ is the product measure $\BP^\Phi_{0,y}(\Phi\in\cdot) \otimes \BP(\Gamma\in\cdot)$. By $B(x,r)$, we denote the ball of radius $r \geq 0$ centred at $x \in \R^d$ and for convenience abbreviate $B(0,r)$ to $B_r$.

We shall state the proposition here but defer its proof to the next
section (Section \ref{s:mix_2pp}) as it requires some more
technicalities. A function $F\colon \R^d \times \bN^4 \to \R$ is said to be {\em local} if there exists $r_F \in [0,\infty)$ such that
\begin{equation}\label{e:zetalocal}
F(x,\varphi,\mu,\tphi,\tmu) = F(x,\varphi \cap B_{r_F},\mu \cap B_{r_F},\tphi \cap B(x,r_F),\tmu \cap B(x,r_F)).
\end{equation}

\begin{proposition}\label{p:dec_corr}
Let $\Phi,\Gamma$ be independent stationary point processes with
intensity $\gamma$ such that they have exponentially fast decay of
correlations with the same decay function $\delta$ and constants
$C_k$, $k \in \N$, with $C_k = O(k^{ak})$ for some $a < 1$. Let
$F\colon \R^d \times \bN^4 \to [0,1]$ be a measurable
translation invariant function such that
$F(x,\varphi,\mu,\tphi,\tmu) = 0$ if $0 \notin \varphi$ or
$x \notin \tphi$. Assume  that $F$ is local  as in \eqref{e:zetalocal}.
Then, for $\beta <\frac{b(1-a)}{(d+2)}, \beta \leq 1$, where $b$ is the exponent associated to the decay function $\delta$, and $y \in \R^d$ with $\|y\| \geq 8r_F^{1/\beta}$, we have that
\begin{align}
\Big| \BE^{\Phi}_{0,y} \big[ F ( y,\Phi,\Gamma,\Phi,\Gamma ) \big] \, \rho^{(2)}(0,y) - 
\BE^{(y)}\big[ F (y,\Phi,\Gamma,\tPhi,\tGamma )\big] \, \gamma^2 \Big| 
  & \leq \tC \tvphi(\|y\|),   \label{e:decorrxi}
  \end{align}
where $\tvphi$ is a fast-decreasing function and $\tC$ is a finite constant where both $\tilde\delta$ and $\tilde{C}$ can be chosen independently of $F$ and in particular $r_F$.
\end{proposition}
More precisely, $\tC$ depends only on $d,b,\beta,\gamma$ and $C_k$'s and $\tvphi$ depends only on the these parameters as well as $\delta$.
\begin{proof}[Proof of Theorem \ref{tgen_mix_pert_pp}]
We need to prove only \eqref{e:alphabdstopset} as the remaining
claim on integrability of $\kappa$ can be deduced from
\eqref{e:alphabdstopset}, the fast-decreasing nature of $\hvphi$, that
$\alpha_{\Phi}(\md y) = \rho^{(2)}(0,y)\md y + \gamma \delta_0$
with $\rho^{(2)}$ being bounded due to the fast decay of
correlations (see \eqref{e:corrbd}), and the trivial bound of
$\kappa(0) \le 2$.

Fix a measurable function $f\colon\R^d\times\R^d\to[0,1]$ 
and set $t:= \|y\|$. Assume without loss
of generality $t \geq 8$. Let
$\beta:= \min\{\frac{b(1-a)}{2(d+2)},1\}$ as in the theorem and
set $r:= (\frac{t}{8})^{\beta}$. We define
\begin{align*}
S(y,\varphi,\mu):=S(\theta_y\varphi, \theta_y\mu)+y,\quad (y,\varphi,\mu) \in \R^d \times \bN\times\bN.
\end{align*}

Consider the RHS of \eqref{e:kappaEdef}. Recalling $\tf$ as
defined in \eqref{e:tfdef}, we set 
$$ 
F(y,\varphi,\mu,\tphi,\tmu) := \tf(y,\varphi \cap B_r,\mu \cap B_r, \tphi \cap  B(y,r),\tmu \cap B(y,r)).
$$
We have that $F$ is a local functional with $r_F:=r$ as in
\eqref{e:zetalocal}. By our assumption on $f$ and
translation invariance of $\tilde K$, $F$ is also translation invariant and
$F \in [0,1]$ as needed in Proposition \ref{p:dec_corr}. Since the stopping set $S$ determines $\tilde K$ by assumption \ref{e:KdetS}, we have that if $S(0,\varphi,\mu) \subset B_r$ and $S(y,\tphi,\tmu) \subset B(y,r)$ then $F(y,\varphi,\mu,\tphi,\tmu) = \tf(y,\varphi,\mu,\tphi,\tmu)$. From this observation and since $\tf$ is bounded by $1$, we have that
\begin{align*}
  \BE^{\Phi}_{0,y}  &\Big[ \Big|\tf ( y,\Phi,\Gamma,\Phi,\Gamma ) - F ( y,\Phi,\Gamma,\Phi,\Gamma ) \Big| \Big] \\
& \le \BP^{\Phi}_{0,y}(S(0,\Phi,\Gamma) \not\subset B_r) + \BP^{\Phi}_{0,y}(S(y,\Phi,\Gamma) \not\subset B(y,r)) 
\le 2\delta_1(r),
\end{align*} 
where we have used assumption \eqref{stopbounds}.
Similarly, we obtain that
\begin{align*}
\BE^{(y)}\Big[ \Big| \tf ( y,\Phi,\Gamma,\tPhi,\tGamma ) - F ( y,\Phi,\Gamma,\tPhi,\tGamma ) \Big| \Big]
 &  \leq 2\delta_1(r).
 \end{align*}
Thus combining the above bounds with the triangle inequality we derive that 
\begin{align}
\big| \BE^{\Phi}_{0,y}& \big[\tf(y,\Phi,\Gamma,\Phi,\Gamma) \big] 
- \BE^{(y)}[ \tf(y,\Phi,\Gamma,\tPhi,\tGamma)] \big|  \nonumber \\
  \label{e:kappaybd1}  
&\leq 4\delta_1(r)  + \Big| \BE^{\Phi}_{0,y} 
\big[ F ( y,\Phi,\Gamma,\Phi,\Gamma ) \big]-\BE^{(y)}\big[ F (y,\Phi,\Gamma,\tPhi,\tGamma )\big] \big|. 
\end{align}
Now we bound the last term in the RHS of \eqref{e:kappaybd1} as follows. By a simple use of triangle inequality, $F \in [0,1]$, fast decay of correlations for $\Phi$ and \eqref{e:decorrxi} together yields that 
\begin{align}
  \big| &\BE^{\Phi}_{0,y} \big[ F ( y,\Phi,\Gamma,\Phi,\Gamma ) \big]- 
  \BE^{(y)}\big[ F (y,\Phi,\Gamma,\tPhi,\tGamma )\big] \big| \nonumber \\
&\le  \big| \BE^{\Phi}_{0,y}[ \cdots ] \big| \times \big |1 - \gamma^{-2} \rho^{(2)}(0,y) \big|
+ \gamma^{-2} \big| \BE^{\Phi}_{0,y}[ \cdots ] \rho^{(2)}(0,y)  - \BE^{(y)} [  \cdots ] \gamma^2 \big|  \label{e:EyF_bound} \\
& \le \gamma^{-2}\big( C_2\delta(\|y\|) + \tC\tvphi( \|y\|) \big). \nonumber 
\end{align}
Further $\tvphi$ above is same as the fast-decreasing function in  Proposition \ref{p:dec_corr}. 
Substituting the above bound in \eqref{e:kappaybd1}, we obtain that
\begin{equation}
\label{e:bdalpha2}
\kappa(y) \leq 4\delta_1((\|y\|/8)^{\beta}) + \gamma^{-2}( C_2\delta(\|y\|) + \tC\tvphi(\|y\|)).
\end{equation}
Since $\tvphi$ and $\delta$ are fast decreasing, so is $\hvphi := (2\gamma)^{-2}( C_2\delta + \tC\tvphi)$ and
thus the the proof of \eqref{e:alphabdstopset} is complete.

We now prove the claim \eqref{e:kappadelcompsupp}. Under the condition on the constants $C_k$, we have that for all bounded subsets $B$, there exists $a > 0$ such that $\BE[e^{a\Phi(B)}] < \infty$. Thus the correlation functions determine the distribution of the point process $\Phi$ (see e.g. \cite[Proposition 4.12]{LastPenrose17}). In particular, the restrictions of the correlation functions $\rho^{(p)}$ to $B$ determine the distribution of $\Phi \cap B$. Suppose that $A,B$ are bounded subsets that are at least $r_0$ apart i.e., $\inf_{x \in A, y \in B}\|x-y\| \geq r_0$. By assumption, we have that for all $p,q \geq 1$ and $x_1,\ldots,x_p \in A$, $x_{p+1},\ldots,x_{p+q} \in B$, it holds that
$$ \rho^{(p+q)}(x_1,\ldots,x_{p+q}) = \rho^{(p)}(x_1,\ldots,x_{p}) \rho^{(q)}(x_{p+1},\ldots,x_{p+q}).$$
Thus, we have that $\Phi \cap A$ and $\Phi \cap B$ are independent point processes. Similarly, we can argue about independence of $\Gamma \cap A$ and $\Gamma \cap B$.

Now we follow the above derivation upto \eqref{e:kappaybd1} by choosing $\|y\| \geq 3r_0$ and $r := \|y\|/3 \geq r_0$. Since $\Phi \cap B_r,\Gamma \cap B_r$ are independent of $\Phi \cap B(y,r),\Gamma \cap B(y,r)$ respectively, we have that the second term on the RHS of \eqref{e:kappaybd1} vanishes and hence we obtain \eqref{e:kappadelcompsupp}
\end{proof}

\subsection{Mixing of local functionals of point processes}
\label{s:mix_2pp}

The aim of this subsection is to prove the key asymptotic decorrelation
inequality \eqref{e:decorrxi}, used in the proof of
Theorem \ref{tgen_mix_pert_pp}.  Our proof will be via the factorial
moment expansion (FME) for functionals as in Theorem \ref{t:FME2pp} and we will borrow the terminology and notions from therein.

Given a function $F\colon \bN \times \bN \to \R$ 
we define $F^+\colon (\R^d)^2\times\bN^2 \to \R$ by 
\begin{align}\label{F+}
F^+(x,y,\varphi,\mu) = F(\varphi+\delta_x+\delta_y,\mu).
\end{align}
Similarly as in Proposition \ref{p:dec_corr} we say that the function $F$ is {\em local} if for some $r_F > 0$,  
\begin{align} 
\label{e:local}
F(\varphi,\mu) & = F(\varphi \cap B_{r_F},\mu \cap B_{r_F}),  \quad \varphi,\mu \in \bN.
\end{align}

In the next lemma we use a lower index $\Phi$ to indicate the
dependence of the correlation functions $\rho^{(n)}_\Phi$ on $\Phi$.
As in Subsection \ref{subhigherpoint} we denote by $\rho^{(l)}_{\Phi,y_1,\ldots,y_m}$ the $l$-th correlation function of the reduced Palm distributions $\BP^{!\Phi}_{y_1,\ldots,y_m}$ of $\Phi$. Recall that $o$ denotes the null measure, and also we shall use difference operators as introduced in Section \ref{subFME}.

\begin{lemma}[FME expansion for functionals of two independent point processes.]
\label{l:FME2pp}
Let $\Phi,\Gamma$ be two independent point processes having correlation
functions and with bounded intensity functions 
$\rho^{(1)}_\Phi,\rho^{(1)}_\Gamma$. Let
$\Phi, \Gamma$ satisfy exponentially fast decay of correlations as in
\eqref{defn:admissible} with the same decay function
$\delta$ and same constants $C_k$ such that $C_k = O(k^{ak})$ for
some $a < 1$. Let $F$ be a bounded local function.  Then we have that
for Lebesgue a.e.\  $y \in \R^d$ with $\rho^{(2)}_\Phi(0,y) > 0$,
\begin{align*}
\BE^{\Phi}_{0,y} \big[ F(\Phi,\Gamma) \big] & = 
\int  F(\varphi,\mu)\, \BP^{\Phi}_{0,y}((\Phi,\Gamma)\in \md(\varphi,\mu)) \\
& = \int  F^+(0,y,\varphi,\mu)\, \BP^{!\Phi}_{0,y}((\Phi,\Gamma)\in \md(\varphi,\mu)) \\
&  = \sum_{k=0}^{\infty}\sum_{l=0}^{\infty} \frac{1}{k!l!}
\int_{(\R^d)^l}\int_{ (\R^d)^k} D^{k,2}_{z_1,\ldots,z_k}[D^{l,1}_{y_1,\ldots,y_l}F^+(0,y,o,o)] \\
& \qquad \times \rho_{\Gamma}^{(k)}(z_1,\ldots,z_k) \, \rho_{\Phi,0,y}^{(l)}(y_1,\ldots,y_l) \, \md(z_1, \ldots,z_k)\, \md(y_1, \ldots,y_l),
\end{align*}
where $\rho_{\Phi,0,y}^{(l)}$ denotes the $l$th 
correlation function of $\Phi$ under the Palm distribution $\BP^{!\Phi}_{0,y}$ and $D^{\cdot,1}, D^{\cdot,2}$ denote difference operators applied to $\varphi$ and $\mu$ respectively in $F^+$.
\end{lemma}
\begin{proof}
Locality of $F$ trivially ensures that \eqref{continfty2} holds for
any point processes $\Phi,\Gamma$. So we only need to verify
\eqref{FMEass2} but the additional complication is that we have to
check it for $\Phi$ under $\BP_{0,y}$.

For $\rho \in\{ \rho_{\Phi}, \rho_{\Gamma}\}$ and $\alpha \in\{ \alpha_{\Phi},\alpha_{\Gamma}\}$, the following holds:
\begin{equation}\label{e:palmcorr}
\rho^{(l)}_{0,y}(y_1,\ldots,y_l) = \frac{\rho^{(l+2)}(0,y,y_1,\ldots,y_l)}{\rho^{(2)}(0,y)},
\quad \alpha^{(l+1)}\text{-a.e.\ $(y,y_1,\ldots,y_l)$},
\end{equation}
which follows from \eqref{corrPalm} and the translation invariance of the correlation
functions of $\Phi$ and $\Gamma$; see \eqref{corrinvariance}.
Now from the fast decay of correlations of $\Phi$ and $\Gamma$ we have that (see \cite[(1.12)]{BYY19})
\begin{equation}
\label{e:corrbd}
 \sup_{(y_1,\ldots,y_l) \in (\R^d)^{l}} \rho^{(l)}(y_1,\ldots,y_l) \leq lC_l \kappa_0^l,
\end{equation}
where $\kappa_0 := \sup_{y \in \R^d}\rho^{(1)}(y)$ and is bounded by
assumption. 

Suppose that $F$ is bounded by $M$, then by the recursive definition of the difference operators we have that for all $\varphi,\mu \in \bN$ 
and $l,k \in \N_0$,  
\begin{align*}
\big|D^{k,2}_{z_1,\ldots,z_k}[D^{l,1}_{y_1,\ldots,y_l}F^+(0,y,o,o)]\big| \leq M2^{l+k}.
\end{align*}
Furthermore,  by the locality of $F$,  we have that $D^{k,2}_{z_1,\ldots,z_k}[D^{l,1}_{y_1,\ldots,y_l}F^+(0,y,o,o)] = 0$ 
if for some $i$,  $y_i \notin B_{r_F}$ or $z_i \notin B_{r_F}$ (see \cite[(3.8)]{BYY19}).  Thus combining this observation and the 
above bounds\ with \eqref{e:palmcorr} and \eqref{e:corrbd},   
we have that
\begin{align*}
& \int_{\R^{d(l+k)}}\big| D^{k}_{z_1,\ldots,z_k}[D^{l}_{y_1,\ldots,y_l}F^+(0,y,o,o)] \big| \, \rho_{\Gamma}^{(k)}(z_1,\ldots,z_k) \, \rho_{\Phi,0,y}^{(l)}(y_1,\ldots,y_l) \\
& \qquad \qquad \times \md(z_1, \ldots,z_k)\, \md(y_1, \ldots,y_l) \\
& = \,  \int_{B_{r_F}^{(l+k)}}\big| D^{k}_{z_1,\ldots,z_k}[D^{l}_{y_1,\ldots,y_l}F^+(0,y,o,o)] \big| \, \rho_{\Gamma}^{(k)}(z_1,\ldots,z_k) \, \rho_{\Phi,0,y}^{(l)}(y_1,\ldots,y_l) \\
& \qquad \qquad \times \md(z_1, \ldots,z_k)\, \md(y_1, \ldots,y_l) \\
& \le \, \rho_{\Phi}(0,y)^{-1} \, M \, C_l \, C_k \, l \, k \, (2 \, \pi_d \, \kappa_0 \, r_F^d)^{l+k},
\end{align*}
where $\pi_d$ is the volume of the unit ball. Thus under our assumption on the correlation constants $C_l$ and positivity of $\rho_{\Phi}(0,y)$, we have that \eqref{FMEass2} holds. 
\end{proof}

Now we have all ingredients to prove Proposition \ref{p:dec_corr}. If
there was no dependence on $\Gamma$, we could directly apply
\cite[Theorem 1.11]{BYY19} to bound by $\tC \tvphi(s)$ where $\tvphi$
is another fast-decreasing function depending on
$\delta$. Neverthless, we can still use the techniques as in
\cite[Theorem 1.11]{BYY19} to bound the second term by $\tC \tvphi(s)$
and this is what we do in our proof.

\begin{proof}[Proof of Proposition \ref{p:dec_corr}]
Our proof strategy is to apply FME as given in Lemma
\ref{l:FME2pp} to the two expectations in the LHS of
\eqref{e:decorrxi}.  We first make a key observation that makes
bounding these expectations via FME tractable.

Fix $y \in \R^d$ and without loss of generality, $r :=r_F \geq 1$ 
and $t := \|y\| > 8r \ge \, 8r^{1/{\beta}}$. For $F$ as in the proposition, with a slight abuse of notation, set 
\begin{align*}
F^+(y,\varphi,\mu,\tphi,\tmu) = F( y, \varphi+\delta_0, \mu,\tphi+\delta_y,\tmu),
\end{align*}
hoping that this causes no confusion with \eqref{F+}. Since there are four point processes involved, we shall use the following notation for difference operators which is consistent with the terminology in Theorem \ref{t:FME2pp}. Denoting $\varphi_1 = \varphi, \varphi_2 = \tphi, \varphi_3 = \mu$ and $\varphi_4 = \tmu$, we use $D^{k,i}_{\cdots}$ for the difference operators applied to the $i$th counting measure $\varphi_i$ by fixing the other $\varphi_j, j \neq i$. Further, we can iterate them as follows: For $j,k,l,m \geq 0$
\begin{align*}
&    D^{j}_{z_1,\ldots,z_{j}} 
D^{k}_{z_{j+1},\ldots,z_{j+k}} \big[ D^{l}_{y_1,\ldots,y_{l}} 
D^{m}_{y_{l+1},\ldots,y_{l+m}}F^+(y,\varphi,\mu,\tphi,\tmu) \big] \\
& = D^{j,3}_{z_1,\ldots,z_{j}} 
D^{k,4}_{z_{j+1},\ldots,z_{j+k}}D^{l,1}_{y_1,\ldots,y_{l}} 
D^{m,2}_{y_{l+1},\ldots,y_{l+m}}F^+(y,\varphi,\mu,\tphi,\tmu).
\end{align*}
Though the order of iteration is not important, we shall stick to the
above convention to simplify our notation and drop the superscripts on
difference operators.

By the property of difference operators and the locality of $F$, we have that  \begin{equation} \label{e:nondegdiffop}
D^{j}_{z_1,\ldots,z_{j}} 
D^{k}_{z_{j+1},\ldots,z_{j+k}} \big[ D^{l}_{y_1,\ldots,y_{l}} 
D^{m}_{y_{l+1},\ldots,y_{l+m}}F^+(y,\varphi,\mu,\tphi,\tmu) \big] = 0
 \end{equation}
for all $z_1,\ldots,z_{j+k}\in\R^d$ and $y_1,\ldots,y_{l+m}\in\R^l$ such that
$z_i \notin B_r \cup B(y,r)$ for some $i\in[j+k]$ or $y_i \notin B_r \cup B(y,r)$ 
for some $i\in[l+m]$.
Also, because $F^+ \in [0,1]$, we have that 
\begin{equation}
\Big| D^{j}_{z_1,\ldots,z_{j}} 
D^{k}_{z_{j+1},\ldots,z_{j+k}} \big[ D^{l}_{y_1,\ldots,y_{l}} 
D^{m}_{y_{l+1},\ldots,y_{l+m}}F^+(y,\varphi,\mu,\tphi,\tmu) \big] \Big| \leq 2^{j+k+l+m}.
\label{e:bdsDiff}
\end{equation}
Now using the above facts on difference operators and their symmetry, 
We shall apply and compare the FME expansions of $\BE^{\Phi}_{0,y} \big[ F ( y,\Phi,\Gamma,\Phi,\Gamma ) \big] \, \rho^{(2)}(0,y)$ 
and $\BE^{(y)}\big[ F (y,\Phi,\Gamma,\tPhi,\tGamma )\big] \, \gamma^2$. 

In case $\varphi = \tphi, \mu = \tmu$, we abbreviate $F^+(y,\varphi,\mu,\tphi,\tmu)$ by $F^+(y,\varphi,\mu)$. Thus by applying FME expansion in Lemma \ref{l:FME2pp} straightforwardly to
$\BE^{\Phi}_{0,y}[F(y,\Phi,\Gamma,\Phi,\Gamma)]$ and then using the symmetry
of the difference operators, \eqref{e:nondegdiffop}, and the Palm
correlation formula \eqref{e:palmcorr}, we obtain that
\begin{align*}
& \BE^{\Phi}_{0,y} \big[ F ( y,\Phi,\Gamma,\Phi,\Gamma ) \big] \, \rho^{(2)}(0,y) \\
& = \sum_{k=0}^{\infty}\sum_{l=0}^{\infty} \frac{\rho^{(2)}(0,y)}{k!l!}
\int_{(\R^d)^l}\int_{ (\R^d)^k} D^{k}_{z_1,\ldots,z_k}[D^{l}_{y_1,\ldots,y_l}F^+(0,y,o,o)] \\
& \qquad \times \rho_{\Gamma}^{(k)}(z_1,\ldots,z_k) \, \rho_{\Phi,0,y}^{(l)}(y_1,\ldots,y_l) \, \md(z_1, \ldots,z_k)\, \md(y_1, \ldots,y_l) \\
&   = \sum_{k=0}^{\infty}\sum_{k_1 = 0}^{k}\sum_{l=0}^{\infty}\sum_{l_1 =0}^l \frac{1}{k_1!l_1!(k-k_1)!(l-l_1)!} \\
& \quad \times \int_{ B_r^{k_1} \times B(y,r)^{k-k_1}\times B_r^{l_1} \times B(y,r)^{l-l_1} } D^{k_1}_{z_1,\ldots,z_{k_1}} 
D^{(k-k_1)}_{z_{k_1+1},\ldots,z_k}[D^{l_1}_{y_1,\ldots,y_{l_1}} 
D^{(l-l_1)}_{y_{l_1+1},\ldots,y_l}F(y,\delta_0,o,\delta_y,o)]  \, \, \\
& \qquad \qquad \times \rho^{(k)}_k(z_1,\ldots,z_k) \, \rho^{(l+2)}_\Phi(0,y,y_1,\ldots,y_l) 
\, \md(z_1,\ldots,z_k,y_1 \ldots,y_l)
\end{align*}

On the other hand, note that by locality of $F$,
$F (y,\Phi,\Gamma,\tPhi,\tGamma )$ is also a function of two point processes
$\Phi' := \big( \Phi \cap B_r \big) \cup \big( \tPhi \cap B(y,r) \big)$ and
$\Gamma' := \big( \Gamma \cap B_r \big) \cup \big( \tGamma \cap B(y,r) \big)$. Obserrve that the independence of $\Phi,\Gamma$ with $\tPhi,\tGamma$ factorizes
their respective correlation functions i.e.,
$$\rho_{\Phi \cup \tPhi}^{(k)}(x_1,\ldots,x_k) = \sum_{S \subset [k]}\rho^{|S|}_{\Phi}(x_i : i \in S)\rho^{k-|S|}_{\tPhi}(x_i : i \notin S),$$
and a similar decomposition holds for $\rho_{\Gamma \cup \tGamma}$ as well; For example, see \cite[(1.9)]{BYYsupp19}. Further if we assume that $x_1,\ldots,x_{k_1} \in B_r$ and $x_{k_1+1},\ldots,x_k \in B(y,r)$, then from the above decomposition we have that
$$\rho_{\Phi'}^{(k)}(x_1,\ldots,x_k) = \rho_{\Phi}^{(k_1)}(x_1,\ldots,x_{k_1})\rho_{\tPhi}^{(k-k_1)}(x_{k_1+1},\ldots,x_k),$$
and a similar factorization applies to $\Gamma'$ as well. Also by the independent superposition property, it holds that the Palm correlations for above choice of $x_1,\ldots,x_k$ factorizes as
$$\rho_{\Phi',0,y}^{(k)}(x_1,\ldots,x_k) = \rho_{\Phi,0}^{(k_1)}(x_1,\ldots,x_{k_1})\rho_{\tPhi,y}^{(k-k_1)}(x_{k_1+1},\ldots,x_k).$$
Now, as before, applying FME expansion in Lemma \ref{l:FME2pp} with respect to $\Phi',\Gamma'$, using the symmetry
of the difference operators, \eqref{e:nondegdiffop} and the Palm
correlation formula \eqref{e:palmcorr}
\begin{align*}
& \BE^{(y)}\big[ F (y,\Phi,\Gamma,\tPhi,\tGamma )\big] \gamma^2  \\
& = \sum_{k=0}^{\infty}\sum_{l=0}^{\infty} \frac{\rho^{(1)}(0)\rho^{(1)}(y)}{k!l!}
\int_{(\R^d)^l}\int_{ (\R^d)^k} D^{k}_{z_1,\ldots,z_k}[D^{l}_{y_1,\ldots,y_l}F^+(0,y,o,o)] \\
& \qquad \times \rho_{\Gamma'}^{(k)}(z_1,\ldots,z_k) \, \rho_{\Phi',0,y}^{(l)}(y_1,\ldots,y_l) \, \md(z_1, \ldots,z_k)\, \md(y_1, \ldots,y_l) \\
&  = \sum_{k=0}^{\infty}\sum_{k_1 = 0}^{k}\sum_{l=0}^{\infty}\sum_{l_1 =0}^l \frac{1}{k_1!l_1!(k-k_1)!(l-l_1)!} \\
& \quad \times \int_{ B_r^{k_1} \times B(y,r)^{k-k_1}\times B_r^{l_1} \times B(y,r)^{l-l_1}}  D^{k_1}_{z_1,\ldots,z_{k_1}} 
D^{(k-k_1)}_{z_{k_1+1},\ldots,z_k}[D^{l_1}_{y_1,\ldots,y_{l_1}} D^{(l-l_1)}_{y_{l_1+1},\ldots,y_l}F(y,\delta_0,o,\delta_y,o)]  \, \, \\ 
& \qquad \times \rho^{(k_1)}_\Gamma(z_1,\ldots,z_{k_1})\rho^{(k-k_1)}_\Gamma(z_{k_1},\ldots,z_k) \\
& \qquad \quad \times \rho^{(l_1+1)}_\Phi(0,y_1,\ldots,y_{l_1}) \, \rho^{(l-l_1+1)}_\Phi(y,y_{l_1},\ldots,y_l) \, \md(z_1.\ldots.z_k,y_1\ldots, y_l) 
\end{align*}
So, from the above two identities and \eqref{e:bdsDiff}, we derive that
\begin{align}
& |\BE^{\Phi}_{0,y}[F (y,\Phi,\Gamma,\Phi,\Gamma )]\rho^{(2)}(0,y) -\BE^{(y)}\big[ F (y,\Phi,\Gamma,\tPhi,\tGamma )\big] \gamma^2| \nonumber \\
& \leq \sum_{k=0}^{\infty}\sum_{k_1 = 0}^{k}\sum_{l=0}^{\infty}\sum_{l_1 =0}^l \frac{2^{k+l} }{k_1!l_1!(k-k_1)!(l-l_1)!} \int_{ B_r^{k_1} \times B(y,r)^{k-k_1}\times B_r^{l_1} \times B(y,r)^{l-l_1} }  \nonumber \\
& \qquad \times \Big| \rho^{(k)}_k(z_1,\ldots,z_k) \, \rho^{(l+2)}_\Phi(0,y,y_1,\ldots,y_l) \nonumber \\
& \qquad \qquad - \rho^{(k_1)}_\Gamma(z_1,\ldots,z_{k_1}) \, \rho^{(k-k_1)}_\Gamma(z_{k_1},\ldots,z_k) \, \rho^{(l_1+1)}_\Phi(0,y_1,\ldots,y_{l_1}) \, \rho^{(l-l_1+1)}_\Phi(y,y_{l_1},\ldots,y_l) \Big| \nonumber \\
\label{e:xi-Tibounds} & \qquad \qquad \qquad \times \md(z_1.\ldots.z_k,y_1\ldots, y_l).
\end{align}
We shall now bound the term in the modulus above.  Set
\begin{align*}
A &=  \rho^{(l+2)}_\Phi(0,y,y_1,\ldots,y_l), \qquad \qquad \qquad \qquad  \qquad A' = \rho_\Gamma(z_1,\ldots,z_k) \\
B &= \rho^{(l_1+1)}_\Phi(0,y_1,\ldots,y_{l_1}) \rho^{(l-l_1+1)}_\Gamma(y,y_{l_1+1},\ldots,y_l),  \quad 
B' = \rho^{(k_1)}_\Gamma(z_1,\ldots,z_{k_1})\rho^{(k-k_1)}_\Gamma(z_{k_1+1},\ldots,z_k).
\end{align*}
Note that by our choice of $t$ and $r$, the points
$0,y_1,\ldots,y_{l_1}$ and $y,y_{l_1+1},\ldots,y_l$ are separated by
distance at least $t - 2r \geq 3t/4$. So are the points
$z_1,\ldots,z_{k_1}$ and $z_{k_1+1},\ldots,z_k$. Then, using
\eqref{e:corrbd}, the assumption of fast-decay of correlations, and setting $\gamma_1 = \max \{\gamma,1\}$, we
can derive that
\begin{align*}
|AA' - BB'| & \leq |AA' - A'B| + |A'B - BB'|  \\
& \leq |A'| \, |A - B| + |B| \, |A' - B'| \\
& \leq  |A'| \, |A-B| + |A-B| \, |A'-B'| + |A| \, |A' - B'| \\
& \leq kC_k\gamma^k C_{l+2} \delta\Big(\frac{3t}{4}\Big) 
+ C_{l+2}C_k \delta\Big(\frac{3t}{4}\Big)^2 + (l+2)C_{l+2}\gamma^{l+2}\delta\Big(\frac{3t}{4}\Big)   \\
& \leq \delta\Big(\frac{3t}{4}\Big) C_kC_{l+2} (k\gamma^k + 1 + (l+2)\gamma^{l+2}) 
\leq \delta\Big(\frac{3t}{4}\Big) C_kC_{l+2}(k+1)(l+2)\gamma_1^{k+l+2}
\end{align*}
Thus substituting the above bounds into \eqref{e:xi-Tibounds}, we have that \eqref{e:xi-Tibounds} can be bounded above by
\begin{align}
& \delta\Big(\frac{3t}{4}\Big) \sum_{k=0}^{\infty}\sum_{k_1 = 0}^{k}\sum_{l=0}^{\infty}\sum_{l_1 =0}^l 
\frac{C_kC_{l+2}(2\pi_dr^d)^{k+l}(k+1)(l+2)\gamma_1^{k+l+2}}{k_1!l_1!(k-k_1)!(l-l_1)!} \nonumber \\
\label{e:diffT4T8} & \leq \delta\Big(\frac{3t}{4}\Big) \sum_{k=0}^{\infty}\sum_{l=0}^{\infty} 
\frac{C_kC_{l+2}(4\pi_dr^d)^{k+l}(k+1)(l+2)\gamma_1^{k+l+2}}{k!l!}.
\end{align}
The above double series splits into the product of two series.
Each of these series can be treated with the methods used around
\cite[(3.26)]{BYY19}.  It follows that for some constants $c_1,c_2$, 
$$ \sum_{k=0}^{\infty}\sum_{l=0}^{\infty} \frac{C_kC_{l+2}(4\pi_dr^d)^{k+l}(k+1)(l+2)
\gamma_1^{k+l+2}}{k!l!} \leq c_1e^{c_2r^{\frac{2+d}{1-a}}} = c_1e^{c_2(\frac{t}{8})^{\beta \frac{2+d}{1-a}}}.$$
The constants $c_1,c_2$ depend only on $\beta,d,\gamma_1$ and $C_k$'s but are independent of $F$ and $r = r_F$. Thus, using that  $\delta$ is exponentially fast decreasing with exponent $b$ and $\beta \frac{2+d}{1-a} < b$, we have that
\begin{align*}
|\BE^{\Phi}_{0,y}[F (y,\Phi,\Gamma,\Phi,\Gamma )]\rho^{(2)}(0,y) -\BE^{(y)}\big[ F (y,\Phi,\Gamma,\tPhi,\tGamma )\big] \gamma^2| \leq c_1\delta\Big(\frac{3t}{4}\Big) e^{c_2(\frac{t}{8})^{\beta \frac{2+d}{1-a}}} \le \tC \tvphi(t),
\end{align*}
for an (exponentially) fast-decreasing function $\tvphi$, finite constant $\tC$ and
non-zero constant $\tc$. By choice of $c_1,c_2$ above, we also obtain that $\tilde{C}$ depends only on $\beta,d,b,\gamma_1$ and $C_k$'s and $\tvphi$ depends only on these parameters and $\gamma$ but both are independent of $F$ as well as $r_F$. This completes the proof of
\eqref{e:decorrxi}.
\end{proof}

\subsection{Decay of correlations and void probabilities: Examples}
\label{s:ex_pp_mix}

A key assumption on point processes $\Phi,\Gamma$ in Theorem
\ref{tgen_mix_pert_pp} is exponentially fast decay of correlations
with certain assumptions on the decay constants. 
If, for instance, $\Phi$ is a
stationary {\em $\alpha$-determinantal process} with
$\alpha = -1/m$, $m \in \N$ and kernel $K$, then it has exponential
fast decay of correlations if $|K(x,y)| \leq \omega(|x-y|)$ where
$\omega$ is an exponentially fast decreasing function. Furthermore,
the constants $C_l = O(l^{al})$ for some $a < 1$; see \cite[Proposition 2.2]{BYY19}
It is known that more point processes (for ex., zeros of Gaussian
entire functions, Gibbs point proceses, Cox point processes et al.)
satisfy exponentially fast decay of correlations; 
see \cite[Section 2.2.2 and 2.2.3]{BYY19}. 
However the condition on growth of $C_l$ 
must be checked in each case.  
For the (rarified) Gibbs processes studied
in \cite{SchrYuk13} the condition holds. The same is true for the related ``subcritical" Gibbs processes studied in \cite[Section 3.3.5]{Betsch22}. For permanental processes (see Proposition \ref{p:perm}) 
whose kernel has bounded support, it can be easily
seen that the left hand-side of \eqref{defn:admissible} vanishes for $s$ larger than some
$r_0>0$, uniformly in $p,q$ and $x_1,\ldots,x_{p+q}$. The same applies to the stationary version
of a {\em shot noise Cox process}; see e.g.\ \cite[Example 15.15]{LastPenrose17}.
In both cases we can apply \eqref{e:kappadelcompsupp}. For permanental, the trivial upper bound on permanent gives that $C_K \leq k! \|K\|_{\infty}^k$ for all $k \in \N$. In the example of shot noise Cox processes,
assuming that the intensity field has moments of all orders, we can obtain bounds on $C_k$ via H\"{o}lder's inequality.
For a  kernel with unbounded support it does not seem to be possible
to bound the constants $C_l$ in the required way.

In many examples, the verification of decay bounds for stopping sets boils
down to suitable void probability bounds; see for example
Propositions \ref{p:nn2pp}, \ref{p:nn1pp} and Example \ref{e:weightedVoronoi}. 
Let $\Phi$ be a stationary simple point process with finite intensity measure
and Palm probability measure $\BP_0$ and $\BP_{0,y}$. The reduced
Palm probability measures are denoted by
$\BP^!_0:=\BP_0(\Phi-\delta_0\in\cdot)$ and 
$\BP^!_{0,y}:=\BP_{0,y}(\Phi-\delta_0-\delta_y\in\cdot)$. (The careful reader will again
notice a slight abuse of notation.) The required bounds are of the form
\begin{align}\label{void}
\max\{\BP(\Phi(B_t)=0), \BP^{!}_{0}(\Phi(B_t)=0),\BP^{!}_{0,y}(\Phi(B_t) =0)\} \le \delta_1(t), \, 
\alpha_\Phi\text{-a.e.\ $y\in\R^d$},\,t\ge 0,
\end{align}
where $\delta_1$ is a fast decaying function.
Of course, a stationary Poisson process has this property
with $\delta_1(t)=e^{-ct^d}$ for some $c>0$.
We now present less trivial examples.

\begin{example}
\label{ex:alphaDPPvoid}
\rm Assume that $\Phi$ is
a stationary $\alpha$-determinantal point process with $-1/\alpha \in \N$. Then we have
for each bounded  Borel set $B$,
\begin{align*}
\max\{\BP(\Phi(B)=0), \BP^{!}_{0}(\Phi(B) =0), \BP^{!}_{0,y}(\Phi(B)=0)\} 
\leq  ce^{-c'\lambda_d(B)}, \quad y \in \R^d.
\end{align*}
for some constants $c,c'$ that depend on $\alpha$ and $d$; see \cite[Corollary 1.10]{BYYsupp19}. Though the Corollary in \cite{BYYsupp19} is not stated for the (non-Palm) void probability bound $\BP(\Phi(B)=0)$, the proofs therein work more easily in this case and yield the above bound. Thus a stationary $\alpha$-determinantal point process as above with an exponentially fast decaying kernel is well suited for our applications of Theorem \ref{tgen_mix_pert_pp}. In particular they satisfy the assumptions of upcoming Propositions \ref{p:nn2pp} and \ref{p:nn1pp}.
\end{example}

Let $K\colon\R^d\times\R^d\to \R$ be a
symmetric jointly continuous function, which is 
non-negative definite and translation invariant.
Let $k\in\N$ and $\{Z_1(x):x\in\R^d\},\ldots,\{Z_k(x):x\in\R^d\}$  
be independent centered Gaussian fields with covariance function $K/2$.
Define $W(x):=Y_1(x)^2+\cdots+Y_k(x)^2$, $x\in\R^d$.
A {\em Cox process} with random intensity measure
$W(x) \, dx$ is a  {\em $k/2$-permanental process} with kernel
$K$ and reference measure $\lambda_d$; see
e.g.\ \cite[Chapter 14]{LastPenrose17} for more details.

\begin{proposition}\label{p:perm} Suppose that $\Phi$ is a
$k/2$-permanental process with a continuous and positive semi-definite kernel
$K$ and reference measure $\lambda_d$. Assume moreover that
\begin{align}\label{8025}
\int K(0,x)^2\,\md x<\infty
\end{align}
and $K(0,0)>0$. Then there exist $c,c'>0$ such that for $r\ge 0$ and  $\alpha_\Phi$-a.e.\ $y\in\R^d$
\begin{align}\label{8024}
\max\{\BP(\Phi(B_r)=0),\BP^!_0(\Phi(B_r)=0), \BP^!_{0,y}(\Phi(B_r)=0)\}\le c\exp\big[-c' r^{d/2}\big].
\end{align}
\end{proposition}
Thus if $\Phi$ is a $k/2$-permanental process as above with a compactly supported kernel $K$, then it is well suited for our applications of Theorem \ref{tgen_mix_pert_pp}. In particular it will satisfy the assumptions of upcoming Propositions \ref{p:nn2pp} and \ref{p:nn1pp}.
\begin{proof}
Let  $B\subset\R^d$ be a compact set. As in \cite[(14.15)]{LastPenrose17}
we can write 
\begin{align*}
K(x,y)=\sum^\infty_{j=1}\gamma_{B,j}g_{B,j}(x)g_{B,j}(y),\quad x,y\in B,
\end{align*}
where the functions $g_{B,j}(x)$, $j\in\N$, are pairwise orthogonal
in $L^2((\lambda_d)_B)$ with norm one. 
The non-negative numbers $\gamma_{B,j}$ are the eigenvalues
of the linear integral operator $K_B$ on $L^2((\lambda_d)_B)$  
associated with the restriction of $K$ to $B\times B$. By
\cite[(14.32)]{LastPenrose17} and the inequality $1+s\le e^s$, $s\in\R$,
we have
\begin{align}\label{8022}
\BP(\Phi(B)=0)\le\exp\bigg[-\frac{k}{2}\sum_j\tilde\gamma_{B,j}\bigg]
\end{align}
where $\tilde\gamma_{B,j}:=\gamma_{B,j}/(1+\gamma_{B,j})$.
On the other hand it follows from the Cauchy--Schwarz inequality that
the operator norm of $K_B$ is bounded by
\begin{align*}
\|K\|_B:=\bigg(\int_{B\times B} K(x,y)^2\, d(x,y)\bigg)^{1/2}.
\end{align*}
In particular we obtain $\gamma_{B,j}\le \|K\|_B$, $j\in\N$, and
it follows from \eqref{8022} that
\begin{align}\label{8023}
\BP(\Phi(B)=0)\le\exp\bigg[-\frac{kK(0,0)\lambda_d(B)}{2 (\|K\|_B +1)}\bigg],
\end{align}
where we have used that 
$$
\sum_j\gamma_{B,j}=\int_B K(x,x)\,\md x=K(0,0)\lambda_d(B).
$$
By translation invariance of $K$ and Fubini's theorem we have
\begin{align*}
\|K\|^2_B=\int_B \lambda_d(B\cap (B+x))K(0,x)^2\,\md x.
\end{align*}
Using assumption \eqref{8025}
we obtain that $\lambda_d(B_r)^{-1}\|K\|^2_{B_r}\to \int K(0,x)^2\,\md x<\infty$
as $r\to\infty$. In view of \eqref{8023} this proves the first part of \eqref{8024}.

To treat Palm probabilities we use that
\begin{align*}
\BP^!_{0,y}(\Phi\in\cdot)= (\BE W(0)W(y))^{-1}\BE W(0)W(y)\I\{\Phi\in\cdot\},\quad \alpha_\Phi\text{-a.e.\ $y\in\R^d$}.
\end{align*}
This can be proved in a straigthforward way by using the Mecke equation for Cox processes;
see \cite[Theorem 13.8]{LastPenrose17}.
A straightforward calculation with a bivariate normal distribution
yields the well-known formula
\begin{align}\label{e:normalbi}
4\ \BE W(0)W(y)=k^2 K(0,0)^2+2k K(0,y)^2,\quad y\in\R^d.
\end{align}
Taking a compact set $B\subset\R^d$ we hence obtain  
for $\alpha_\Phi$-a.e.\ $y\in\R^d$ that 
\begin{align}\notag \label{Coxnull}
\BP^!_{0,y}(\Phi(B)=0)&\le c\,\BE W(0)W(y)\I\{\Phi(B)=0\}\\
&=c\,\BE W(0)W(y)\exp\bigg[-\int_B W(x)\,\md x\bigg]\notag\\
&\leq c \sqrt{\BE[ W(0)^2W(y)^2]} \sqrt{\BE\bigg[\exp\bigg[-2\int_B W(x) \,\md x\bigg]\bigg]}\notag\\
&\leq c\sqrt{\BE[W(0)^4]} \sqrt{\BP(\Phi(B)=0)},
\end{align}
where the equality come from conditioning w.r.t.\ $W$ and $c:=4 k^{-2}K(0,0)^{-2}$.
As $\BE[W(0)^4]<\infty$, the third part of \eqref{8024} can be derived from our previous bound on $\BP(\Phi(B)=0)$. The remaining part of \eqref{8024} can be proven similarly (The argument is simpler).
\end{proof}

In the next proposition we consider a {\em Gibbs process} $\Phi$
with a {\em Papangelou intensity} $\lambda \colon \R^d\times\bN\to[0,\infty)$,
a measurable function. The distribution of such a process is determined by the
so-called GNZ-equations 
\begin{align}\label{eGNZ}
\BE\Bigg[\int f(x,\Phi)\,\Phi(\mathrm{d}x)\Bigg]=
\BE\bigg[ \int f(x,\Phi+\delta_x) \lambda(x,\Phi)\,\mathrm{d}x\bigg],
\end{align}
which should hold for all measurable $f\colon\BX\times\bN\to[0,\infty)$.
We shall assume that $\lambda$ is translation invariant and bounded from 
above by some $a\ge 0$.
Examples of Gibbs processes satisfying \eqref{void} for the stationary
void probabilities can be found in the seminal paper \cite{SchrYuk13}.
We shall use here a result from \cite{LastOtto23}
to establish \eqref{void} for a certain class of Gibbs processes.
Taking  a reflection symmetric and bounded Borel set $N\subset\R^d$ 
we shall assume that
\begin{align}\label{paplocal}
\lambda(x,\mu)=\lambda(x,\mu_{N_x}),\quad (x,\mu)\in\R^d\times\bN,
\end{align}
where $N_x:=N+x$.

\begin{proposition}\label{p:voidGibbs} Let $\Phi$ be a stationary Gibbs process with a bounded Papangelou
intensity satisfying \eqref{paplocal}. Assume that $\lambda(0,o)>0$. 
Then there exist $c,c'>0$ such that
\begin{align}\label{G1}
\max\{\BP(\Phi(B_r)=0),\BP^{!}_{0}(\Phi(B_r)=0)\}\le c\exp\big[-c' r^d\big],\quad r\ge 0.
\end{align}
Assume moreover that 
\begin{align}\label{G11}
\{y\in N:\lambda(0,\delta_y)=0\}\subset \{y\in N:\lambda(0,\mu+\delta_y)=0\},\quad \mu\in\bN,
\end{align}
and
\begin{align}\label{G12}
\sup\{\lambda(y,\mu+\delta_0)/\lambda(y,\delta_0):y\in N,\lambda(y,\delta_0)>0,\mu\in\bN\}<\infty.
\end{align}
Then
\begin{align}\label{G3}
\BP^{!}_{0,y}(\Phi(B_r) =0)
\le 
c \exp\big[-c' r^d\big],\quad  \alpha_\Phi\text{-a.e.\ $y\in\R^d$},\,r\ge 0.
\end{align} 
\end{proposition}
\begin{proof} It follows from \cite[Corollary 7.7]{LastOtto23}
that 
\begin{align}\label{harduppbou2}
\BP(\Phi(B_r)=0)\le \exp\bigg[-\int \I\{x\in B_r\} \, b \, \lambda(x,\Phi_{B^c_r})\,\md x\bigg],
\end{align}
where $b:=e^{-a\lambda_d(N)}$. Since $N$ is bounded there exists $r_0>0$ such that
$N\subset B_{r_0}$, so that $N_x\subset B(x,r_0)$. Hence 
$N_x\cap B^c_r=\emptyset$ for $\|x\|\le r-r_0$ and it follows that
\begin{align*}
\BP(\Phi(B_r)=0)\le \exp\bigg[-\int \I\{x\in B_{r-r_0}\}b \lambda(x,o)\,\md x\bigg],
\end{align*}
By translation invariance of $\lambda$ we have $\lambda(x,0)=\lambda(0,0)$
and the first part of \eqref{G1} follows.
Iterating \eqref{eGNZ} easily shows that we can choose
$\BP^!_{0,y}$ such that
\begin{align}\label{G15}
\BP^!_{0,y}(\Phi\in\cdot)
=(\BE \lambda^{(2)}(y,\Phi))^{\oplus}\BE \lambda^{(2)}(y,\Phi)\I\{\Phi\in\cdot\},\quad y\in\R^d,
\end{align}
where $\lambda^{(2)}\colon\R^d\times\bN\to[0,\infty)$ is defined by 
$\lambda^{(2)}(y,\mu):=\lambda(0,\mu)\lambda(y,\mu+\delta_0)$
and $a^{\oplus}:=\I\{a\ne  0\}a^{-1}$ is the generalized 
inverse of $a\in\R$. We have
\begin{align*} 
\BE \lambda^{(2)}(y,\Phi)\ge \BE\I\{\Phi(N)=\Phi(N_y)=0\}\lambda(0,o)\lambda(y,\delta_0)
\end{align*}
Since $\Phi$ is stochastically dominated by a stationary Poisson process with intensity $a$
(see \cite{GeorKun97}) we obtain that
\begin{align*}
\BE \lambda^{(2)}(y,\Phi)\ge \lambda(0,0)\lambda(y,\delta_0)\exp[-a\lambda_d(N\cup N_y)]
\ge \lambda(0,o)\lambda(y,\delta_0)\exp[-2a\lambda_d(N)].
\end{align*}
If $y\notin N$, then $\lambda(y,\delta_0)=\lambda(y,o)=\lambda(0,o)>0$.
If $y\in N$ and $\lambda(y,\delta_0)=0$ then \eqref{G11} implies $\lambda^{(2)}(y,\Phi)=0$.
In any case 
\begin{align*}
(\BE \lambda^{(2)}(y,\Phi))^{\oplus}\le c \lambda(y,\delta_0)^\oplus.
\end{align*}
Using \eqref{G15} and then assumption  \eqref{G12} (and the boundedness of $\lambda$) we obtain that
\begin{align*}
\BP^{!}_{0,y}(\Phi(B_r)=0)\le c\,\BE \lambda(y,\Phi+\delta_0)\lambda(y,\delta_0)^\oplus\I\{\Phi(B_r)=0\}
\le c'' \,\BP(\Phi(B_r)=0),
\end{align*}
for some $c''>0$. Hence \eqref{G3} follows from the first part of the proof.

The second part of \eqref{G1} follows from 
\begin{align*}
\BP^!_{0}(\Phi\in\cdot)
=(\BE \lambda(0,\Phi))^{\oplus}\BE \lambda(0,\Phi)\I\{\Phi\in\cdot\},
\end{align*}
and $\BE \lambda(0,\Phi)\ge \lambda(0,0)\exp[-2\lambda_d(N)]$. Assumptions
\eqref{G11} and \eqref{G12} are not required in this case.
\end{proof}

\begin{example}\rm Suppose that $U\colon \R^d\to [0,\infty)$ is a measurable
and symmetric function with bounded support. Let $a>0$ and define
\begin{align*}
\lambda(x,\mu):=a \exp\bigg[-\int U(y-x)\,\mu(\mathrm{d}y)\bigg],\quad 
		(x,\mu) \in \R^d\times\bN.
\end{align*}
A point process with Papangelou intensity $\lambda$ is 
a Gibbs process with {\em pair potential} $U$.
This $\lambda$ satisfies all assumptions from Proposition \ref{p:voidGibbs}, including
\eqref{G11} and \eqref{G12}. The first part of \eqref{G1} is covered by
\cite[Lemma 3.3]{SchrYuk13}.
An interesting special case is the {\em Strauss  process}, where
$\lambda(x,\mu)=ab^{\mu(B(x,r_0))}$ for some $b\in[0,1]$ and $r_0> 0$. For $b=0$ this process describes hard spheres in equilibrium. Thus if $\Phi$ is a Gibbs process with pair potential $U$ and $a$ small enough, then it is well suited for our applications of Theorem \ref{tgen_mix_pert_pp}. In particular it will satisfy the assumptions of upcoming Propositions \ref{p:nn2pp} and \ref{p:nn1pp}. Conditions for the precise choice of $a$ can be inferred from the 'subcritical' condition  in \cite[Section 3.3.5]{Betsch22} or 'rarefication condition' in \cite{SchrYuk13}.
\end{example}

\section{Local transport kernels of point processes}
\label{s:localalgo}

In this section, we will provide more natural examples of
  invariant probability kernels satisfying the assumptions of Theorem
  \ref{maintheoremmixing2} and thereby preserving equality of
  asymptotic variances. The driving intuition behind these examples is
  that local probability kernels should have good mixing as required
  by Theorem \ref{maintheoremmixing2}. We characterize `locality' by
  requiring that the kernels are determined by `nice' random stopping
  sets of the underlying point process (and possibly some independent
  point process). The key theoretical tool formalizing this is Theorem
  \ref{tgen_mix_pert_pp}, which was stated and proven in Section
  \ref{s:mix_marked_pp}. This applies to point processes having fast
  decay of correlation functions as in Definition
  \ref{defn:admissible} and proceeds via factorial moment
  expansions. The framework of exponentially fast decay of correlations
  and stopping sets introduced at the beginning of Section
  \ref{s:mix_marked_pp} is necessary to follow the proof of results in
  this section. In order to follow the results more easily, the reader
  can assume that the underlying point process in our examples is
  Poisson. We remind the reader that other examples of point processes satisfying the decay of correlation conditions and void probability assumptions in upcoming results can be found in Section \ref{s:ex_pp_mix}.

Motivated by the random organization model, we first start with a
  simple example of probability kernels determined by a bounded set
  around a point in Section \ref{s:bddperturbations}. Next, in Section
  \ref{s:nnshift}, we study examples based on nearest-neighbour shifts
  of a point process and here already the stopping set framework of
  Theorem \ref{tgen_mix_pert_pp} is necessary. Finally, in
  Section \ref{s:lloyd}, we show that non-hyperuniformity of the
  Poisson process is preserved under finitely many iterations of the
  Lloyd's algorithm.  Applications to random measures and random sets
  are discussed in Sections \ref{s:volumes} and \ref{s:hyprandset}. Though Theorem
  \ref{tgen_mix_pert_pp} allows us to choose more general probability
  kernels, in this section we consider only allocations.

\subsection{The random organization model}
\label{s:bddperturbations}

Our methods directly apply to a prominent model of
self-organization in driven systems known as \textit{random organization}~\cite{Corte08},
which has recently attracted considerable attention. 
In the terminology of stochastic geometry, the model starts from a particle process,
where each particle is a compact, convex set and the particle centers form a Poisson point process.
The model then iteratively shifts all particles in a cluster with more than one particle.
The shifts are random with uniformly distributed directions but within a fixed distance.

Heuristic arguments and simulations suggest that there is a phase
transition such that below a critical intensity, the
model relaxes to a frozen state where no particle moves; but above the
critical intensity, the model remains active for all times. For a
related model in one dimension, such a phase transition has been proven
rigorously~\cite{Ayyer2023}.
Simulations, moreover, suggest that the model is hyperuniform at
the critical point~\cite{HexnerLevine15} and in the active phase for
variants of the model~\cite{HexnerLevine17, hexner17}.
For a closely related model on the lattice, the facilitated exclusion
process (also known as the conserved lattice gas model), hyperuniformity
at the critical density has recently been proven in one
dimension~\cite{GoldsteinLebowitzSpeer24}.
Here we show that hyperuniformity cannot be achieved in a finite number of
steps, i.e., hyperuniformity can only be obtained in the limit of an infinite number of steps.

We now give a general framework that includes all bounded perturbations as well
as shows that recursively applying bounded perturbations to point processes with
fast decay of correlations preserve the variance asymptotics. Informally, in
this model, points are perturbed locally depending on the local configuration
(say the configuration within a unit ball). 

\begin{proposition}
\label{p:var_bdd_pert}
Let $\Phi$ be a stationary point process with non-zero intensity and having exponentially fast decay of correlations with the decay function $\delta$ and constants
$C_k, k \in \N$, such that $C_k = O(k^{ak})$ for some $a < 1$.  Let
$Y\colon \bN \to B_1$ be a measurable function such that
$Y(\varphi) = Y(\varphi \cap B_1)$. Define a sequence of point
processes $\Phi_k, k \geq 1$ as follows.  $\Phi_0 = \Phi$ and for all
$k \geq 1$,
\begin{align}\label{e:defnpsik}
\Phi_k(B):= \sum_{x \in \Phi_{k-1}} \delta_{x + Y(\Phi_{k-1} - x)}.
\end{align}
Then,  we have that for all $k \geq 1$,  
$$\lim_{r \to \infty} \lambda_d(B_r)^{-1}\BV[\Phi_k(B_r)] 
= \lim_{r \to \infty} \lambda_d(B_r)^{-1}\BV[\Phi(B_r)].$$
\end{proposition}
\begin{proof}
Let $\varphi\in\bN$. For each $x \in \varphi$ and for each $k \geq 1$, we define recursively
    $x^{(k)}\equiv x^{(k)}(\varphi,x)\in\R^d$ and $\varphi_k\in\bN$   as follows. 
We set $x^{(0)} = x$ and
    $x^{(k)} = x^{(k-1)} + Y(\varphi_{k-1} - x^{(k-1)})$, where
    $\varphi_k = \{x^{(k)}\}_{x \in \varphi}$.
We set the (canonically defined) transport kernel $\tilde{K}$ to be
$\tilde{K}(\varphi,x) := \delta_{x^{(k)}(\varphi,x)}$; see also \eqref{factor}. Thus by the recursive nature of
the definition of \eqref{e:defnpsik}, we can verify inductively that
$$ 
\Phi_k =  \sum_ {x\in \Phi} \delta_{x^{(k)}(\Phi,x)} = K\Phi,
$$  
where $K(x):=\tilde{K}(\Phi,x)$. Note that $K$ depends only on a
  single point process $\Phi$. Thus by definition it is easy to verify
  that $K$ satisfies the assumptions of Theorem \ref{tgen_mix_pert_pp} 
with the stopping set $S = B_{6k}$ and so $\delta_1$ is compactly supported. Thus trivially
\eqref{e:intvarphi1} holds and so does equality of asymptotic variances.
\end{proof}
It is also of interest to consider models where the perturbations
include additional randomness as in the random organization model. We shall show that our results
can apply to such models by considering one such model and when the
underlying point process $\Phi$ is a Poisson process. As will be seen,
the extension to finitely dependent point processes is immediate but
extension to more general point processes shall need an extension of
Lemma \ref{l:FME2pp} or Proposition \ref{p:dec_corr} to general marked
point processes. 

The framework is as follows.  Let $Y_1\colon \bN\to \{0,1\}$ and
$Y_2\colon \bN \to B_1$ be two measurable functions such that 
$Y_i(\cdot) = Y_i(\cdot \cap B_1)$ for $i\in\{1,2\}$.
For $i\in\{1,2\}$, $\varphi\in\bN$ and $x \in \R^d$, we set  $Y_i(x,\varphi) :=Y_i(\varphi - x)$. 
Let $U_{j,i}, j,i \geq 1$ be i.i.d.\
$B_1$-valued random vectors. 
Starting with  $\varphi_0 = \varphi\in\bN$ we define $\varphi_k\in\bN$, $k\in\N$, as follows.
Given $\varphi_{k-1}=\sum_i \delta_{x_i}$, we set
\begin{align}\label{e:defnpsik1}
\varphi_k(B):= 
\sum_{i } \delta_{x_i + Y_1(x_i,\varphi_{k-1})( Y_2(x_i,\varphi_{k-1}) + U_{k,i})},\quad k \geq 1.
\end{align}

In other words, a point $x \in \varphi_{k-1}$ is displaced only if
$Y_1(x,\varphi_{k-1}) = 1$ and if so, it is displaced by
$Y_2(x,\varphi_{k-1})$ and additionally by an independent random
vector. 

\begin{proposition}
\label{t:var_bdd_pert_rand}
Let $\Phi$ be a stationary Poisson process of non-zero intensity $\gamma$ and
$\{U_{k,n}\}$ be i.i.d.\ $B_1$-valued random vectors independent of
$\Phi$. Starting with $\Phi_0 = \Phi$, we define recursive
displacements $\Phi_k$, $k \geq 1$, using \eqref{e:defnpsik1}. Then, we
have that for all $k \geq 1$,
$$\lim_{r \to \infty} \lambda_d(B_r)^{-1}\BV[\Phi_k(B_r)] = \lim_{r \to \infty} \lambda_d(B_r)^{-1}\BV[\Phi(B_r)] = \gamma.$$
\end{proposition}
\begin{proof} 
To apply Corollary \ref{c:alloc} we need to redefine $\Phi$ and $\Phi_k$ on a suitable
probability space $(\Omega',\mathcal{A}',\BP')$ equipped with a flow. Let 
$\Omega'$ be the measurable set of all $\omega\in\bN(\R^d\times B_1^\infty)$ such that
$\bar\omega:=\omega(\cdot \times B_1^\infty)\in\bN_s$. For $\omega\in\Omega'$ and $x\in\R^d$
we define $\theta_x\omega\in\Omega'$ by
$\theta_x\omega(B\times C):=\omega((B+x)\times C)$. Let $\BP'$ be the distribution
of a Poisson process with intensity measure $\lambda_d\times \BQ^\infty$, where
$\BQ$ is the distribution of $U_{1,1}$. Then $\BP'$ is stationary w.r.t.\ $\{\theta_x:x\in\R^d\}$.
For $\omega\in\Omega'$ and $x\in\bar\omega$ there exists a unique
$(u_n)_{n\ge 1}\in B_1^\infty$ such that $(x,(u_n)_{n\ge 1})\in \omega$.
We write $u_n(\omega,x):=u_n$, $n\in\N$. If $x\notin\bar\omega$ we let
$u_n(\omega,x)$ equal some fixed value in $B_1$. Then $u_n$ is shift-invariant,
that is $u_n(\theta_y\omega,x-y)=u_n(\omega,x)$ for each $y\in\R^d$.

Next we define for each $k\in\N_0$ a measurable mapping $\tau_k\colon \Omega'\times\R^d\to\R^d\cup\{\infty\}$
such that $\tau_k(\omega,x)\in\R^d$ whenever $x\in\bar\omega$.
We do this recursively as follows. 
Let $\tau_0(\omega,x):=x$ if $x\in\bar\omega$.
Otherwise set $\tau_0(\omega,x):=\infty$. Let $k\in\N$ and assume that $\tau_{k-1}$ is given. Define
\begin{align*}
\chi_{k-1}(\omega):=\sum_{x\in\bar\omega}\delta_{\tau_{k-1}(\omega,x)}.
\end{align*}
If $x\in\bar\omega$ we define
\begin{align*}
\tau_k(\omega,x):=\tau_{k-1}(\omega,x) + Y_1(\chi_{k-1}(\omega) -\tau_{k-1}(\omega,x))
(Y_2(\chi_{k-1}(\omega) - \tau_{k-1}(\omega,x)) + u_k(\omega,x)).
\end{align*}
Otherwise set $\tau_k(\omega,x):=\infty$.
One can easily establish by induction, that $\tau_k$ is an invariant allocation;
see \eqref{allocation}. 

Let $\Psi$ denote the identity on $\Omega'$. Then $\Phi':=\Psi(\cdot\times\R^d)$
has the same distribution as $\Phi$. Moreover, $\Psi$ is an independent marking
of $\Phi'$; see \cite[Chapter 5]{LastPenrose17}. Therefore
\begin{align*}
\Phi'_k:=\int\I\{\tau_k(\Psi,x)\in\cdot\}\,\Phi'(dx)
\end{align*}
has the same distribution as $\Phi_k$. 

It remains to check that $\tau_k$ satisfies
the assumption \eqref{maintheoremmixingcondition} from Corollary \ref{c:alloc}.
Since $Y_2$ and the $u_n$ are $B_1$-valued, we obtain by induction that
$\|\tau_k(\omega,x)-x\|\le 2k$ for all $\omega\in\Omega'$ and all $x\in\bar\omega$.
Using the constraints on the domains of $Y_1$ and $Y_2$ it then follows that 
$\tau_k(\omega,x)=\tau_k(\omega_{B_{9k}(x)},x)$, where $\omega_B$
is the restriction of $\omega$ to $B\times B_1^\infty$
for $B\in\cB^d$. 
It follows rather straight from the Mecke equation for $\Psi$ 
(see \cite[Theorem 4.1]{LastPenrose17}) that
the Palm probability measure $\BP'^{\Phi'}_0$ is the distribution of
$\Psi+\delta_{(0,U_0)}$, where $U_0$ has distribution $\BQ^\infty$ and is independent
of $\Psi'$. Similarly we can choose for $y\ne 0$ the 
Palm probability measure $\BP'^{\Phi'}_{0,y}$ as the distribution of
$\Psi+\delta_{(0,U_0)}+\delta_{(y,U_y)}$, where $U_y$ has distribution $\BQ^\infty$ and is independent
of $(\Psi',U_0)$. Therefore for we obtain for $\|y\|\ge 18k$ that
\begin{align*}
\BP'^{\Phi'}_{0,y}(\tau_k(y)-y),\tau_k(0)\in\cdot)
=\BP'^{\Phi'}_{0}(\tau_k(y)-y\in\cdot)\otimes \BP'^{\Phi'}_{y}(\tau_k(0)\in\cdot)
\end{align*}
By stationarity we have $\BP'^{\Phi'}_{y}(\tau_k(y)-y\in\cdot)=\BP'^{\Phi'}_{0}(\tau_k(0)\in\cdot)$.
Therefore we obtain $\kappa(y)=0$ and the assertion follows from 
Corollary \ref{c:alloc}.
\end{proof}

\subsection{Nearest-neighbour shifts of point processes}
\label{s:nnshift}

Our next example uses unbounded stopping sets that have good tails
  as required by Theorem \ref{tgen_mix_pert_pp}. These are based on
  nearest neighbour shifts of points in a point process to an
  independent point process and within the same point process. Since
  the framework of Voronoi tesselation is useful for these examples
  and also other upcoming examples, we will introduce it
now.

Set $d(x,A):=\inf\{\|y-x\|:y\in A\}$ to denote the distance between a
point $x\in\R^d$ and a set $A\subset\R^d$, where
$\inf\emptyset:=\infty$. Let $\varphi\in\bN$ and $x\in\R^d$.  We call
$p\in\varphi$ {\em the nearest neighbour} of $x$ in $\varphi$ if
$\|x-p\|\le \|x-q\|$ for all $q \in \varphi_{\{p\}^c}$. If there
is more than one such $p$, we take the lexicographically smallest
point to be the nearest neighbour. In any case, we set $N(x,\varphi) = p$ and for
completeness, define $N(x,\varphi) := \infty$ if
$\varphi=\emptyset$. Given $\varphi \in \bN$ and $x \in \R^d$, we
define the {\em Voronoi cell} of $x$ (with respect of $\varphi$) as
follows :
\begin{equation}
\label{e.vortes}
C(x,\varphi) := \{ y \in \R^d : N(y,\varphi) = N(x,\varphi) \}.
\end{equation}
If $\varphi \neq \emptyset$, these cells form a partition of $\R^d$.  

Given simple point processes $\Phi,\Gamma$ on $\R^d$, we define the random
measure $\Psi$ as a perturbation of $\Phi$ to its nearest-neighbour in
$\Gamma$. More formally, let 
$Y(x):=N(x,\Gamma)-x$, $x\in\R^d$ and
$$ \Psi := \sum_{x \in \Phi} \delta_{x + Y(x)} = \sum_{z \in \Gamma}\Phi(C(z,\Gamma))\delta_z.$$ 
The below propositions can be proven by Theorem \ref{tgen_mix_pert_pp}
and a straightforward construction of stopping set to determine
the nearest neighbour. 
\begin{proposition}
\label{p:nn2pp}
Let $\Phi,\Gamma$ be independent stationary point processes with non-zero intensities and having exponentially fast decay of correlations with the same decay function
$\delta$ and constants $C_k, k \in \N$ satisfying that
$C_k = O(k^{ak})$ for some $a < 1$. Assume that there exists a fast
decreasing function $\delta_1$ such that
$$ \BP(\Gamma(B_0(t)) = 0)  \leq \delta_1(t).$$
Define {\em weighted Voronoi-cell} measure $\Psi$ as above. Then, we have that 
$$\lim_{r \to \infty} \lambda_d(B_r)^{-1}\BV[\Psi(B_r)] = \lim_{r \to \infty} \lambda_d(B_r)^{-1}\BV[\Phi(B_r)].$$
\end{proposition}
\begin{proof}
Since the other assumptions are as in Theorem \ref{tgen_mix_pert_pp}, 
we only need to provide the construction
of an appropriate stopping set for the transport kernel
    $K(x) = \delta_{x+Y(x)}$ and verify \eqref{stopbounds}.

For $x \in \varphi$, we set $S(x,\varphi,\mu) := S(x,\mu) := B(x,|N(x,\mu) - x|)$ with the understanding that
$B(x,\infty) = \R^d$. (There is no dependence
on $\varphi$.)
Given $x\in\R^d$, the mapping $(\varphi,\mu)\mapsto S(x,\mu)$  
is measurable and satisfies \eqref{stopping2}. Hence it is a stopping set.
Further, observe that
since $\Phi$ and $\Gamma$ are independent, we obtain that for any $y \in \R^d$,
$$ 
\BP_{0}(S(0,\Gamma) \not\subset B_t) = \BP_{0,y}(S(0,\Gamma) \not\subset B_t) = \BP(\Gamma(B_t) = 0),
$$
where $\BP_{0},\BP_{0,y}$ are Palm probabilities with respect to
$\Phi$. With this observation, the void probability assumption on $\Gamma$ and Theorem \ref{tgen_mix_pert_pp}, the
proof is complete.
\end{proof}
In the same spirit as above, one can also consider nearest-neighbour
shifts within a point process as follows. Again, let
$\Phi$ be a stationary point process and define
$Y(x):=N(x,\Phi_{\{x\}^c})-x$, $x\in\R^d$ i.e., $x + Y(x)$ is
the nearest-neighbour of $x$ in $\Phi$ excluding itself. In the
trivial case of $\Phi = \{x\}$, we have that $Y_x = \infty$. Thus,
we have that the perturbed measure is
$$ 
\Psi := \sum_{x \in \Phi} \delta_{x + Y(x)} = \sum_{z \in \Phi}\Phi(A(z,\Phi))\delta_z,
$$ 
where
$A(z,\varphi) = \{ x \in \varphi : x \neq z, N(x,\varphi_{\{x\}^c}) = z \}$, the points whose nearest neighbour is
$z$. Now, we have the following proposition whose proof is similar to
that of Proposition \ref{p:nn2pp} but with suitable modifications.
\begin{proposition}
\label{p:nn1pp}
Let $\Phi$ be a stationary point process with non-zero intensity and having exponentially fast decay
of correlations with the decay function $\delta$ and constants
$C_k$, $k \in \N$, satisfying that $C_k = O(k^{ak})$ for some 
$a<1$. Assume that there exists a fast decreasing function
$\delta_1$ such that
\begin{align*}
\max\{ \BP^!_0(\Phi(B_t) = 0), \sup_{y \in \R^d}\BP^!_{0,y}(\Phi(B_t) = 0)\} \leq \delta_1(t),
\end{align*}
where $\BP_0$ and $\BP_{0,y}$ are the Palm probability measures of $\Phi$.
Define the {\em weighted Voronoi-cell} measure $\Psi$ as above. Then, we have that 
$$\lim_{r \to \infty} \lambda_d(B_r)^{-1}\BV[\Psi(B_r)] = \lim_{r \to \infty} \lambda_d(B_r)^{-1}\BV[\Phi(B_r)].$$
\end{proposition}

\subsection{Lloyd's algorithm}
\label{s:lloyd}
Given a convex bounded set $A \subset \R^d$, we denote its centroid/center of mass by $\Ce(A) := \frac{1}{\lambda_d(A)}\int_A x \, \md x$. Given $\varphi \in \bN$, define the centroidal shift by $\Ce(x,\varphi) := \Ce(C(x,\varphi)),$ with $C(x,\varphi)$ being the Voronoi cell as in \eqref{e.vortes}. Now define inductively a sequence of counting measures successively perturbing each point to the centroid of its Voronoi cell as follows: 
\begin{equation}
 \label{e:Lloyditeration}   
 \varphi_0 := \varphi, \varphi_k := \sum_{x \in \varphi_{k-1}} \delta_{\Ce(x,\varphi_{k-1})}, k \geq 1.
 \end{equation}
We call $\varphi_k$ the $k$-th iterate of {\em Lloyd's algorithm}  applied to $\varphi$; see~\cite{K19} and references therein. Since the Voronoi cells have disjoint interiors, all $\varphi_k$'s are simple. The algorithm is used to obtain (at least approximately) centroidal Voronoi
tessellations, for which the Voronoi center coincides with the center of mass for each cell. The upcoming proposition proves the heuristic argument from \cite[Supplementary Note~8]{K19} that Lloyd's algorithm cannot alter the value of the asymptotic variance for any finite number of iterations when starting with a Poisson process.
This value may, however, change in the limit of an infinite number of
iterations (as for random organization in Sec.~\ref{s:bddperturbations}).
Such a weak convergence is indeed suggested by simulations (even though the
exact limit value is difficult to assess)~\cite{K19}.
\begin{proposition}
\label{p:lloyd}
Let $\Phi$ be a stationary Poisson process of non-zero intensity $\gamma$. For $k \in \N$, define $\Phi_k$ as the $k$-th iterate of Lloyd's algorithm as in \eqref{e:Lloyditeration} applied to $\Phi$.  Then, for all $k \in \N$, we have that 
$$\lim_{r \to \infty} \lambda_d(B_r)^{-1}\BV[\Phi_k(B_r)] = \lim_{r \to \infty} \lambda_d(B_r)^{-1}\BV[\Phi(B_r)] = \gamma.$$
\end{proposition}
To analyse a single iterate of the Lloyd's algorithm,
  we need to define a stopping set for Voronoi cells of a Poisson
  process and this is classically done using the Voronoi flower
  construction; see proof of Theorem \ref{t:HUrandvol} or \cite{Zu99} . 
However, we will borrow
  a coarser construction from \cite[Section
  5.1]{PenroseLLN2007} and more importantly, we can
  adapt it suitably to construct stopping sets for multiple
  iterates of Lloyd's algorithm.
\begin{proof}
Define $Y: \bN \to \R^d$ as $Y(\varphi) := \Ce(0,\varphi)$.  Given $x \in \varphi$,  define recursively $x^{(k)}, \varphi_k,  k \geq 1$ as follows. $x^{(0)} := x$,  $x^{(k)} :=  x^{(k-1)} + Y(\varphi_{k-1} - x^{(k-1)})$ where $\varphi_k := \{x^{(k)}\}_{x \in \varphi}$ for $k \in \N$.  We set the transport kernel to be $K^{(k)}(x) := \delta_{x^{(k)}}$. Thus by the recursive nature of the definition of \eqref{e:Lloyditeration},  we can verify inductively that
$$ \Phi_k =  \sum_ {X \in \Phi} \delta_{X^{(k)}} = K^{(k)}\Phi.$$  
We only need to provide the construction of an appropriate stopping set for the transport kernel $K^{(k)}(x)$ satisfying \eqref{stopbounds}, as other assumptions in Theorem \ref{tgen_mix_pert_pp} hold trivially. 

\paragraph{\ul{Stopping set construction:}} Let $\varphi \in \bN$ and suppose $0 \in \varphi$. We shall now recursviely construct $S_k := S_k(0,\varphi)$ and as before, we set $S_k(x) := S_k(0,\varphi-x) + x, x \in \varphi$. Trivially, set $S_0 := \{0\}$. We first start with defining $S_1 = S$.

Let $H_i, 1 \leq i \leq m$ be a finite collection of infinite cones with apex at $0$ (but not containing $0$) and angular radius $\pi/12$ such that $\R^d \setminus \{0\} = \cup_{i=1}^m H_i$. Let $R_1 := R_1(0, \varphi)$ denote the maximum distance of $0$ to the nearest points of $\varphi$ in the cones $H_i$:
$$R_1(0,\varphi) := \max_{i=1,\ldots,m} \inf \{ r : \varphi \cap B_r \cap H_i \neq \emptyset \}.$$
Set $S(0) := S(0,\varphi) := B(0,R_1)$. Verifying \eqref{stopping2}, we have that $S(0,\varphi)$ is a stopping set. By geometric considerations and stopping set property of $S$, we shall now show that $C(0,\varphi) \subset S(0,\varphi)$ and $C(0,\varphi)$ remains unaffected by changes outside $S(0,\varphi)$. Indeed we have that $C(0,\varphi) \cap H_i \subset B_r$ if $\varphi \cap B_r \cap H_i \neq \emptyset$ and this gives that $C(0,\varphi) \subset S(0,\varphi)$. Also,
$$ \|x\| \leq \inf_{z \in \varphi} \|z-x\| \, \text{ iff }  \|x\| \leq \min_{y \in \varphi \cap S(0,\varphi)} \|x-y\|, \quad x \in \R^d.$$
Thus $x \in C(0,\varphi)$ iff $\|x\| \leq  \min_{y \in \varphi \cap S(0,\varphi)} \|x-y\|$. Hence $C(0,\varphi) = C(0,\varphi \cap S(0,\varphi))$ and since $Y(\varphi)$ is determined by $C(0,\varphi)$, we have also that $Y(\varphi) = Y(\varphi \cap S(0,\varphi))$. In other words $C(0,\varphi), Y(0,\varphi)$ are determined by $\varphi \cap S(0,\varphi)$.

Note that by definition $S(y,\varphi) := y + S(0,\varphi-y) = B(y,R_1(y))$ is the Voronoi stopping set associated to $y$ for $y \in \varphi$ where $R_1(y) := R_1(y,\varphi) : = R_1(0,\varphi-y)$. A similar convention applies to the forthcoming stopping sets $S_k$'s and radii $R_k$'s. Additionally, we set $H_i(x) := x + H_i$ for $1 \leq i \leq m,x \in \R^d$. Thus, $y^{(1)} \in S(y,\varphi)$. Hence for a compact set $B$, it holds that $y^{(1)} \in B$ if $S(y,\varphi) \subset B$. The latter event is determined by $\varphi \cap B$ as $S(y,\phi)$ is a stopping set and so $y^{(1)} \in B$ is also determined by $\varphi \cap B$. Recall that we denote the iterates of $0$ under Lloyd's algorithm as $0^{(1)},\ldots,0^{(k)},\ldots$ 

In order to prepare for our recursive definition, we shall rewrite definition of $R_1$ differently as follows. Recall that  $S_0(z) = \{z\}, z^{(0)} = z$ for all $z \in \R^d$. Thus $R_1$ can be equivalently defined as
$$ R_1 = R_1(0,\varphi) = \max_{i=1,\ldots,m} \inf_{z \in \varphi} \{ \sup_{y \in S_0(z)} |y - 0^{(0)}| : \, S_0(z) \subset H_i(0^{(0)}) \},$$
and $S_1 = S = B_{R_1}$. We now define $S_2(0) :=  B(0^{(1)},R_2) \cup B_{R_1}$ where 
$$ R_2 := R_2(0,\varphi) := \max_{i=1,\ldots,m} \inf_{z \in \varphi} \{ \sup_{y \in S_1(z)} |y - 0^{(1)}| : \, S_1(z) \subset H_i(0^{(1)}) \}.$$
Observe that $S_1 \subset S_2$. Now, iteratively, we define $S_{k}(0) := S_{k-1}(0) \cup B(0^{(k-1)},R_k)$, with $R_{k}$ defined via
$$ R_{k} := R_{k}(0,\varphi) :=   \max_{i=1,\ldots,m} \inf_{z \in \varphi} \{ \sup_{y \in S_{k-1}(z)} |y- 0^{(k-1)}| : S_{k-1}(z) \subset H_i(0^{(k-1)}) \}.$$
Note that we have suppressed $\varphi$ in $S_{k-1}$ on the RHS but evidently the iterates $0^{(1)},\ldots,0^{(k)}$ depend on $\varphi$ and so do $R_k$ and $S_k$.Note that when $\varphi$ is taken to be a point process $\Phi$, the above elements $S_k(\cdot),R_k(\cdot), y^{(k)}$ are random elements.

By definition $S_k$ is monotonic increasing. We will now show that $S_k$ is a stopping set that determines the $k$-th iterate of Lloyd's algorithm as well as satisfies the probability bounds required for the application of Theorem \ref{tgen_mix_pert_pp}. 
\begin{lemma}
\label{l:Skstoppingset}
For all $k \in \N$, $S_k$ is a stopping set and determines $C(0^{(k-1)},\varphi_{k-1})$ (i.e., the Voronoi cell of the origin in the $k$-th iterate of Lloyd's algorithm) and hence $0^{(k)}$ and $K^{(k)}(0)$ as well.
\end{lemma}
\begin{proof}
By the definition of $R_k$, we find that for all $i \in \{1,\ldots,m\}$ there exists $z_i \in \varphi$ with $S_{k-1}(z_i) \subset H_i(0^{(k-1)}) \cap B(0^{(k-1)},R_k)$. This implies that $z_i^{(k-1)} \in S_{k-1}(z_i) \subset H_i(0^{(k-1)}) \cap B(0^{(k-1)},R_k)$  and furthermore $z_i^{(k-1)} \in \varphi_{k-1}$. So $R_1(0^{(k-1)},\varphi_{k-1}) \leq R_k(0,\varphi)$ which implies that $S(0^{(k-1)},\varphi_{k-1}) \subset B(0^{(k-1)},R_k)$ and hence $C(0^{(k-1)},\varphi_{k-1}) \subset B(0^{(k-1)},R_k)$. Thus $C(0^{(k-1)},\varphi_{k-1})$, $0^{(k)}$ and $K^{(k)}(0)$ are all determined by $S_{k-1}(0) \cup B(0^{(k-1)},R_k) = S_k(0)$. It now remains to show the stopping set property of $S_k$.  We shall again use \eqref{stopping2} to verify the stopping set property.

As discussed after defining $R_1$, the stopping set claim for $l=1$ holds and so we now assume that the claim holds upto $l - 1$ for some $l \geq 2$. Now for such $l$, observe that
$$R_k(0,\varphi)=R_k(0,(\varphi \cap S_k(0,\varphi)) \cup (\varphi'\cap S_k(0,\varphi)^c)), \quad \varphi,\varphi'\in\bN_s,$$
as $0^{(k-1)}$ is determined by $\varphi \cap S_{k-1}(0,\varphi) \subset \varphi \cap S_{k}(0,\varphi)$. So, by definition of $S_k$, we can verify \eqref{stopping2} for $S_k$ and thus the claim of $S_k$ being a stopping set follows.  
\end{proof}

\paragraph{\ul{Stopping set estimates:}} Now we shall again recursively derive the tail estimates necessary for application of Theorem \ref{tgen_mix_pert_pp}. 
\begin{lemma} \label{l:Skstopsetprob} For all $k \in \N$, there exists $b'_k,b_k > 0$ (depending on $m,d$) such that $\BP_0(S_k(0) \subsetneq B_t) \leq b'_ke^{-b_kt^d}$ for all $t > 0$. 
\end{lemma} 
\begin{proof} The proof is by induction. Firstly for $S_1 = S$, we have by $R_1 = R$ and the union bound that
$$ \BP_0(S(0,\Phi) \not \subset B_t) = \BP_0(R   >   t) \leq \sum_{i=1}^m \BP( \Phi \cap H_i \cap B_t = \emptyset) \leq me^{-\frac{\pi_d t^d}{m}}, $$
where the last step is due to Poissonian assumption of $\Phi$ and $\pi_d$ is the volume of the $d$-dimensional unit ball. Now proceeding inductively for $k \geq 2$, we have that 
\begin{align}
\BP_0(S_k(0) \not \subset B_t) & \leq \BP_0(S_{k-1}(0) \not \subset B_{t/4}) + \BP_0(S_{k-1}(0) \subset B_{t/4}, R_k(0^{(k-1)})   >   t/2) \nonumber \\
\label{e:SlboundSk-1} & \leq b'_{k-1}\exp\{-b_{k-1}\frac{t^d}{4^d}\} + \BP_0(S_{k-1}(0) \subset B_{t/4}, R_k(0^{(k-1)})   >   t/2),
\end{align}
where we have used that $0^{(k-1)} \in S_{k-1}(0)$ in the first and the induction hypothesis in the second inequality. The proof is complete if we show that the latter probability can also be bounded by an exponentially decaying term.

Thus it remains to bound  $\BP_0(S_{k-1}(0) \subset B_{t/4}, R_k(0^{(k-1)}) > t/2)$. For the same, define for all $i = 1,\ldots,m$ 
$$ R_{k,i} := \inf_{z \in \Phi} \{ \sup_{y \in S_{k-1}(z)} |y-0^{(k-1)}| : S_{k-1}(z) \subset H_i(0^{(k-1)}) \}.$$
Thanks to union bound and that $R_k = \max_{i=1,\ldots,m}R_{k,i}$, it now suffices to bound  for all $i =1,\ldots,m$
$$ \BP_0(S_{k-1}(0) \subset B_{t/4}, R_{l,i}  >  t/2).$$ 
Without loss of generality, we shall consider only $\BP(S_{k-1}(0) \subset B_{t/4}, R_{k,1}   >  t/2)$ . If $S_{k-1}(0) \subset B_{t/4}$ then we have that $0^{(k-1)} \in B_{t/4}$. Hence, there exists $a > 0$ such that for all $t$ large enough, there exists $Z \in \R^d$ (only dependent on $\Phi \cap B_{t/4}$) such that $B(Z,2at) \subset H_1(0^{(k-1)}) \cap B(0^{(k-1)},t/2) \setminus B_{t/4}$. If there exists an $y \in  \Phi \cap B(Z,at)$ such that $S_{k-1}(y) \subset B(y,at)$ then we have that 
$$ S_{k-1}(y) \subset B(Z,2at) \subset H_1(0^{(k-1)}) \cap B(0^{(k-1)},t/2),$$
and so by definition $R_{k,1} \leq t/2$. By the stopping set property (Lemma \ref{l:Skstoppingset}), the events $S_{k-1}(y) \subset B(y,at)$ for some $y \in B(Z,at)$ depend only on $\Phi \cap B(Z,2at)$. Also note that by choice of $Z$, $\Phi \cap B(Z,2at) = \Phi \cap B_{t/4}^c \cap B(Z,2at)$ where $\Phi \cap B_{t/4}^c$ and $Z$ are independent. This independence will be used crucially in some of the probability derivations below. 
From this observation and that $0 \notin B(Z,2at)$, we can derive that for all $t$ large enough, 
\begin{align*}
 &\BP_0(S_{k-1}(0) \subset B_{t/4}, R_{k,1}(0) > t/2) \\
 & \leq \, \BP_0 \bigl(\{S_{k-1}(y) \subset B(y,at) \textrm{ for some } y \in \Phi \cap B(Z,at)\}^c \bigr) \\
& =  \, \BP \bigl(S_{k-1}(y) \not \subset B(y,at) \textrm{ for all } y \in \Phi \cap B(Z,at) \bigr) \\
& \leq \, \BP \bigl( \Phi \cap B(Z,at) = \emptyset \bigr) + \BP \bigl(S_{k-1}(y) \not \subset B(y,at) \textrm{ for some } y \in \Phi \cap B(Z,at) \bigr) \\
& \leq e^{-\gamma \pi_d (at)^d} + \BE\bigg[ \int \I\{y\in B(Z, at)\}\I\{S_{k-1}(y) \not \subset B(y,at)\} \, \Phi(\md y)\bigg].
\end{align*}
As $Z$ is independent of $\Phi\cap B_{t/4}^c$, it is also independent of $S_{k-1}(y) \not \subset B(y,at)$ for $y\in B_{t/4+at}^c$. Further, as $B(Z, 2at) \subset B_{t/4}^c$, we also have $B(Z, at) \subset B_{t/4+at}^c$. Together with the Mecke formula, this yields
\begin{align*}
&\BE\bigg[ \int \I\{y\in B(Z, at)\}\I\{S_{k-1}(y) \not \subset B(y,at)\} \, \Phi(\md y)\bigg]\\
& = \, \BE\bigg[ \int_{B_{t/4+at}^c} \I\{y\in B(Z, at)\}\I\{S_{k-1}(y) \not \subset B(y,at)\} \, \Phi(\md y)\bigg]\\
& = \, \gamma \int_{B_{t/4+at}^c} \BP(y\in B(Z, at))\BP_y(S_{k-1}(y) \not \subset B(y,at)) \, \md y\\
& = \, \gamma \pi_d (at)^d \BP_0(S_{k-1}(0) \not \subset B_{at})\\
& \leq  \, \gamma \pi_d (at)^d b'_{k-1}e^{-b_{k-1}(at)^d},
\end{align*}
where we have used the induction hypothesis in the last inequality. As explained above, this yields the bound that for all $t$ large enough
$$ \BP_0(S_{k-1}(0) \subset B_{t/4}, R_k(0^{(k-1)}) > t/2) \leq m \big( e^{-\gamma \pi_d (at)^d} + \gamma \pi_d (at)^d b'_{k-1}e^{-b_{k-1}(at)^d} \big),$$ 
and substituting into \eqref{e:SlboundSk-1} completes the proof of the lemma.
\end{proof}

\paragraph{\ul{Completing the proof:}}
Now we return to the proof of Proposition \ref{p:lloyd}. Following the proof method as in Lemma \ref{l:Skstopsetprob}, we can also derive a similar bound for $\sup_{y \in \R^d} \BP_{0,y}(S_k(0,\Phi) \not \subset B_t)$. The only difference is that we need to choose $B(Z,at)$ such it does not contain $y$ as well.
We remark that the choice of (random) $Z$ could depend on $y$ but the constant `$a$' will be independent of $y$ and this suffices to give the necessary bounds for our purposes. Thus, we have verified the required stopping set assumption in Theorem \ref{tgen_mix_pert_pp} and so the proof of Proposition \ref{p:lloyd} is complete.
\end{proof}

\section{Transports of Lebesgue measure} \label{s:volumes}

In this section, we consider transport kernels acting on Lebesgue measure, the simplest example of a hyperuniform random measure. But we shall see that this already yields interesting examples.
For the first general result we work in the setting of
Subsection \ref{subPalm}.

\begin{theorem}\label{t:transportLebesgue}
Let $K$ be an invariant probability kernel from $\Omega\times\R^d$ to $\R^d$,
satisfying
\begin{align}\label{e:0129}        
\int\big\|\BE[\Ks_y\otimes \Ks_0]- \BE[\Ks_0]^{\otimes 2}\big\|\,\md y<\infty.
\end{align}
Then the random measure $\int K(x,\cdot)\,\md x$ is hyperuniform w.r.t.\ any
$W \in \cK_0$.
\end{theorem}
\begin{proof} We apply Theorem \ref{maintheoremmixing2} with $\Phi:=\lambda_d$.
It is easy to see that $\alpha_\Phi=\lambda_d$, $\beta_\Phi=0$ and $\BP^\Phi_0=\BP$.
Further we can  choose $\BP^\Phi_{0,y}=\BP$ for all $y\in\R^d$. The result follows.
\end{proof}

\begin{example}\rm Suppose that $Z=\{Z(x):x\in\R^d\}$ is a
stationary $\R^d$-valued Gaussian random field with c\`adl\`ag-paths, 
as in Example \ref{ex:gdf}. Assume that
\begin{equation}\label{mixinggaussianfield2}
        \int \|\BC(Z(y), Z(0))\|\,\md y < \infty.
\end{equation}
Then it follows from Lemma \ref{l:gauss beta cov bound} and Theorem \ref{t:transportLebesgue}
that  $\int \I\{x+Z(x)\in\cdot\}\,\md x$ is a hyperuniform random measure.
\end{example}

We continue with the Lebesgue counterpart of Theorem \ref{tgen_mix_pert_pp}.

\begin{theorem}\label{tgen_mix_pert_ppLeb}
Let $\Gamma$ be a stationary point process with
non-zero intensity $\gamma$ and having exponentially fast
decay of correlations with decay function $\delta$ and
constants $C_k$, $k \in \N$, with $C_k = O(k^{ak})$ for some $a <1$. 
Let $\tilde K$ be an invariant probability kernel from
$\bN\times \R^d$ to $\R^d$.
Assume that  there is a stopping set $S\colon \bN \to \mathcal{F}$ such that
\begin{align}
\tilde{K}(\mu,0,\cdot) = \tilde{K}(\mu_{S(\mu)},0,\cdot), \quad \mu\in\bN,
\end{align}
and that there exists a decreasing function $\delta_1 \leq 1$ such that
\begin{align}\label{stopbounds2}
\BP(S(\Gamma) \not\subset B_t)\le \delta_1(t),\quad t\ge 0.
\end{align}
Assume finally that
\begin{equation} \label{e:intvarphi2}
\int_1^{\infty}s^{\frac{d}{\beta}-1}\delta_1(s)\,\md  s < \infty,
\end{equation}
where $\beta$ is such that $\beta < \frac{b(1-a)}{(d+2)}, \beta \leq 1$.
Then the random measure $\int \tilde{K}(\Gamma,x,\cdot)\,\md x$ is 
hyperuniform w.r.t.\ any $W \in \cK_0$.
\end{theorem}
\begin{proof} The theorem can be proved as Theorem \ref{tgen_mix_pert_pp}.
In fact, it can be significantly simplified, as there is only one point process
and no Palm probabilities are involved. So one can derive an analogue of Proposition \ref{p:dec_corr} by using FME for a single point process without Palm probabilities as in Theorem \ref{t:FME} instead of Lemma \ref{l:FME2pp}.
\end{proof}

As an application we consider shifts to the  $k$-th nearest neighbour of 
a point process. 

\begin{example}\label{e:weightedVoronoi}\rm
Fix $k\in\N$ and let $\mu\in\bN$ and $x\in\R^d$.
Order the points of the support of $\mu$ by ascending distance from $x$,
using lexicographic order to break ties. 
Let $N_k(x,\mu)$ denote the $k$-th point of the support of $\mu$
w.r.t.\ this order. If the support of $\mu$ has less than $k$ points,
then let $N_k(x,\mu):=x$. Assume that the point process  $\Gamma$ satisfies the assumptions
of Theorem \ref{tgen_mix_pert_ppLeb}. 
Assume further that there exists an exponentially fast decreasing function  $\delta_1$
such that
\begin{align}\label{e:k}
\BP(\Gamma(B_t)\le  k-1) \leq \delta_1(t),\quad t>0.
\end{align}
We will derive from Theorem \ref{tgen_mix_pert_ppLeb} that the random measure 
\begin{align*}
\Psi:=\int \I\{N_k(y,\Gamma)\in\cdot\}\,\md y
\end{align*}
is hyperuniform w.r.t.\ any  $W \in \cK_0$.
Note that
\begin{align}\label{e:Psi}
\Psi=\sum_{x\in\Gamma} \, \lambda_d(C_k(x,\Gamma)) \, \delta_x,
\end{align}
where $C_k(x,\mu):=\{y\in\R^d:N_k(y,\mu)=x\}$. If the support of $\mu$ has at least $k$ points, then $\{C_k(x,\mu):x\in\mu\}$ partitions $\R^d$. 
But note that for $k\ge 3$ it is possible that almost surely a positive fraction of the $C_k(x,\Gamma)$'s will be empty,
so that \eqref{e:Psi} involves some thinning.
For $k=1$ we obtain the Voronoi tessellation mentioned in Subsection \ref{s:nnshift};
see Figure \ref{fig:weighted_voronois} for an illustration of $\Psi$.

To apply Theorem \ref{tgen_mix_pert_ppLeb}  we need to construct a suitable stopping set $S$.
As after Proposition \ref{p:nn2pp} we do this  by setting 
$S(\mu):=B_{|N_k(0,\mu)|}$ if the support of $\mu$ has at least $k$ points.
Otherwise we set $S(\mu):=\R^d$. By \eqref{stopping2}, $S$ is a stopping set.
Since $\Gamma$ is a simple point process we have
$$
\BP(S(\Gamma)\not\subset B_t)=\BP(\Gamma(B_t)\le k-1),
$$
so that Theorem \ref{tgen_mix_pert_ppLeb} indeed applies.
Assumption \eqref{e:k} allows for a similar discussion as made in Subsection \ref{s:ex_pp_mix} on void probabilities which gives examples of point processes satisfying \eqref{e:k} in the case of $k = 1$. Stationary $\alpha$-determinantal processes for $\alpha = -1/m, m \in \N$ as in Example \ref{ex:alphaDPPvoid} satisfy \eqref{e:k};
see \cite[Corollary 1.10]{BYYsupp19}. 
Using the methods from the proof of 
Proposition \ref{p:perm} one can prove \eqref{e:k} for permanental processes if the kernel $K$ has suitable integrability
properties.
We expect the Gibbs processes in Proposition \ref{p:voidGibbs}
to satisfy \eqref{e:k} also in case $k \geq 2$, but cannot offer a proof here.
\end{example}
In the case $k=1$ the random measure \eqref{e:Psi} arises by attaching
to each point of $\Gamma$ the volume of its Voronoi cell. This
is closely related to Example~\ref{ex:5.2}, which assigns the volume
of each cell to a point, but in Example~\ref{ex:5.2}, each point is
uniformly distributed inside its cell. Here, the points coincide with
the Voronoi center, see Fig.~\ref{fig:weighted_voronois}. This case was
studied in the physics literature (e.g., see
\cite{FSFB14,ChiecoDurian21}) using empirical data and heuristic
arguments. On a large scale, the random measure $\Psi$ can be seen as an
approximation of Lebesgue measure. For $W\in\cK_0$ the variance of
$\Psi(W)$ is driven by the cells intersecting the boundary of $W$, so
that the hyperuniformity of $\Psi$ should not come as a surprise.

\section{Hyperuniform random sets}
\label{s:hyprandset}

The central idea of this section comes from \cite{KT19a, KT19b}. There
the authors propose a versatile construction principle for hyperuniform
two-phase media, where dispersions are placed in the cells of a Voronoi
tessellation so that each cell has the same local packing fraction. Here
we prove the hyperuniformity of a closely related variant of this
tessellation-based procedure for the Poisson point process.
Our result also generalizes Example~10\ in \cite{KS18}.

Let $\Gamma$ be a stationary Poisson process 
with intensity $\gamma>0$.
Recall from \eqref{e.vortes} the definition of the Voronoi cell $C(x)\equiv C(x,\Gamma)$ of $x\in\Gamma$. 
Let $W\subset\R^d$ be a measurable set with finite volume $\lambda_d(W)< \infty$ and $0$ as an interior point. Fix $\alpha\in (0,1]$. For $x\in\Gamma$ define
\begin{align*}
\tau(x):=\sup\{r\ge 0: \lambda_d((rW+x)\cap C(x))\le \alpha\lambda_d(C(x))\}
\end{align*}
and
\begin{align*}
D(x)\equiv D(x,\Gamma):=(\tau(x)W+x)\cap C(x).
\end{align*}
By our assumption on $W$ and the convexity of $C(x)$ for $x\in\Gamma$, we have 
\begin{align}\label{e6.55}
\lambda_d(D(x))=\alpha\lambda_d(C(x)),\quad x\in\Gamma.
\end{align}
We consider the random closed set
\begin{align}
Z\equiv Z(\Gamma):=\bigcup_{x\in\Gamma}D(x)
\end{align}
and the associated (random) volume measure $\Psi$, defined by
\begin{align}
\Psi(B):=\lambda_d(Z\cap B)=\int \lambda_d(D(x)\cap B)\,\Gamma(\md x),\quad B\in\mathcal{B}^d.
\end{align}
Since $D(x,\Gamma)=D(0,\theta_x\Gamma)+x$, $x\in\Gamma$, it is easy to show that $Z$ is stationary,
or, more specifically, $Z(\theta_y\Gamma)+y=Z$, $y\in\R^d$.
Hence $\Psi$ is stationary as well.
From the refined Campbell theorem \eqref{erefinedC} 
and \eqref{e6.55} we easily obtain that the intensity
of $\Psi$ (the {\em volume fraction of $Z$}) is given by
\begin{align*}
\BE\Psi([0,1]^d)=\alpha\gamma\,\BE^0_\Gamma\lambda_d(C(0))=\alpha,
\end{align*}
where the second identity is well-known, see e.g.\ \cite[(9.18)]{LastPenrose17}.

\begin{theorem}
\label{t:HUrandvol}
Assume that $\Gamma$ is a Poisson process.
Then the random volume measure $\Psi$ is hyperuniform. 
\end{theorem}

\begin{proof} The proof is divided into two parts. In the first
part, we construct a stationary random field $Y(x), x \in \R^d$,
such that $\Psi = K \Phi$ with $\Phi(\md x) = \alpha \, \md x$
being the scaled Lebesgue measure and the transport kernel given
by $K(y) = \delta_{y+Y(y)}, y \in \R^d$. In the second part, we
construct a suitable stopping set verifying the assumptions of
Theorem \ref{tgen_mix_pert_pp}. Since $\Phi$ is trivially
hyperuniform, the conclusion of the theorem follows.

The first part of the proof is based on a pathwise argument. Let
$\Phi$ be the Lebesgue measure scaled by $\alpha$ i.e.,
$\Phi(\md x) = \alpha \md x$. We assert that there is a stationary random
field $(Y(x))_{x\in\R^d}$ such that
\begin{align}\label{e6.34}
\Psi=\alpha \int\I\{x+Y(x)\in \cdot\}\,\md x = \int\I\{x+Y(x)\in \cdot\}\, \Phi(\md x)
\end{align}
The field is constructed in two steps. First we set
$C'(x):=x+\alpha^{1/d} (C(x)-x)$, $x\in\Gamma$.
Then $C'(x)\subset C(x)$ (by convexity) and $\lambda_d(C'(x))=\alpha\lambda_d(C(x))$.
Then we use the following measure-theoretical fact. If $L,L'\in \mathcal{B}^d$ have
the same finite volume, then there is a measurable mapping $T_{L,L'}\colon L\to L'$
such that 
\begin{align*}
\int_{L} \I\{T_{L,L'}(x)\in\cdot\}\,\md x=\lambda_d(L'\cap \cdot).
\end{align*} 
In the interior of a cell $C(x)$, $x\in\Phi$, the random field $Y(y)-y$  is then defined
as the composition of the mapping $y\mapsto \alpha^{1/d}(y-x)+x$ and
$T_{C'(x),D(x)}$. 

In the second part of the proof we need to check the assumptions of
Theorem \ref{tgen_mix_pert_pp}. Since $C'(x),D(x)$, $x \in \R^d$, are
determined by $C(x,\Gamma)$, the Voronoi cell containing $x$, so is
the random vector $Y(x)$ and hence a stopping set for $C(x,\Gamma)$ is
a stopping set for $Y(x)$ and hence for $K(x)$ too.
We use here the \emph{Voronoi flower} (see for example, \cite{Zu99}) of 
$x \in \Gamma$, defined by
\begin{align*}
S(x,\Gamma) := \bigcup_{y \in C(x,\Gamma)} B(y,\|y-x\|).
\end{align*}
Let $X:= N(0,\Gamma)$ be the nearest neighbour
of $0$ in $\Gamma$ and note that $0\in C(X,\Gamma)$.
Let $S'\colon \bN\to \mathcal{F}^d$ be
(implicitly) defined  by $S'(\Gamma)=S(X,\Gamma)$.
Then $S'$ is a stopping set. Indeed, adding points in the complement
of $S(X,\Gamma)$ does not change the Voronoi cell $C(X,\Gamma)$
and hence also not the nearest neighbour of $0$. Moreover, the restriction of
$\Gamma$ to $S'$ determines $C(X,\Gamma)=C(0,\Gamma)$. 
We have for $t\ge 0$ that
\begin{align*}
\{S'(0) \not \subset B_t \}  & \subset \{ X > t/4 \}\cup  \{ X\le t/4 \}
\cap \bigcup_{x \in \Gamma \cap B(t/4)} \{ S(x,\Gamma) \not \subset B(x,t/4) \}.
\end{align*}
Therefore we obtain from the union bound and the Mecke formula 
\begin{align*}
\BP(S'(\Gamma) \not \subset B_t)
& \leq  \BP(\Gamma(B_{t/4})=0)+ \gamma \int_{B_{t/4}} \BP( S(x,\Gamma+\delta_x) \not \subset B(x,t/4))\, \md x\\
&=e^{-\gamma \pi_d t^d/4^d}+\gamma \,\BP(S(0,\Gamma+\delta_0) \not \subset B_{t/4}) \, \pi_d t^d/4^d,
\end{align*}
where we have used stationarity of $\Gamma$ to obtain the final identity.
It is well-known that there exist $c_1,c_2>0$ such that
\begin{align*} 
\BP( \diam (C(0,\Gamma +\delta_0))>s) \leq c_1e^{-c_2s^d},\quad s>0,
\end{align*}
where $\diam B$ is the {\em diameter} of a set $B\subset\R^d$;
see e.g.\ \cite[Theorem 2]{HugSchn07}. Moreover, it is easy to see that
$S(0,\Gamma+\delta_0)\subset B(0,2\diam (C(0,\Gamma +\delta_0)))$.
Since $\Phi$ is a scaled Lebesgue measure we have $\BP^\Phi_0=\BP^\Phi_{0,y}=\BP$.
Hence, the assumptions of Theorem \ref{tgen_mix_pert_pp} are satisfied
with an exponentially decaying $\delta_1$, $\delta = \I\{s=0\}$ (as $\Phi$ is scaled Lebesgue and $\Gamma$ is Poisson) and hence the integrability
of $\kappa$ follows easily from \eqref{e:intvarphi1}. Thus, $\Psi$ has
same asymptotic variance as $\Phi$ and hence is hyperuniform.
\end{proof}

\begin{appendices}

\section{Appendix: Palm calculus and Factorial moment expansions}\label{AppendixPalm}
In this appendix, we recall aspects of the Palm calculus framework necessary for some of our results. Starting with Palm probability measures, we present two-point and higher order Palm probabilities in the first three subsections - Sections \ref{subPalm}, \ref{subtwopoint} and \ref{subhigherPalm}. Then we introduce higher-order correlations and a self-contained derivation of factorial moment expansion in Sections \ref{subhigherpoint} and \ref{subFME} respectively. These are crucial for our stopping set based transport maps in Section \ref{s:mix_marked_pp} and the ensuing applications in Sections \ref{s:localalgo} and \ref{s:volumes}.
\subsection{Palm probability measures}\label{subPalm}

If $\Phi$ is a simple point process, then the Palm probability
measure $\BP^\Phi_0$ is the conditional probability measure under
the condition that $0\in\Phi$. Here it is important that $\BP^\Phi_0$
describes the statistical behaviour of the whole stochastic
experiment and not just the conditional distribution
of $\Phi$.  This can be conveniently treated
within the setting from \cite{Neveu} and \cite{LaTho09}.

Assume that $\R^d$ acts measurably on $(\Omega,\mathcal{F})$.
This means that there is a  family
of measurable mappings $\theta_s\colon\Omega\to\Omega$, $s\in \R^d$,
such that $(\omega,s)\mapsto \theta_s\omega$ is measurable,
$\theta_0$ is the identity on $\Omega$ and
\begin{align}\label{flow}
\theta_x \circ \theta_y =\theta_{x+y},\quad x,y\in \R^d,
\end{align}
where $\circ$ denotes composition. The family $\{\theta_x:x\in\R^d\}$ 
is said to be (measurable) {\em flow} on $\Omega$.
We assume that the probability measure $\BP$ is
{\em stationary} (under the flow), i.e.
\begin{align}\label{Pstat}
\BP\circ\theta_x=\BP,\quad x\in \R^d,
\end{align}
where $\theta_x$ is interpreted as a mapping from $\mathcal{F}$ to $\mathcal{F}$
in the usual way:
$$
\theta_xA:=\{\theta_x\omega:\omega\in A\},\quad A\in\mathcal{F},\, x\in \R^d.
$$
A random measure on $\R^d$ 
is said to be {\em invariant} (w.r.t.\ to the flow) or {\em flow-adapted}  if
\begin{align}\label{adapt}
\Phi(\omega,B+x)=\Phi(\theta_x\omega,B),\quad \omega\in \Omega,\,x\in \R^d, B\in\cB^d.
\end{align}
In this case $\Phi$ is stationary.

Let $\Phi$ be an invariant random measure with positive and
finite intensity $\gamma$.  Let $B\in\cB^d$ have positive and finite
Lebesgue measure. The probability  measure
\begin{align} \label{Palm}
\BP^\Phi_0(A):=\gamma^{-1}\lambda_d(B)^{-1}\iint \I_A(\theta_x\omega)\I_B(x)\,
\Phi(\omega,\md x)\,\BP(\md \omega), \quad A\in\mathcal{A},
\end{align}
is called the {\em Palm probability measure} of $\Phi$.
It follows from stationarity that this definition is indeed independent of
the choice of $B$. Therefore we obtain the
{\em refined Campbell theorem}
\begin{align}\label{erefinedC}
\iint f(x,\theta_x\omega)\,\Phi(\omega,\md x)\, \BP(\md \omega)=
\gamma \iint f(x,\omega)\,\md x\,\BP^\Phi_0(\md \omega)
\end{align}
for all measurable $f\colon\R^d\times\Omega \to [0,\infty]$.
This generalizes \eqref{eCampbell}.
We write this as
\begin{align} \label{erefC}
\BE\int f(x,\theta_x)\,\Phi(\md x)=
\gamma\,\BE^{\Phi}_0\int f(x,\theta_0)\,\md x,
\end{align}
where $\BE^{\Phi}_0$ denotes expectation with respect to $\BP^\Phi_0$. 
In particular the reduced second moment measure $\alpha_\Phi$ of $\Phi$ (see \eqref{ersecm})
is given by 
\begin{align}\label{alphaPalm}
\alpha_\Phi=\gamma\, \BE^\Phi_0\Phi.
\end{align}
If $\Phi$ is a point process, then $\BP^\Phi_0$ is concentrated on the event
$\{\omega\in\Omega:\Phi(\omega,\{0\})\ge 1\}$.

\subsection{Two-point Palm probabilities}\label{subtwopoint}

In this subsection we assume that $(\Omega,\mathcal{A})$ is a {\em Borel space}
(see \cite{LastPenrose17}) equipped with a flow $\{\theta_x:x\in\R^d\}$.
We consider an  invariant random measure $\Phi$ which is locally square-integrable. 
We assert that there is a probability kernel $(y,A)\mapsto \BP^\Phi_{0,y}(A)$
from $\R^d$ to $\Omega$ such that
\begin{align}\label{2pointPalm}
\BE \int f(x,y,\theta_0)\,\Phi^2(\md (x,y))
= \iiint f(x,x+y,\theta_{-x}\omega)\,\BP^\Phi_{0,y}(\md \omega)\,\alpha_\Phi(\md y)\,\md x
\end{align}
for all measurable $f\colon\R^d\times\R^d\times\Omega \to [0,\infty]$.
Note that this generalizes \eqref{e2.55}.
If $\Phi$ is a simple point process, then $\BP^\Phi_{0,y}$ can
be interpreted as the conditional probability measure $\BP(\cdot\mid 0,y\in\Phi)$.

To prove \eqref{2pointPalm}, we take a measurable $A\subset \R^d\times\Omega$ and
consider the measure $\nu_A$ on $\R^d$, defined by
\begin{align*}
\nu_A(B):=\BE\int \I\{x\in B,(y-x,\theta_x)\in A\}\,\Phi^2(\md (x,y)),\quad B\in\cB^d.
\end{align*}
Since $\Phi$ is locally square-integrable, the measure $\nu_A$ is locally finite.
Moreover, it easily follows from the stationarity of $\BP$ and 
the invariance \eqref{adapt} that $\nu_A$ is invariant under translations.
Therefore we have that $\nu_A(B)=\mu(A)\lambda_d(B)$, where 
\begin{align*}
\mu(A):=\nu_A([0,1]^d)=\BE\int \I\{x\in [0,1]^d,(y-x,\theta_x)\in A\}\,\,\Phi^2(\md (x,y)).
\end{align*}
Clearly $\mu(\cdot)$ is a measure and basic principles of integration theory imply that
\begin{align*}
\BE\int f(x,y-x,\theta_x)\,\,\Phi^2(\md (x,y))
=\iint f(x,y,\omega)\,\mu(\md (y,\omega))\,\md x
\end{align*}
for each measurable $f\colon\R^d\times\R^d\times\Omega \to [0,\infty]$.
By definition we have $\mu(\cdot\times\Omega)=\alpha_\Phi$. Since we have assumed
$(\Omega,\mathcal{F})$ to be Borel, we can disintegrate $\mu$ in the
form $\mu(\md (y,\omega))=\BP^\Phi_{0,y}(\md \omega)\alpha_\Phi(\md y)$ for
a probability kernel $\BP^\Phi_{0,\cdot}(\cdot)$; see e.g.\ \cite[Theorem A.14]{LastPenrose17}. 
Therefore
\begin{align}\label{2pointPalm2}
\BE\int f(x,y-x,\theta_x)\,\,\Phi^2(\md (x,y))
=\iiint f(x,y,\omega)\,\BP^\Phi_{0,y}(\md \omega)\, \alpha_\Phi(\md y)\,\md x.
\end{align}
A simple transformation yields \eqref{2pointPalm}. 

\subsection{Palm probabilities of higher order}
\label{subhigherPalm}

Again we assume here that $(\Omega,\mathcal{A})$ is a {\em Borel space}.
Let $\Phi$ be a stationary random measure on $\R^d$.
Given $n\in\N$ with $n\ge 2$ we define 
the {\em $n$th reduced moment measure} $\alpha_n$ of $\Phi$.
by
\begin{align}\label{ereducedn}
\alpha_n:=\BE \int \I\{x\in[0,1]^d,(y_1-x,\ldots,y_n-x)\in \cdot\}\,\Phi^n(\md (x,y_1,\ldots,y_{n-1})).
\end{align}
This is a measure  on  $(\R^d)^{n-1}$. Assume now that $\BE\Phi(B)^n<\infty$
for each bounded set $B\in\mathcal{B}^d$. Then the measure $\alpha_n$ is
locally finite. Assume that $(\Omega,\mathcal{A})$ is a Borel space
equipped with a flow and that $\Phi$ is invariant.
Then there is a probability kernel $(y_1,\ldots,y_{n-1},A)\mapsto \BP^\Phi_{0,y_1,\ldots,y_{n-1}}(A)$
from $(\R^d)^{n-1}$ to $\Omega$ such that
\begin{align}\label{npointPalm}\notag
\BE &\int f(x,y_1,\ldots,y_{n-1},\cdot)\,\Phi^n(\md (x,y_1,\ldots.y_{n-1}))\\
&= \iiint f(x,x+y_1,\ldots,x+y_{n-1},\theta_{-x}\omega)\,\BP^\Phi_{0,y_1,\ldots,y_{n-1}}(\md \omega)\,\alpha_n(\md (y_1,\ldots,y_{n-1}))\,\md x
\end{align}
for all measurable $f\colon(\R^d)^n\times\Omega \to [0,\infty]$.
This can be proved similarly to \eqref{2pointPalm}.
Note that
\begin{align*}
\BE \Phi^n
= \iint \I\{(x,x+y_1,\ldots,x+y_{n-1})\in\cdot\}\,\alpha_n(\md (y_1,\ldots,y_{n-1}))\,\md x.
\end{align*}
Therefore we can rewrite \eqref{npointPalm} as
\begin{align}\label{npointPalm2}\notag
\BE \int f(x_1,\ldots,x_n,\cdot)&\,\Phi^n(\md(x_1,\ldots,x_n))\\
&= \iint f(x_1,\ldots,x_n,\omega)\,\BP^\Phi_{x_1,\ldots,x_n}(\md\omega)\,\BE \Phi^n(\md (x_1,\ldots,x_n)),
\end{align}
where
\begin{align}\label{defPalm}
\BP^\Phi_{x_1,\ldots,x_n}:=\BP^\Phi_{0,x_2-x_1,\ldots,x_n-x_1}(\theta_{-x_1}\in\cdot),\quad x_1,\ldots,x_n\in\R^d,
\end{align}
are the $n$-fold Palm probability measures of $\Phi$. By our definition \eqref{defPalm}
we have the invariance property
\begin{align}\label{multPalm}
\BP^\Phi_{x_1+x,\ldots,x_n+x}(\theta_x\in\cdot)=\BP^\Phi_{x_1,\ldots,x_n},\quad x_1,\ldots,x_n,x\in\R^d,
\end{align}
If $\Phi$ is not stationary we can still define Palm probability measures via \eqref{npointPalm},
provided that the measure $\BE\Phi^n$ is $\sigma$-finite.

\subsection{Higher order correlations of point processes}\label{subhigherpoint}

Let $\Phi$ be a point process on $\R^d$ represented as in \eqref{e1.5}.
For $n \geq 1$, define the 
$n$-th factorial product of $\Phi$ as the point process on $(\R^d)^n$ defined
by
\begin{align*}
\Phi^{(n)}:=\sideset{}{^{\ne}}\sum_{m_1,\ldots,m_n}\I\{(X_{m_1},\ldots,X_{m_n})\in\cdot\},
\end{align*}
where $\sum^{\ne}$ denotes that no two indices are equal. 
The intensity measure $\alpha^{(n)}:=\BE \Phi^{(n)}$ is known as
the {\em $n$-th factorial moment measure} of $\Phi$.
It is $\sigma$-finite if and only if the same holds for  
$n$-th moment measure $\BE\Phi^n$. If this is the case, there exists
a probability kermel $(x_1.\ldots,x_n)\mapsto\BP^!_{x_1,\ldots,x_n}$
from $(\R^d)^n$ to $\bN$ satisfying
\begin{align}\label{npointreducedPalm}\notag
\BE \int f(x_1,\ldots,x_n,&\Phi-\delta_{x_1}-\cdots-\delta_{x_n})\,\Phi^{(n)}(\md(x_1,\ldots,x_n))\\
&= \iint f(x_1,\ldots,x_n,\mu)\,\BP^!_{x_1,\ldots,x_n}(\md\mu)\,\alpha^{(n)}(\md(x_1,\ldots,x_n))
\end{align}
for each measurable $f\colon(\R^d)^n\times\bN\to[0,\infty]$. 
The probability  measures $\BP^!_{x_1,\ldots,x_n}$ are the
{\em ($n$-th order) reduced Palm distributions} of $\Phi$.
As opposed to $\BP^{\Phi}_{x_1,\ldots,x_n}$, these are probability measures
on $\bN$.

We call $\rho^{(n)} \colon (\R^d)^{(n)} \to [0,\infty)$
the {\em $n$-th correlation function} of $\Phi$ if it satisfies
\begin{equation}\label{e:lcorrfn}
\BE \int f(x_1,\ldots,x_n)\,\Phi^{(n)}(\md(x_1,\ldots,x_n))
= \int f(x_1,\ldots,x_n)\rho^{(n)}(x_1,\ldots,x_n) \, \md(x_1,\ldots,x_n),
\end{equation}
for each measurable $f\colon (\R^d)^n \to [0,\infty]$. 
This function exists if  $\alpha^{(n)}=\BE\Phi^{(n)}$ is $\sigma$-finite
and absolutely continuous w.r.t.\ Lebesgue measure. 
Synonymously we say  that $\rho^{(n)}$ is the  $n$-th correlation function of the distribution
$\BP(\Phi\in\cdot)$ of $\Phi$.

Assume that $\Phi$ is stationary and that  $\BE\Phi(B)^n<\infty$
for each bounded set $B\in\mathcal{B}^d$. 
Define the measure $\alpha^!_n$ by 
\eqref{ereducedn} with $\Phi^n$ replaced by  $\Phi^{(n)}$. 
Then the correlation functions
exist iff $\alpha_n^!$ is absolutely continuous. If $\rho_n$ denotes a density, then we may
choose
\begin{align}\label{corrinvariance}
\rho^{(n)}(x_1,\ldots,x_n)=\rho_n(x_2-x_1,\ldots,x_n-x_1),\quad x_1,\ldots,x_n\in\R^d,
\end{align}
to obtain a translation invariant version of $\rho^{(n)}$.
It is also possible to obtain a translation invariant version of the
reduced Palm distributions of order $2$. Modifying the proof of  
\eqref{2pointPalm} in an obvious way, we obtain
a probability kernel $y\mapsto \BP^!_{0,y}$ from $\R^d$ to $\bN$
satisfying
\begin{align}\label{2pointPalmreduced}
\BE \int f(x,y,\Phi -\delta_x-\delta_y)\,\Phi^{(2)}(\md (x,y))
= \iiint f(x,x+y,\theta_{-x}\mu)\,\BP^!_{0,y}(\md \mu)\,\alpha^!_2(\md y)\,\md x
\end{align}
for all measurable $f\colon\R^d\times\R^d\times\bN \to [0,\infty]$.
Therefore we can and will assume that
\begin{align}\label{defPalmreduced2}
\BP^!_{x,y}=\BP^!_{0,y-x}(\theta_{-x}\in\cdot),\quad x,y\in\R^d.
\end{align}

Let $n,l\in\N$ and assume that $\alpha^{(n+l)}$ is $\sigma$-finite.
If the correlation function $\rho^{(n+l)}$
of $\Phi$ exists (under $\BP$), then it can be easily shown that the correlation functions
$\rho^{(l)}_{x_1,\ldots,x_n}$ of $\BP^!_{x_1,\ldots,x_n}$ exist
for $\alpha^{(n)}$-a.e.\  $(x_1,\ldots,x_n)$. Moreover,
\begin{align*}
\I\{\rho^{(n)}(x_1,\ldots,x_n)=0\}\rho^{(n+l)}(x_1,\ldots,x_{n+l})=0,\quad \alpha^{(n+l)}\text{-a.e.\ $(x_1,\ldots,x_{n+l})$}
\end{align*}
and 
\begin{align}\label{corrPalm}
\rho^{(l)}_{x_1,\ldots,x_n}(x_{n+1},\ldots,x_{n+l})=\frac{\rho^{(n+l)}(x_1,\ldots,x_{n+l})}{\rho^{(n)}(x_1,\ldots,x_n)},
\quad \alpha^{(n+l)}\text{-a.e.\ $(x_1,\ldots,x_{n+l})$}.
\end{align}
All these facts can be derived from \cite[Theorem 1]{Hanisch1982}; see also \cite[Proposition 2.5]{Blaszczyszyn95}.

\subsection{Factorial moment expansion}\label{subFME}
In this section, we formulate factorial moment expansion for functions of a point process and use the same to also formulate one for functions of two independent point processes. Though the former was originally proven by \cite{Blaszczyszyn95,Bartek97}, we give here a self-contained derivation under different assumptions that suffice for our purposes.
Let $\Phi$ be a point process on a Borel space $(\BX,\mathcal{X})$;
see \cite{LastPenrose17}. We assume that $\Phi$ is uniformly $\sigma$-finite,
that is, there exists an increasing sequence $B_k\in\mathcal{X}$, $k\in\N$, with union $\BX$ such that
$\BP(\Phi(B_k)<\infty)=1$ for all $k\in\N$. Then the factorial moment measures
$\alpha^{(n)}$ of $\Phi$ are well-defined for each $n\in\N$.  
Given $n\in\N$  and $x_1,\ldots,x_n\in\BX$ we define the difference 
operators $D^n_{x_1,\ldots,x_n}$ (acting on functions $h\colon\bN\to\R$) as in \cite[Chapter 18]{LastPenrose17}. For $\mu \in \bN$, the first difference operator is defined as
$$ D_{x}h(\mu) = D^1_{x}h(\mu) := h(\mu + \delta_x) - h(\mu), \, \, x \in \BX,$$
and higher-order difference operators are defined recursively
$$ D^n_{x_1,\ldots,x_n}h(\mu)  := D^1_{x_1} \big( D^{n-1}_{x_2,\ldots,x_{n-1}}h(\mu) \big) = \sum_{J \subset [n]} (-1)^{n-|J|}h(\mu + \sum_{j \in J}\delta_{x_j}),$$
with $[n] = \{1,\ldots,n\}$ and $|J|$ denoting cardinality of the set $J$. Trivially, we set $D^0h \equiv h$. We will use $o$ to denote the null measure.

\begin{theorem}\label{t:FME} Let $h\colon\bN\to\R$ be a measurable function. Assume that
\begin{align}\label{FMEass}
\sum_{n=1}^{\infty}\frac{1}{n!}\int \big|D^{n}_{x_1,\ldots,x_n}h(o)\big|\,\alpha^{(n)}(\md(x_1,\ldots,x_n))<\infty.
\end{align}
Assume also that
\begin{align}\label{continfty}
\lim_{k\to\infty}h(\Phi_{B_k})=h(\Phi),\quad \BP\text{-a.s.}
\end{align}
Then $\BE|h(\Phi)|<\infty$ and
\begin{align}\label{FMEgeneral}
\BE h(\Phi) = \sum_{n=0}^{\infty}\frac{1}{n!}\int D^{n}_{x_1,\ldots,x_n}h(o)\,\alpha^{(n)}(\md(x_1,\ldots,x_n)).
\end{align}
\end{theorem}
Before the proof, we compare our assumptions with those of \cite[Theorem 3.1]{Bartek97}.
Let $\prec$ be a measurable total order on $\BX$ restricted to
$\{(x,y)\in\BX^2:x\ne y\}$. Given $\mu\in\bN(\BX)$ and $y\in\BX$
we denote by $\mu_y$ the restriction of $\mu$ to $\{x\in\BX:x\prec y\}$
Assume there exists a sequence $z_k\in\BX$, $k\in\N$,
such that $\{x\in\BX:x\prec z_k\}\uparrow \BX$
and $\BP(\Phi_{z_k}(\BX)<\infty)=1$ for each $k\in\N$. Assume also that
$\Phi$ is simple.
This is essentially the setting from \cite{Bartek97}.
Let $n\in\N$. Using the symmetry properties of the difference operators
we obtain from
\eqref{820} after some calculations
\begin{align*}
h(\Phi_{z_k})&= h(o) + \sum_{m = 1}^n \int \I\{z_k\prec x_m\prec\cdots
\prec x_1\}
D^{m}_{x_1,\ldots,x_m}h(o) \,\Phi^m(\md(x_1, \ldots, x_m))\\
&\quad + \int \I\{z_k\prec x_{n+1}\prec\cdots \prec x_1\}
D^{n+1}_{x_1,\ldots,x_{n+1}}h(\Phi_{x_{n+1}}) \,\Phi^{n+1}(\md(x_1,
\ldots, x_{n+1})).
\end{align*}
Assume now that
\begin{align}\label{B1}
\sum_{m = 1}^n \int \I\{z_k\prec x_m\prec\cdots \prec x_1\}
\big|D^{m}_{x_1,\ldots,x_m}h(o)\big| \,\alpha^{(m)}(\md(x_1, \ldots,
x_m))<\infty
\end{align}
and
\begin{align}\label{B2a}
\iint \I\{z_k\prec x_{n+1}\prec\cdots \prec
x_1\}\big|D^{n+1}_{x_1,\ldots,x_{n+1}}h(\mu_{x_{n+1}})\big|
\,\BP^!_{x_1,\ldots,x_{n+1}}(d\mu)
\,\alpha^{(n+1)}(\md(x_1, \ldots, x_{n+1}))<\infty.
\end{align}
Then it follows from dominated convergence that
\begin{align}\label{FMEprec}
\BE &h(\Phi)=f(0)+\sum_{m = 1}^n \int \I\{x_m\prec\cdots \prec x_1\}
D^{m}_{x_1,\ldots,x_m}h(o) \,\alpha^{(m)}(\md(x_1, \ldots, x_m))\\ \notag
&+\iint \I\{x_{n+1}\prec\cdots \prec x_1\}
D^{n+1}_{x_1,\ldots,x_{n+1}}h(\mu_{x_{n+1}})
\,\BP^!_{x_1,\ldots,x_{n+1}}(d\mu)
\,\alpha^{(n+1)}(\md(x_1, \ldots, x_{n+1})),
\end{align}
which is the main result from \cite{Blaszczyszyn95,Bartek97}. If
\begin{align}\label{B2}
\lim_{n\to\infty}\iint
\big|D^{n+1}_{x_1,\ldots,x_{n+1}}h(\mu_{x_{n+1}})\big|
\,\BP^!_{x_1,\ldots,x_{n+1}}(d\mu)
\,\alpha^{(n+1)}(\md(x_1, \ldots, x_{n+1}))=0,
\end{align}
then we obtain  the infinite series representation \eqref{FMEgeneral}.
Assumptions \eqref{B1} and \eqref{B2} are (slightly) weaker than
\eqref{FMEass}.
On the other hand they require additional assumptions on $\Phi$. Moreover,
\eqref{B2} involves the Palm distributions of $\Phi$ and seems to be hard to
check for an unbounded function $h$.  Since condition \eqref{FMEass}
involves only the factorial moment measures, it seems to be both
mathematically more natural and easier to check in specific examples.
\begin{proof}(Proof of Theorem \ref{t:FME}) Let us abbreviate $\Phi_k:=\Phi_{B_k}$, $k\in\N$. We have
\begin{align}\label{820}
h(\Phi_k)= h(o) + \sum_{m = 1}^\infty \frac{1}{n!} 
\int  D^{n}_{x_1,\ldots,x_n}h(o) \,(\Phi_k)^{(n)}(\md(x_1, \ldots, x_n)) 
\end{align}
provided that $\Phi(B_k)<\infty$. This is implicit in \cite{Poinas19}. Writing $\Phi_k$ as a finite
sum of Dirac measures and using formula \cite[(18.3)]{LastPenrose17}, 
the identity can be checked by a direct computation.
It follows that
\begin{align}\label{821}
|h(\Phi_k)|\le |h(o)| + \sum_{n=1}^\infty \frac{1}{n!} 
\int  \big|D^{n}_{x_1,\ldots,x_n}h(o)\big| \,\Phi^{(n)}(\md(x_1, \ldots, x_n)). 
\end{align}
Here the right-hand side is independent of $k\in\N$ and integrable
by assumption \eqref{FMEass}.
Dominated convergence shows that the  expectation of the right-hand
side of \eqref{820} tends to the right-hand side of the asserted formula
\eqref{FMEgeneral}. By assumption \eqref{continfty} and
\eqref{821} we can use dominated convergence once again to conclude
that $\BE h(\Phi_k)\to \BE h(\Phi)$ as $k\to\infty$ and that $\BE|h(\Phi)|<\infty$.
Hence the result follows from \eqref{820}. 
\end{proof}

In the following we formulate the FME for functions of two independent
point processes on $\BX$. To do so, we need to introduce some notation.
Let $g\colon\bN(\BX)\times\bN(\BX)\to\R$ be a function, $m\in\N$ and
$x_1,\ldots,x_m\in\BX$. Then $D^{m,1}_{x_1,\ldots,x_m}g$ is obtained by applying
the difference operator  $D^m_{x_1,\ldots,x_m}$ to $g(\cdot,\mu_2)$ for each 
(fixed) $\mu_2$. The result is again a function on  $\bN(\BX)\times\bN(\BX)$.
The function $D^{m,2}_{x_1,\ldots,x_m}g$ is defined in a similar way. For $m=0$ we set
$D^{0,1}_{x_1,\ldots,x_m}g=D^{0,2}_{x_1,\ldots,x_m}g:=g$. Given $n\in\N_0$ and 
$y_1,\ldots,y_n\in\BX$ these operators can be iterated as
$D^{m,1}_{x_1,\ldots,x_m}[D^{n,2}_{y_1,\ldots,y_n}g]$.
For $\mu\in\bN(\BX)$ and $c\in\R$ we set $\int c \,\md \mu^{(0)}:=c$.

\begin{theorem} 
\label{t:FME2pp}
Suppose that $\Phi_1,\Phi_2$ are independent uniformly $\sigma$-finite point
processes on $\BX$. Let $g\colon\bN(\BX)\times\bN(\BX)\to\R$ be a measurable function
such that
\begin{align}
\label{FMEass2}
\sum_{m,n=0}^{\infty} \frac{1}{m!n!}
\iint \big| D^{m,1}_{x_1,\ldots,x_m}[D^{n,2}_{y_1,\ldots,y_n}g](o,o)\big|\,\alpha_2^{(n)}(\md(y_1,\ldots,y_n))
\,\alpha_1^{(m)}(\md(x_1,\ldots,x_m))<\infty,
\end{align}
where $\alpha^{(m)}_i$ ($i=1,2$) is the $m$-th factorial moment measure of $\Phi_i$.
Assume also that
\begin{align}
\label{continfty2}
\lim_{k\to\infty}g((\Phi_1)_{B_k},(\Phi_2)_{B_k})=g(\Phi),\quad \BP\text{-a.s.}
\end{align}
Then $\BE|g(\Phi_1,\Phi_2)|<\infty$ and
\begin{align*}
\BE &g(\Phi_1,\Phi_2)\\
&=\sum_{m,n=0}^{\infty} \frac{1}{m!n!}
\iint  D^{m,1}_{x_1,\ldots,x_m}[D^{n,2}_{y_1,\ldots,y_n}g](o,o)\,\alpha_2^{(n)}(\md(y_1,\ldots,y_n))
\,\alpha_1^{(m)}(\md(x_1,\ldots,x_m)).
\end{align*}
\end{theorem}
\begin{proof}  Define a point process $\Phi$ on
$\BX':=\BX\times\{1,2\}$ by setting
$\Phi(\cdot\times\{1\}):=\Phi_1$ and $\Phi(\cdot\times\{2\}):=\Phi_2$.
Define the measurable map $T\colon\bN(\BX')\to \bN(\BX)\times \bN(\BX)$
by $T(\mu):=(\mu(\cdot\times\{1\}),\mu(\cdot\times\{2\}))$, $\mu\in\bN(\BX')$.
We have
\begin{align}\label{e:disjoint}
T(\Phi)=(\Phi_1,\Phi_2).
\end{align}

We wish to apply Theorem \ref{t:FME} with the function $h:=g\circ T$.
Fix $n\in\N$ and let $f\colon(\BX')^n\to[0,\infty)$ be a measurable, symmetric
function. Since $\Phi_1$ and $\Phi_2$ are independent, we easily get that the $n$-th factorial moment measure
$\alpha^{(n)}$ of $\Phi$ satisfies
\begin{align*}
\int f\,\md \alpha^{(n)}=\sum_{m=0}^n \binom{n}{m}&\iint f((x_1,1),\ldots,(x_m,1),(y_1,2),\ldots,(y_{n-m},2))\\
&\qquad \times\alpha^{(m)}_1(\md(x_1,\ldots,x_m))\,\alpha^{(n-m)}_2(\md(y_1,\ldots,y_{n-m})).
\end{align*}
Therefore we obtain from the symmetry properties of the difference operators that
\begin{align*}
&\int D^{n}_{(x_1,i_1),\ldots,(x_n,i_n)}h(o)\,\alpha^{(n)}(\md((x_1,i_1),\ldots,(x_n,i_n)))\\
&=\sum_{m=0}^n \binom{n}{m}\iint D^m_{(x_1,1),\ldots,(x_m,1)}\big[D^{n-m}_{(y_1,2),\ldots,(y_{n-m},2)}h\big](o)\\
&\qquad\qquad\qquad \times\alpha^{(m)}_1(\md(x_1,\ldots,x_m))\alpha^{(n-m)}_2(\md(y_1,\ldots,y_{n-m}))\\
&=\sum_{m=0}^n \binom{n}{m}\iint D^{m,1}_{x_1,\ldots,x_m}\big[D^{n-m,2}_{y_1,\ldots,y_{n-m}}g\big](o,o)
\,\alpha^{(m)}_1(\md(x_1,\ldots,x_m))\,\alpha^{(n-m)}_2(\md(y_1,\ldots,y_{n-m})),
\end{align*}
where we have used the definition $h=g\circ T$ to get the second equality.
The same calculation applies to the integrals of the absolute value of the
difference operator. Therefore the assertions follows from Theorem \ref{t:FME}.
\end{proof}

\section{Total variation bounds for Gaussian vectors}
\label{s:tvgrv}

\begin{lemma}\label{l:gauss beta cov bound}\rm{}
Suppose that $X_1,X_2$ are $\R^d$-valued jointly Gaussian and identically distributed random vectors. Then there exists a constant $c>0$ such that
\begin{equation}\label{e:gauss beta cov bound}
    \| \BP((X_1,X_2)\in \cdot) - \BP(X_1\in \cdot)^{\otimes  2} \| \leq c \| \BC[X_1,X_2] \|,
\end{equation}
where $c$ only depends on the dimension $d$, the covariance matrix $\BC[X_1]$, and the chosen matrix norm.
\end{lemma}
\begin{proof}
First of all, as all matrix-norms are equivalent, we will assume that the matrix norm $\|\cdot\|$ is the spectral norm. Note that it is submultiplicative.
As both sides of \eqref{e:gauss beta cov bound} are invariant under joint deterministic translations of $X_1,X_2$, without loss of generaltity, we can assume that $\BE[X_1]$ = 0. Further, as $\BC[X_1]$ is positive semi-definite, there exists an invertible matrix $L\in\R^{d\times d}$ such that 
\begin{equation*}
    \BC[L X_1] = L\BC[X_1]L^T = \text{diag}(1,\ldots, 1,0,\ldots,0).
\end{equation*}
Because $L$ is invertible, we have
\begin{equation*}
    \| \BP((LX_1,LX_2)\in \cdot) - \BP(LX_1\in \cdot)^{\otimes  2} \| = \| \BP((X_1,X_2)\in \cdot) - \BP(X_1\in \cdot)^{\otimes  2} \|.
\end{equation*}
Additionally,
\begin{equation*}
    \|\BC[LX_1,LX_2]\| = \| L\BC[X_1,X_2] L^T \| \leq \|L\|^2 \|\BC[X_1, X_2]\|.
\end{equation*}
Hence, without loss of generality, we can assume that $\BC[X_1] = \text{diag}(1,\ldots, 1,0,\ldots,0)$. Because zeros on the diagonal only lead to a reduction in the dimension, we can even assume that $\BC[X_1]=I_d$. Finally, we can also assume that $\|\BC[X_1,X_2]\|\leq \tfrac{1}{2}$ as the LHS of \eqref{e:gauss beta cov bound} is bounded by $1$. Let $A:=\BC[X_1,X_2], \Sigma:=\begin{pmatrix}
    I_d & A \\
    A^T & I_d
\end{pmatrix}$.
Using the block-form of $\Sigma$, we can derive that  
\begin{equation*}
    \det(\Sigma) = \det(I_d) \det(I_d- A^T I_d^{-1} A) = \det(I_d - A^T A) \geq (1- \|A\|^2)^d.
\end{equation*}
This bound, Pinsker's inequality, 
and a well known formula for the Kullback–Leibler divergence of two normal distributions (\cite[Chapter II, Section 4.1.10]{Soch2024}) yield 
\begin{align*}
    \| \BP((X_1,X_2)\in \cdot) - \BP(X_1\in \cdot)^{\otimes  2} \| 
    &\leq \sqrt{ 2 D_{KL}(N(0,\Sigma)\| N(0,I_{2d}))}\\
    &= \sqrt{ -\log(\det(\Sigma))} \\
    &\leq \sqrt{ -\log((1 - \|A\|^2)^d))} \\
    &= \sqrt{ -d\log(1 - \|A\|^2)}.
\end{align*}
Finally, the assertion can be proven using $-\log(1-x)\leq \frac{4}{3}x$ for $x\in[0, \frac{1}{4}]$ and $\|A\|\leq \frac{1}{2}$, as
\begin{equation*}
    \| \BP((X_1,X_2)\in \cdot) - \BP(X_1\in \cdot)^{\otimes  2} \| \leq \sqrt{ -d \log(1 - \|A\|^2)} \leq \sqrt{\frac{4d}{3}} \|A\|.
\end{equation*}
\end{proof}

\end{appendices}

\section*{Acknowledgements}
\addcontentsline{toc}{section}{Acknowledgements}
DY's research was supported by ANRF Core Research Grant CRG/2023/002667,
SERB-MATRICS grant MTR/2020/000470 and CPDA from the Indian Statistical
Institute. The work was also facilitated by his visits to KIT and he is
thankful to the institute for hosting him. He is also thankful to
Manjunath Krishnapur for discussions regarding total variation bounds
for Gaussian vectors. 
This work was also supported by the Deutsche Forschungsgemeinschaft
(DFG, German Research Foundation) through the SPP 2265, under grant
numbers KL 3391/2-2, ME 1361/16-1, WI 5527/1-1, and LO 418/25-1, as well
as by the Helmholtz Association and the DLR via the Helmholtz Young
Investigator Group ``DataMat''.

\end{document}